\documentclass[a4paper,11pt,reqno]{amsart}
\usepackage{epsdice}
\usepackage{latexsym}
\usepackage{amssymb,amsmath,amsthm}
\usepackage{tikz,calc}
\usepackage{caption}
\usepackage{graphics}
\usepackage{wrapfig}
\usepackage{tabularx}
\newcolumntype{C}{>{\raggedleft\arraybackslash$}p{1.1em}<{$}}
\newcolumntype{s}{>{\raggedleft\arraybackslash$}p{.2em}<{$}}
\usepackage{placeins}
\usepackage{url}
\usepackage[vcentermath]{youngtab}
  
\usepackage{booktabs}  
\usepackage{color, tikz-cd, needspace, soul, faktor, bbm}
\usepackage{hyperref}
\newtheorem{theorem}{Theorem}[section]
\newtheorem{corollary}[theorem]{Corollary}

\newtheorem{proposition}[theorem]{Proposition}

\newtheorem{lemma}[theorem]{Lemma}

\numberwithin{equation}{section}

\theoremstyle{definition}
\newtheorem{definition}[theorem]{Definition}

\newtheorem{remark}[theorem]{Remark}
\newtheorem{example}[theorem]{Example}

\newcommand{\R}{\mathbb{R}}
\newcommand{\C}{\mathbb{C}}
\newcommand{\F}{\mathbb{F}}

\renewcommand{\P}{\mathbb{P}}

\renewcommand{\epsilon}{\varepsilon}

\DeclareMathOperator{\GL}{GL}

\DeclareMathOperator{\Stab}{Stab}
\DeclareMathOperator{\Sym}{Sym}

\DeclareMathOperator{\End}{End}
\DeclareMathOperator{\Irr}{Irr}

\DeclareMathOperator{\tr}{tr}
\DeclareMathOperator{\sgn}{sgn}

\DeclareMathOperator{\id}{id}

\newcommand{\Ind}{\smash{\big\uparrow}}
\newcommand{\Res}{\smash{\big\downarrow}}
\newcommand{\ind}{\hbox{$\hskip-0.1pt\raisebox{1pt}{$\uparrow$}$}}
\newcommand{\res}{\hbox{$\hskip-0.1pt\downarrow$}}

\renewcommand{\theta}{\vartheta}
\newcommand{\Id}{\mathrm{id}}

\newcommand{\isoblock}[1]{\Mat(#1)} 

\newcommand{\mfrac}[2]{{\textstyle\frac{#1}{#2}}}

\newcounter{thmlistcnt}
\newenvironment{thmlist}%
	{\setcounter{thmlistcnt}{0}%
	\begin{list}{\emph{(\roman{thmlistcnt})}}{%
		\usecounter{thmlistcnt}%
		\setlength{\topsep}{0pt}%
		\setlength{\leftmargin}{0pt}%
		\setlength{\itemsep}{0pt}%
		\setlength{\labelwidth}{25pt}
		\setlength{\itemindent}{30pt}}%
	}%
	{\end{list}}%
	
\newcounter{thmlistcntE}
\newenvironment{thmlistE}%
	{\setcounter{thmlistcntE}{0}%
	\begin{list}{\emph{(\roman{thmlistcntE})}}{%
		\usecounter{thmlistcntE}%
		\setlength{\topsep}{0pt}%
		\setlength{\leftmargin}{36pt}%
		\setlength{\itemsep}{0pt}%
		\setlength{\labelwidth}{36pt}
		\setlength{\itemindent}{0pt}}%
	}%
	{\end{list}}%

\newcounter{thmlistcntEL}
\newenvironment{thmlistEL}%
	{\setcounter{thmlistcntEL}{0}%
	\begin{list}{\emph{(\roman{thmlistcntEL})}}{%
		\usecounter{thmlistcntEL}%
		\setlength{\topsep}{0pt}%
		\setlength{\leftmargin}{30pt}%
		\setlength{\itemsep}{0pt}%
		\setlength{\labelwidth}{30pt}
		\setlength{\itemindent}{0pt}}%
	}%
	{\end{list}}%
	
	\newcounter{thmlistcntES}
\newenvironment{thmlistES}%
	{\setcounter{thmlistcntES}{0}%
	\begin{list}{\emph{(\roman{thmlistcntES})}}{%
		\usecounter{thmlistcntEL}%
		\setlength{\topsep}{0pt}%
		\setlength{\leftmargin}{20pt}%
		\setlength{\itemsep}{0pt}%
		\setlength{\labelwidth}{30pt}
		\setlength{\itemindent}{0pt}}%
	}%
	{\end{list}}%

\newcounter{defnlistcnt}
\newenvironment{defnlist}%
	{\setcounter{defnlistcnt}{0}%
	\begin{list}{(\alph{defnlistcnt})}{%
		\usecounter{defnlistcnt}%
		\setlength{\topsep}{0pt}%
		\setlength{\leftmargin}{0pt}%
		\setlength{\itemsep}{0pt}%
		\setlength{\labelwidth}{30pt}
		\setlength{\itemindent}{35pt}}%
	}%
	{\end{list}}%
	
\newcounter{defnlistcntE}
\newenvironment{defnlistE}%
	{\setcounter{defnlistcntE}{0}%
	\begin{list}{(\alph{defnlistcntE})}{%
		\usecounter{defnlistcntE}%
		\setlength{\topsep}{0pt}%
		\setlength{\leftmargin}{40pt}%
		\setlength{\itemsep}{0pt}%
		\setlength{\labelwidth}{40pt}
		\setlength{\itemindent}{0pt}}%
	}%
	{\end{list}}%
	
\newcounter{remarklistcnt}
\newenvironment{remarklist}%
	{\setcounter{remarklistcnt}{0}%
	\begin{list}{(\alph{remarklistcnt})}{%
		\usecounter{remarklistcnt}%
		\setlength{\topsep}{3pt}%
		\setlength{\leftmargin}{30pt}%
		\setlength{\itemsep}{3pt}%
		\setlength{\labelwidth}{30pt}
		\setlength{\itemindent}{0pt}}%
	}%
	{\end{list}}%

\usepackage{enumitem}
\newenvironment{axioms}[2][] 
 {\enumerate[label=(#2\arabic*#1),leftmargin=40pt, topsep=3pt,
 ref=\textup{(#2\arabic*#1)}]}
 {\endenumerate}

\DeclareMathOperator{\Ann}{Ann}
\DeclareMathOperator{\RId}{RId}

\DeclareMathOperator{\diag}{diag}
\DeclareMathOperator{\Mat}{Mat}
\DeclareMathOperator{\codim}{codim}

\newcommand{\triv}{\mathbbm{1}}
\newcommand{\conj}[2]{{}^{#1}\hskip-0.75pt#2}
\newcommand{\Eb}[1]{E^\bullet\hskip-0.5pt(#1)}

\newcommand{\MC}{\mathrm{MC}}

\newcommand{\subsubsubsection}[1]{\medskip\noindent{\scalebox{1}{\textit{#1.}}}}

\newlength\mylength

\def\vlongmapsto#1{%
\begin{tikzpicture}
\draw (0,0.5mm) -- (0,-0.5mm);
\setlength{\mylength}{\widthof{#1}}
\draw[->] (0,0) -- (1.2\mylength,0) node[above,midway] {#1};
\end{tikzpicture}
}

\newcommand{\LwAsV}{V_{\eta_G,w}} 
\newcommand{\LwAsL}{L} 
\newcommand{\VOrLw}{V_{\eta_G,w}} 
\newcommand{\VOrLwInProof}{L_w} 

\newcommand{\mbaseline}{1.7em}

\newcommand{\mbaselinemed}{2.8em}

\newcommand{\epsface}[1]{\raisebox{-1pt}{\epsdice{#1}}}

\newcommand{\IrrC}{\Irr}

\newcommand{\mbf}[1]{\mathbf{#1}}

\newcommand{\tn}{t}

\newcommand{\MX}{M}

\newcommand{\XY}{W}

\newcommand{\IP}{P}

\makeatletter
\renewcommand{\tocsubsection}[3]{%
  \indentlabel{\@ifnotempty{#2}{\ignorespaces#1 #2\quad}}#3}

\newcommand\@dotsep{4.5}
\def\@tocline#1#2#3#4#5#6#7{\relax
  \ifnum #1>\c@tocdepth 
  \else
    \par \addpenalty\@secpenalty\addvspace{#2}%
    \begingroup \hyphenpenalty\@M
    \@ifempty{#4}{%
      \@tempdima\csname r@tocindent\number#1\endcsname\relax
    }{%
      \@tempdima#4\relax
    }%
    \parindent\z@ \leftskip#3\relax \advance\leftskip\@tempdima\relax
    \rightskip\@pnumwidth plus1em \parfillskip-\@pnumwidth
    #5\leavevmode\hskip-\@tempdima{#6}\nobreak
    \leaders\hbox{$\m@th\mkern \@dotsep mu\hbox{.}\mkern \@dotsep mu$}\hfill
    \nobreak
    \hbox to\@pnumwidth{\@tocpagenum{\ifnum#1=1\fi#7}}\par
    \nobreak
    \endgroup
  \fi}
\AtBeginDocument{%
\expandafter\renewcommand\csname r@tocindent0\endcsname{0pt}
}
\def\l@subsection{\@tocline{2}{0pt}{2.5pc}{5pc}{}}
\makeatother

\begin{document}

\title[Weak lumping on left cosets]{Weak lumping of left-invariant random walks on left cosets of finite groups}


\author{Edward Crane}
\author{\'Alvaro Guti\'errez}
\author{Erin Russell}
\author{Mark Wildon}

\address{University of Bristol, School of Mathematics, Fry Building, Woodland Road, Bristol, BS8 1UG, United Kingdom}

\email{edward.crane@bristol.ac.uk}  \email{a.gutierrezcaceres@bristol.ac.uk}
\email{erin.russell@bristol.ac.uk}
\email{mark.wildon@bristol.ac.uk}

\date{\today}

\begin{abstract}
Let $G$ be a finite group and let $H$ be a subgroup of $G$. The \emph{left-invariant random walk} driven by a probability measure $w$ on~$G$ is the Markov chain in which from any state $x \in G$, the probability of stepping to $xg \in G$ is $w(g)$. The initial state is chosen randomly according to a given distribution. The walk is said to \emph{lump weakly on left cosets} if the induced process on $G/H$ is a time-homogeneous Markov chain. We characterise all the initial distributions and weights $w$ such that the walk is irreducible and lumps weakly on left cosets, and determine all the possible transition matrices of the induced Markov chain. In the case where $H$ is abelian we refine our main results to give a necessary and sufficient condition for weak lumping by an explicit system of linear equations on $w$, organized by the double cosets $HxH$. As an application we consider shuffles of a deck of $n$ cards such that repeated observations of the top card form a Markov chain. Such shuffles include the random-to-top shuffle, and also, when the deck is started in a uniform random order, the top-to-random shuffle. We give a further family of examples in which our full theory of weak lumping is needed to verify that the top card sequence is Markov.
\end{abstract}



\maketitle
\thispagestyle{empty}	

\vspace*{-12pt}
\setcounter{tocdepth}{1}
\tableofcontents

\section{Introduction}\label{sec:introduction}
Let $G$ be a finite group. We define a \emph{weight} to be a function on~$G$ taking non-negative real values, at least one of which is positive.  Thinking of $w$ as a measure on $G$, given any subset $K$ of $G$, we write $w(K)$ for $\sum_{g \in K} w(g)$.
The \emph{left-invariant random walk} on $G$ driven by the weight~$w$ is the
 time-homogeneous $G$-valued Markov chain $X = (X_0, X_1, X_2, \dots)$ with transition probabilities 
\[ \P[X_{t+1} = xg \mid X_t = x] = \frac{w(g)}{w(G)}.\]
The starting value $X_0$ may be either deterministic or random.
Let $H$ be a subgroup of $G$. 
(Throughout, $H$ and $G$ have these meanings.)
Many natural questions concern the induced random process 
$(X_0 H, X_1 H, \ldots)$
taking values in the set $G/H$ of left cosets of $H$ in $G$. For instance, if $G = \Sym_n$ and $H = \Sym_{n-1}$, then the left-invariant random walk 
on $G$ models a sequence of random shuffles of a deck of $n$ cards, and the induced process on $G/H$
models the sequence of cards appearing as the top card in the deck after each shuffle.
This is the setting of the extended example of our main results in \S\ref{subsec:shuffles},
where we give further justification for our focus on left cosets.

In our stochastic setting, it is natural to start the left-invariant random walk $X$ at a \emph{random} starting point $X_0$, distributed according to a chosen
probability distribution $\alpha$ on~$G$. We denote the resulting Markov chain by $\MC(\alpha,w)$. 
By the main theorem of \cite{BW},
the induced process on $G/H$ 
is a time-homogeneous Markov chain \emph{for every initial distribution} $\alpha$ if and only if, for every double coset $HxH$ of $H$ in $G$, $w(gH)$ is constant over all $gH \subseteq HxH$.
(See \S\ref{sec:doubleCosets} for background on double cosets.)
In this case we say that the random walk on $G$ \emph{lumps strongly} on the left cosets of~$H$.

\begin{definition}\label{defn:lumpsWeakly}
Let $w$ be a weight. 
\begin{defnlist}
\item 
We say that the left-invariant random walk driven by $w$ 
\emph{lumps weakly to $G/H$
when started at the distribution $\alpha$}
if the induced process $(X_t H)_{t \ge 0}$ taking values in $G/H$ is a time-homogeneous
Markov chain, when $X_0$ is distributed according to $\alpha$.
\item We say the left-invariant random walk driven by $w$ \emph{lumps weakly to $G/H$}
if it lumps weakly to $G/H$ for \emph{some} initial distribution $\alpha$.
\end{defnlist}
\end{definition}

More briefly, if (a) holds we say that $\MC(\alpha,w)$ \emph{lumps weakly} to $G/H$
and if (b) holds then $w$ \emph{lumps weakly} or is
\emph{weakly lumping}.
We shall see in this article that weak lumping occurs much more generally than strong lumping, and has a much deeper theory.

We concentrate on the case when the weight $w$ is \emph{irreducible}; that is,
the support of $w$ generates~$G$, or equivalently, 
 the random walk $X$ is an irreducible Markov chain, with the uniform distribution as its unique invariant distribution. We justify this choice in \S\ref{subsec:whyIrreducible}:
we plan to study the reducible case further in a sequel to this paper.
 
  The main contributions of this paper are:
\begin{enumerate}
\item A complete algebraic description of the set of 
irreducible weights $w$ such that the left-invariant random walk on $G$ driven by $w$ lumps weakly to $G/H$ when started at the uniform distribution on $G$
(Theorem~\ref{thm:mainGL} and Corollary~\ref{cor: main testWeight}), 
together with an algorithm that determines whether this holds for any given weight $w$;

\smallskip
\item For any given irreducible weight $w$, 
a complete algebraic description of the set of probability distributions~$\alpha$ on $G$  such that 
$\MC(\alpha,w)$ lumps weakly to $G/H$, and a procedure to compute this set
(Theorem~\ref{thm: main testDist}). 
\end{enumerate}

A key idea
is to interpret a weight $w$ on the group $G$ as the element $\sum_{g \in G} 
w(g) g$ of the  complex group algebra $\C[G]$. 
As we show in Lemma~\ref{lemma:stepIsMultiplication},
steps in the Markov
chain then correspond to multiplication
by $w$ in $\C[G]$, and so we are able to bring representation theory to bear on our problem. 
To avoid difficulties caused by working over a non-algebraically closed field, we must work over $\C$ even
though weights are real valued.

When read with
\S\ref{sec: probabilityPreliminaries} 
on the preliminaries we need from Markov chain theory, and \S\ref{sec:algebraicPreliminaries} on the
algebraic preliminaries,
we hope this paper will be found accessible to a broad readership.

\subsection{Main results}\label{subsec:mainResults}
Given a non-empty subset~$K$ of $G$, the element of $\C[G]$ corresponding to the uniform distribution on~$K$
is 
$$\eta_K = \frac{1}{|K|} \sum_{g \in K} g.$$ Recall that  $e \in \C[G]$ is an 
\emph{idempotent} if \hbox{$e^2 = e$}. For example, $\eta_K$ is an idempotent whenever $K$ is a subgroup of $G$.
Let $E(H)$ be the set of idempotents of $\C[H]$
and let $\Eb{H}$ be the subset of idempotents $e$ such that~$\eta_H e = \eta_H$.
Recall that $(X_t)_{t \ge 0}$ denotes the Markov chain $\MC(\alpha,w)$. 
\newcounter{mainGL}
\setcounter{mainGL}{\value{theorem}}
\begin{theorem}[Characterisation of weak lumping in terms of idempotents]\label{thm:mainGL}
    Let $w$ be an irreducible
    weight on $G$ and let $\alpha$ be a distribution on $G$, both thought as elements of $\C[G]$.
    Then $\MC(\alpha, w)$ lumps weakly to $G/H$ if and only if there exists an idempotent $e\in\Eb{H}$ such that
    \begin{thmlistE}
    \item $\alpha\in\C[G]e$,
    \item $ e w (1-e) = 0$,
    \item $ (e - \eta_H) w \eta_H = 0$. 
    \end{thmlistE}
    In this case, for any $t\ge0$, the conditional distribution of $X_t$ given the sequence of cosets \hbox{$X_0H$\hskip-1pt}, ...,\,$X_t H$ always belongs to $\C[G]e$. 
\end{theorem}

Note that $\C[G]e$ is a left ideal of $\C[G]$.
It follows from the final part of Theorem~\ref{thm:mainGL}, 
by averaging over all sequences of cosets $X_0H$, \ldots, \!$X_{t}H$,
 that the distribution of $X_t$ is in $\C[G]e$ for every time $t$. These observations motivate the following definition.

\begin{definition}\label{defn:lumpsStably}
Let $L$ be a left ideal of $\C[G]$ of the form $L = \C[G]e$ for $e \in \Eb{H}$. 
We say that the left-invariant random walk 
driven by an irreducible weight $w$ \emph{lumps weakly to  $G/H$ with stable ideal $L$}
if $ew(1-e) = 0$ and $(e-\eta_H)w\eta_H = 0$.
\end{definition}
There may be more than one $e \in \Eb{H}$ such that $L = \C[G]e$. However, we shall show (see Lemma~\ref{lemma:circ}) that the conditions $ew(1-e) = 0$ and $(e-\eta_H)w\eta_H =0$ in Definition~\ref{defn:lumpsStably} either hold for all choices of $e$ or none, since they are equivalent to $Lw \subseteq L$ and $L(1-\eta_H)w\eta_H = 0$ respectively. We may write `$w$ lumps stably 
for~$L$'
as shorthand for Definition~\ref{defn:lumpsStably}. It follows from Theorem~\ref{thm:mainGL} that if $w$ lumps stably for $L$ then
\smallskip
\begin{defnlistE}
\item $X = \MC(\alpha, w)$ lumps weakly to $G/H$ for all 
initial distributions $\alpha \in L$, \emph{and}
\item
for any initial distribution $\alpha \in L$ and all $t$,
$L$ always contains the conditional distribution of $X_t$ given 
$X_0H, \ldots, X_tH$.
\end{defnlistE}

\smallskip

\noindent
The card shuffling example in \S\ref{subsec:shuffles} demonstrates that stable lumping is interesting even in cases
where the left-invariant random walk lumps strongly. 
Thus for each $e \in \Eb{H}$, Theorem~\ref{thm:mainGL} gives a necessary and sufficient condition 
for the left-invariant random walk to lump stably for $\C[G] e$.
In Corollary~\ref{cor:real} we refine Theorem~\ref{thm:mainGL} to show
that, in the irreducible case, all weakly lumping weights
can be obtained by considering real idempotents in $E^\bullet(H) \cap \R[H]$.

\begin{remark}
 A left ideal $L$ of $\C[G]$ may be expressed as $L = \C[G]e$ for some idempotent $e \in E(H)$  if and only if $L$ decomposes as a direct sum of its projections to the subspaces $b\hskip1pt\C[H]$ of $\C[G]$, i.e.~$\C[G]e = \bigoplus_{b \in G/H} b \hskip1pt \C[H] e$. 
(Here and throughout, the notation $b \in G/H$ means that~$b$ varies over a set of representatives for the left cosets $bH$
of $H$ in $G$.) 
This decomposition shows that, as a left $\C[G]$-ideal, $\C[G]e$ is isomorphic to the \emph{induced ideal} $(\C[H]e) \Ind_H^G$. (See Definitions~\ref{defn:restrictedAndInducedModules}~and~\ref{defn:inducedIdeal} for the definition
of induction and induced ideals.)
This connection between weak lumping
and induced ideals is a recurring theme in this work and critical to the proof of Theorem~\ref{thm:mainGL}.
\end{remark}

  In Definition~\ref{defn:GLideal} below
we define a \emph{Gurvits--Ledoux ideal} for an irreducible weight
$w$ to be an induced left ideal $L$ of $\C[G]$ 
containing $\eta_G$ and
such that $Lw \subseteq L$. We show in Proposition~\ref{prop:lumpsStably}
that the left-invariant random walk driven by $w$ lumps  stably for the Gurvits--Ledoux ideal~$L$ if and only if $L(1-\eta_H) w \subseteq L(1-\eta_H)$.
This leads to a practical computational test for weak lumping, starting
at a distribution.
By Definition~\ref{defn:GLideal},
$L_{\alpha, w}$  is
the intersection of all Gurvits--Ledoux ideals
containing $\alpha$.

\newcounter{GLmodules}
\setcounter{GLmodules}{\value{theorem}}
\begin{samepage}
\begin{corollary}\label{cor:GLmodules}{\ }
Let $w$ be an irreducible weight.
The Markov chain $\MC(\alpha,w)$ lumps weakly to $G/H$ if and only
if $L_{\alpha,w}(1-\eta_H) \subseteq L_{\alpha, w}$.
\end{corollary} 
\end{samepage}

In \S\ref{sec: tests}, we provide 
a practical computational procedure to compute $L_{\alpha,w}$, and in particular to compute
$L_{\eta_G,w}$
for any given weight $w$. It is an important feature of this test that the computation of the left ideal $L_{\alpha,w}$ and the test of whether $L_{\alpha,w}(1-\eta_H) \subseteq L_{\alpha,w}$ may be performed almost entirely within $\C[H]$, making it more efficient than a direct application of the Gurvits--Ledoux criterion (Theorem~\ref{thm: GL characterisation of WL}) when $H$ is much smaller than $G$.

As a corollary, we obtain a practical test for 
weak lumping, in the wider sense of Definition~\ref{defn:lumpsWeakly}(b).
Set $L_w = L_{\eta_G, w}$.

\newcounter{LTest}
\setcounter{LTest}{\value{theorem}}
\begin{corollary}[Weak lumping test for a weight]
\label{cor: main testWeight}\label{cor:LTest}
Let $w$ be an irreducible weight. The following are equivalent:
\begin{thmlist}
\item The left-invariant random walk driven by $w$ lumps weakly to $G/H$; 
\item $\MC(\eta_G, w)$ lumps weakly to $G/H$;
\item $L_w(1-\eta_H) w\eta_H = 0$;
\item The left-invariant random walk driven by $w$ lumps weakly to $G/H$ with stable ideal $L_w$.
\end{thmlist}
\end{corollary}
  
Dual to the minimal ideal $L_w$ in the previous corollary,
we show in \S\ref{subsec:Jw} that for each irreducible weight $w$ there exists a maximal Gurvits--Ledoux ideal $L$ satisfying 
$L(1-\eta_H) w\eta_H = 0$. We denote this ideal $J_w$.
We provide a practical computational procedure to compute $J_w$ and
use it to describe the set of initial probability distributions $\alpha$ for which $\MC(\alpha,w)$ lumps weakly to $G/H$.

\newcounter{Jthm}
\setcounter{Jthm}{\value{theorem}}
\begin{theorem}[Weak lumping test for an initial distribution]
\label{thm: main testDist}
    Let $w\in\C[G]$ be an irreducible weakly lumping weight.
For each distribution $\alpha$ on $G$, the Markov chain
    $\MC(\alpha, w)$ lumps weakly to $G/H$ if and only if $\alpha \in J_w$.  
\end{theorem}

Returning to Theorem~\ref{thm:mainGL}, given $e \in \Eb{H}$, let
\begin{equation}\label{eq:Theta} 
\Theta(e) =  \bigl\{ w \in \C[G]: ew(1-e) = 0, (e-\eta_H)w\eta_H = 0 \bigr\}
\end{equation}
and let
\[
    \Theta = \bigcup_{e\in \Eb{H}}\Theta(e). 
\] 
Theorem~\ref{thm:mainGL} implies that a weight $w$ lumps weakly
in the sense of Definition~\ref{defn:lumpsWeakly}(b) if and only if $w \in \Theta$;
moreover in this case, as noted above, $w$ lumps weakly to $G/H$ with
stable ideal $\C[G]e$, in the sense of Definition~\ref{defn:lumpsStably}.
Equivalently, the set of weakly lumping irreducible weights is $\Gamma \cap \Delta \cap \Theta$
where $\Delta \subseteq \R[G]$ is the simplex of probability distributions and $\Gamma$ is the set of elements of $\C[G]$ whose support is not contained in any proper subgroup of $G$, i.e.
\[\Gamma = \C[G] \setminus \bigcup_{K \lneq G} \C[K]. \]
Remarkably, as we show in Lemma~\ref{lem:ThetaAlgebra}, 
$\Theta(e)$ is a subalgebra of $\C[G]$; that is, $\Theta(e)$ is a vector subspace of $\C[G]$ closed under multiplication.

We have an interesting characterisation of when the union defining~$\Theta$ is irredundant.
To state it, we require Definition~\ref{def: full induction restriction}: if
the restriction to the subgroup $H$ of the permutation character of $G$ acting on the cosets of~$H$ 
contains every irreducible character of $H$, then we say that $H$ has \emph{full induction restriction}. 
For instance, $H$ has full induction restriction whenever there is a  double coset $HxH$ of the maximum possible size $|H|^2$,
or equivalently, whenever the permutation group of $G$ acting on the left cosets of $H$ has a base
(see \cite[\S 4.13]{CameronPermutationGroups}) of size $2$.

\newcounter{irredundant}
\setcounter{irredundant}{\value{theorem}}
\begin{proposition}
\label{prop: irredundant}
The subgroup $H$ of $G$ has full induction restriction if and only if the union defining $\Theta$ is irredundant, in the sense
that no subalgebra $\Theta(e)$ is contained in another.
\end{proposition}

We remark that for each choice of $e \in \Eb{H}$, the conditions in~\eqref{eq:Theta} 
give a finite system of \emph{linear} equations that define $\Theta(e)$. These equations can be re-expressed so that each equation refers only to values
$w(g)$ for $g$ in a fixed double coset $HxH$. 
This is made explicit in Corollary~\ref{cor:abelian} and~\eqref{eq:ThetaSplit} and makes Theorem~\ref{thm:mainGL}
a computationally effective result.
We also highlight
Proposition~\ref{prop:ThetaAlgebraDoubleCosetDimension},
which states that the number of equations in an irredundant system of equations for the condition $ew(1-e) = 0$
on $\C[HxH]$ is
\[ \bigl\langle \chi_{\C[H]e} \Res_{H \cap xHx^{-1}}, (\chi^{x^{-1}}_{\C[H](1-e)}) \Res_{H \cap xHx^{-1}}
\bigr\rangle. \]
Here $\chi_{\C[H]e}$ is the character of the left $\C[H]$-ideal $\C[H]e$, the downwards arrow denotes restriction,
and the inner product is as defined in~\eqref{eq:innerProduct} below taking $G = H \cap xHx^{-1}$;
see~\S\ref{sec:algebraicPreliminaries}
for the remaining notation. This gives a good flavour of how representation theory leads to 
results on our probabilistic questions. 

Let $\star$ denote the algebra anti-involution on $\C[G]$ defined, for $x \in \C[G]$, by
 $x^\star = \sum_{g \in G} \overline{x(g)} g^{-1}$.
There is a beautiful duality between the left-invariant random walk
driven by a weight~$w$ and its time-reversal, which is the left-invariant random
walk driven by the weight $w^\star$.
(Note that since $w$ is real-valued, $w^\star = \sum_{g \in G} w(g) g^{-1}$.)

\newcounter{Revthm}
\setcounter{Revthm}{\value{theorem}}
\begin{theorem}[Time reversal]\label{thm:mainTimeReversal}
Let $e \in \Eb{H}$ and let $w \in \C[G]$ be a weight. 
The left-invariant random walk on $G$ driven by $w$ lumps weakly to $G/H$ with
stable ideal $\C[G]e$  
if and only if the left-invariant random walk on $G$ driven by $w^\star$ lumps 
weakly to $G/H$ with stable ideal 
$\C[G](1-e^\star+\eta_H)$.
\end{theorem}

We have seen that strong lumping is a sufficient condition for weak lumping. There is another commonly used sufficient condition for weak lumping, called \emph{exact lumping} (see Definition~\ref{defn:exactLumping}). In our setting it corresponds to taking $e = \eta_H$ in Theorem~\ref{thm:mainGL}. Applying Theorem~\ref{thm:mainTimeReversal} to 
the characterisation in~\cite{BW} of strong lumping, stated as (i) in the corollary below,
we obtain a simple criterion
for exact lumping. The weight $w$ in the following corollary may be reducible. 
 
\newcounter{Revcor}
\setcounter{Revcor}{\value{theorem}}
\begin{corollary}\label{cor:strongExact}
The left-invariant random walk driven by a weight $w$
\begin{thmlist}
\item lumps strongly to $G/H$ if and only if $w(gH)$ is constant for left cosets $gH$ in~the same double coset;
\item lumps exactly  to $G/H$ if and only if $w(Hg)$ is constant for right cosets $Hg$ in the same double coset.
\end{thmlist}
\end{corollary}

We remark that Propositions~\ref{prop:condindepstrong} and~\ref{prop:condindepexact} give attractive
reinterpretations of strong and exact lumping 
using conditional independence that may be applied to this corollary. 
In Proposition~\ref{prop:interpolating} we show that this corollary describes the two extreme
cases of a family of results on weak lumping to $G/H$
indexed by the subgroups of~$H$: strong lumping
is the case of the trivial subgroup, and exact lumping 
is the case where the subgroup is $H$ itself.

It is natural to ask for the possible 
transition matrices of the lumped process when 
the left-invariant random walk on $G$ lumps weakly to $G/H$. This is addressed by our next main theorem.
See~\S\ref{subsec:orbitalMatrices} for the definition of the orbital matrices $M_{HxH}$.
To orient the more expert reader we remark that
the algebra $\eta_H \C[G] \eta_H$ in (iii) is isomorphic to the Hecke algebra of $H$-bi-invariant
functions on the double coset space $H \backslash G / H$. The weights appearing in conditions (i), (ii) and (iv) below may be reducible.
 
\begin{theorem}\label{thm:mainTransitionMatrices} 
Let $Q$ be a stochastic matrix with rows and columns indexed by $G/H$. The following are equivalent:
\begin{thmlist}
\item there is a weight $w$ on $G$ such that $\MC(\eta_G,w)$ lumps weakly to $G/H$ and the lumped chain 
has transition matrix~$Q$;
\item $Q$ is the transition matrix of the induced random walk on $G/H$ driven by a weight
$w$ in $\eta_H \C[G]\eta_H$.
\item $Q$ satisfies $Q_{(gH, g'H)} = Q_{(kgH, kg'H)}$ for all $g, g', k \in G$;
\item $Q$ is the transition matrix of the induced random walk on $G/H$
driven by a weight satisfying the conditions in Corollary~\ref{cor:strongExact} to lump 
both strongly and exactly on $G/H$.
\end{thmlist}
\end{theorem}

Our final main result is on the case when
$H$ is abelian. In this case the set $\Eb{H}$ is finite, and so we can make Theorem~\ref{thm:mainGL}
very explicit.
Let \smash{$\widehat{H}$} denote the
set of irreducible linear characters of $H$ and let \smash{$\triv_H \in \widehat{H}$} denote the trivial
character. Let \smash{$e_\phi = |H|^{-1} \sum_{h \in H} \phi(h^{-1}) h$}
denote the centrally primitive idempotent in $\C[H]$ corresponding to the irreducible character $\phi \in \widehat{H}$. 
Let $\langle - , -\rangle$ denote the $G$-invariant inner product on $\C[G]$ defined by
\begin{equation}\label{eq:innerProduct} \Bigl\langle \sum_{g\in G} \alpha(g)g , \sum_{k \in G} \beta(k)k \Bigr\rangle = \frac{1}{|G|}
\sum_{g \in G} \overline{\alpha}(g) \beta(g). \end{equation} 
Given a subspace  $V$ of $\C[G]$, 
let $V^\perp = \{ w \in \C[G] : \text{$\langle v, w \rangle = 0$ for 
all $v \in V$} \}$. 

\begin{corollary}\label{cor:abelian}
Let $D$ be a set of double coset representatives for $H\backslash G / H$.
The left-invariant random walk on $G$ driven by an irreducible weight $w$ lumps weakly on the left cosets of $H$ if 
and only if 
there exists a subset $P \subseteq \widehat{H}$ containing  $\triv_H$ such that
for all $x \in D$ we have $w \in \bigcap_{x \in D} W_x^\perp$, where 
\[ W_x=  \bigl\langle e_\beta  x  e_\gamma : 
\: \beta \in P,\: \gamma \in (\widehat{H} \backslash P) \cup \{\triv_H\} ,\: 
(\beta, \gamma) \not= (\triv_H, \triv_H) 
\bigr\rangle. \]
\end{corollary}

We emphasise that the subspace $W_x$ of $\C[G]$ 
appearing above is contained in the double coset $HxH$ and so each perpendicular space
in the intersection supplies $\dim W_x$ 
independent equations that must be satisfied by the values $w(g)$ for $g \in HxH$.
This is illustrated in our final example in \S\ref{subsec: abelian example}.

\subsection{Extended example: shuffles that frustrate card counters}
\label{subsec:shuffles}

This extended example includes \S 1.2.4 
where we
give an infinite family of shuffles in the symmetric groups $\Sym_n$ 
that lump weakly but not lump strongly
or exactly. 

To get started,
consider a deck of four cards, the Queen on top in position~$1$, then the Jack in position $2$,
then the Ace in position $3$ and finally the King in position $4$ at the bottom, as shown
in the margin. 
\marginpar{$\begin{array}{cc}
     \mathrm{Q} & 1 \\ \mathrm{J} & 2 \\ \mathrm{A} & 3 \\ \mathrm{K}  & 4
    \end{array}$}
(We pick
this non-obvious order to distinguish more forcefully between positions and card values.)
Permutations in $\Sym_4$ act on the deck by permuting positions: if card $C$ is in position
$j$ then, after the shuffle $g \in \Sym_4$, card $C$ is in position~$jg$.
(Note that we act on the right, and so permutations compose left-to-right: $j(gh) = (jg)h$.)

To put this into our algebraic setting, 
let $G = \Sym_4$ and let $H = \Sym_{\{2,3,4\}}$ be the subgroup of permutations that
fix the top card.
The left cosets $G / H$ are $H, (1,2)H, (1,3)H, (1,4)H$;
the left coset $(1,k)H$ is precisely those permutations
moving the card in position~$k$ to position~$1$.
The induced process on left cosets therefore models the changing \emph{value} of the top card:
`what a card-counter sees'. A card-counter acquires no useful information by
memorizing the past history of 
values of the top card if and only if this induced process is a Markov chain. (We shall see an explicit
example of this shortly.)
Dually, the right cosets $H \backslash G$ are $H, H(1,2), H(1,3), H(1,4)$;
the right-coset $H(1,k)$ is precisely those permutations
moving the card in position $1$ to position~$k$.
As noted earlier by Pang in \cite[Example 2.9]{Pang} the 
induced process on right-cosets models the changing \emph{position} of the initial
top card, here the Queen: `follow the lady'.

\subsubsection*{Lumping  on right cosets}
Since the left-invariant random walk (defined for general $G$ and $H$) is left-invariant, we have in particular
\[  \P[X_t = g' | X_{t-1} = g]  = \P[X_t = hg' | X_{t-1} = hg] \]
for all $g$, $g' \in G$ and $h \in H$.
Hence $\P[H X_t = Hg' | X_{t-1} = hg]$ is constant as~$h$ varies,
and any left-invariant random walk lumps strongly on right cosets.  
In particular it is always a Markov chain.
In our shuffling setup, both claims are obvious: the position of the Queen after a shuffle
depends only on its position before the shuffle.
Since strong lumping implies weak lumping, it follows that
 the left-invariant random walk lumps weakly on right cosets
for \emph{any} weight $w$ and initial distribution.
Similarly, if $X_{t-1}$ is distributed uniformly on $Hg$, then
the distribution of~$X_t$ conditioned on $HX_t = Hg'$ is uniform on $Hg'$.
Hence the left-invariant random walk lumps exactly on right cosets, in the sense of Definition~\ref{defn:exactLumping},
for \emph{any} weight~$w$. In our shuffling setup, this says (for instance) that 
if we know the Ace, King and Jack are uniformly distributed on the three positions known
not to have the Queen, then the same is true after the deck is shuffled.
These relatively easy observations contrast with the much deeper theory
for left cosets lumping in this article. 

\subsubsection*{Lumping on left cosets}
Since $H$ has two orbits on $\{1,2,3,4\}$, there are two double cosets, $H$ itself,
corresponding to the orbit $1 H = \{1\}$, and $H(1,2)H$ corresponding to the orbit $1 (1,2) H = \{2,3,4\}$.
The larger double coset is shown in Figure~1.

\smallskip
\begin{figure}[h]
\begin{center}
\includegraphics[page=1]{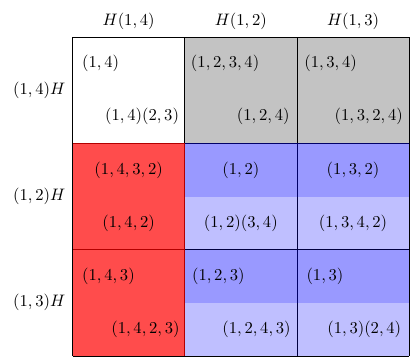}
\end{center}
\caption{The double coset $H(1,2)H$ when $G = \Sym_4$ and $H = \Sym_{\{2,3,4\}}$.
Rows are left cosets and columns are right cosets. 
For later
use in \S1.2.3, the double
cosets $TxT$ where $T = \Sym_{\{2,3\}}$ are coloured. 
Thus \smash{$T(1,3)T = (1,2)T \cup (1,3)T = \bigl\{ (1,3), (1,2,3) \bigl\} \,\cup\, \bigl\{ (1,3,2), (1,2) 
\bigr\}$} is dark blue and $T(1,2,3)T$ is light blue; in black and white any remaining
ambiguity can be resolved by noting that 
since $(1,4) \in N_{\Sym_4}(T)$, we have $T(1,4) = (1,4)T$ and hence \smash{$H(1,4) = T(1,4) \cup T(1,4,3)T$}
and \smash{$(1,4)H = (1,4)T \cup T(1,3,4)T$.}}
\label{fig:DoubleCosets}
\end{figure}
\subsubsubsection{1.2.1.~Strong lumping}\label{1.2.1}
By Corollary~\ref{cor:strongExact}(i), the left-invariant random walk driven by a weight $w$ lumps strongly to $G/H$ if and only
$w(gH)$ is constant for $g \in H(1,2)H$, or equivalently, if and only if $w\bigl( (1,2)H \bigr) = w\bigl( (1,3)H \bigr) = w\bigl( (1,4)H \bigr)$.
(Note the subgroup $H$ is itself a double coset and gives no restriction.)
In Figure~1, the condition is that each row has equal weight.
In particular, this holds whenever the non-identity part of $w$
is uniformly distributed on any fixed right coset of $H$.
Taking one element of the right coset $H(1,2)$ in each left coset of $H$, together with the identity,
we obtain the \emph{random-to-top} shuffle 
$\mfrac{1}{4}\Id + \mfrac{1}{4}(1,2) + \mfrac{1}{4}(1,2,3) + \mfrac{1}{4}(1,2,3,4)$.

\subsubsubsection{1.2.2.~Exact lumping}\label{1.2.2}
By Corollary~\ref{cor:exactLumping},
the left-invariant random walk driven by a weight~$w$ lumps exactly to $G/H$ if,
when started at the uniform distribution on $G$, the distribution conditioned on  the sequence of observations of the value of the top card
is uniform on the relevant left coset. By Corollary~\ref{cor:strongExact}(ii),
this holds if and only if $w(Hg)$ is constant for $g \in H(1,2)H$, or equivalently,
if and only if $w\bigl( H(1,2) \bigr) = w\bigl( H (1,3)\bigr) = w\bigl( H(1,4) \bigr)$.
Dualizing the previous remarks, this holds whenever each column in Figure~1 has equal weight,
and in particular, whenever the non-identity part of $w$ is uniformly distributed on any left coset of~$H$.
Taking one element of the left coset $(1,2)H$ in each right coset of $H$, together with the identity,
 we obtain the \emph{top-to-random} shuffle
$\mfrac{1}{4}\Id + \mfrac{1}{4}(1,2) + \mfrac{1}{4}(1,3,2) + \mfrac{1}{4}(1,4,3,2)$.
See \S\ref{sec: time reversal} for more on the time-reversal
symmetry between the random-to-top and top-to-random shuffles.

\subsubsubsection{1.2.3.~Weak lumping: weights compatible with an idempotent}
Consider the subgroup $T = \Sym_{\{2,3\}}$ of $H = \Sym_{\{2,3,4\}}$. The idempotent
in $\C[\Sym_4]$ corresponding to the uniform distribution on~$T$ is
\smash{$\eta_T = \mfrac{1}{2}\Id + \mfrac{1}{2} (2,3)$}.
By Proposition~\ref{prop:interpolating}, the left-invariant random walk
driven by $w$ lumps stably for $\C[G]\eta_T$,
in the sense of Definition~\ref{defn:lumpsStably},
if and only if
the left-invariant random walk driven by a weight $w$ lumps exactly on the left cosets $G/T$
and \smash{$\mfrac{1}{|TgH|}w(TgH)$} is constant for $TgH \subseteq H(1,2)H$.
In particular, if the non-identity part of $w$ is supported on $H(1,2)H$, then
it is necessary and sufficient that
\begin{itemize}
\item $w(Tg) = w\bigl(Tg(2,3)\bigr)$ for all $g \in H(1,2)H$;
\item $w\bigl( (1,4) H \bigr) = \mfrac{1}{2} w\bigl( (1,2) H \bigr) + \mfrac{1}{2} w\bigl( (1,3)H\bigr)$.
\end{itemize}

In particular, if $0 \le \lambda \le 1$ and 
\begin{equation}\label{eq:w123} w = (1-\lambda) \Id + \mfrac{\lambda}{3} \bigl(
(1,4)(2,3) + (1,4,3) + (1,4,2,3) \bigr) \end{equation}
then both conditions hold. (For the first condition, note that $T(1,4) = \{(1,4)(2,3), (2,3)\} = T(1,4)(2,3)$
so the condition holds for $T(1,4)$, and  since
 $(1,4,3)(2,3) = (1,4,2,3)$, so it holds for $T(1,4,3)$ and $T(1,4,2,3)$, the other relevant cosets.)
But since $w$ is neither constant  on the rows nor on the columns in Figure~1, this weight
does not lump strongly or exactly.

The probabilistic interpretation is as follows: ask a friend to take the deck of cards in its starting configuration,
and, if a fair coin lands heads, swap the middle two cards. 
This gives an initial distribution corresponding to the idempotent
$\eta_T \in \C[G]$.
The deck is then shuffled repeatedly according to $w$, and after each shuffle, the top card is revealed. 
Since $\MC(\eta_T, w)$ lumps weakly, the sequence
of values of the top card
is a Markov chain on the set of cards $\{\mathrm{A}, \mathrm{K}, \mathrm{Q}, \mathrm{J} \}$.
Thus the shuffle frustrates card counters: thanks to the Markov 
property, the card counter gets no benefit from memorizing the past history of the top card. 
Moreover
the distribution, conditioned on the observation that
the card starting in position $k$ is on top, gives equal probability to $g$ and $g(2,3)$ 
for each relevant $g$,
namely those $g$ such that $kg = 1$.
In the special case when $\lambda = \mfrac{3}{4}$ and so $w$ gives equal probability
$\mfrac{1}{4}$ to each shuffle,
 the probabilities~$p_k$ that the card in position $k$ moves to the top are

\smallskip
\begin{center}
\begin{tabular}{ccccc} \bottomrule\\[-8pt] 
$k$ & 1 & 2 & 3 & 4 \\ \midrule
$p_k$ & $\mfrac{1}{4}$ & 0 & $\mfrac{1}{2}$ & $\mfrac{1}{4}$ \\[2pt] \bottomrule
\end{tabular}
\end{center}
\smallskip

\noindent and it follows that, at every time, each card is equally likely to be at the top position. Therefore the sequence of top cards
is not only Markov, but in fact independent and uniformly equidistributed. 
We remark that the entropy of $w$ in this case is $-4 \times \mfrac{1}{4} \log_2 \mfrac{1}{4} = 2$; this is the least
possible entropy of a shuffle that induces a weak lumping with these properties.
The shuffle defined with $\lambda = \mfrac{1}{4}$, \emph{with our chosen initial distribution},
is therefore not only frustrating to card counters,
but also fair and efficient! 

Finally, we  show that while our chosen weight $w$, namely
$(1-\lambda)\Id + \mfrac{\lambda}{3}\bigl(
 (1,4)(2,3) +  (1,4,3) +  (1,4,2,3)\bigr)$, lumps weakly,
 $\MC(\alpha, w)$ may fail to lump weakly to $G/H$ if the initial
distribution is outside $\C[\Sym_4]\eta_T$.
Started deterministically at the identity, the pack reads QJAK top-to-bottom. We have
$$\P[X_3 = \mathrm{J} \mid X_2 = \mathrm{Q}, X_1 = \mathrm{Q}]= 0\,$$
since the event $\{ X_2 = \mathrm{Q}, X_1 = \mathrm{Q} \}$ implies that each of the first two shuffles is the identity, and no shuffle in the support of $w$ moves position $2$ to position~$1$.
On the other hand,
\[ \P[X_3 = \mathrm{J} \mid X_2 = \mathrm{Q}] \neq 0, \]
since the sequence of shuffles
\[ \mathrm{QJAK} \; \vlongmapsto{$\; \scriptstyle (1,4,3)$} \; \mathrm{AJKQ} 
\; \vlongmapsto{$\scriptstyle (1,4)(2,3)$}\;
\mathrm{QKJA} \; \vlongmapsto{$\scriptstyle (1,4,3)$}\; \mathrm{JKAQ} \]
is consistent with the event $\{ X_2 = \mathrm{Q} \}$.

\subsubsubsection{1.2.4.~A variation and infinitely many weakly lumping shuffles}
Take a deck of $n$
cards, and shuffle as follows:
\begin{quote}
Remove the bottom card, insert it under a random card chosen uniformly 
from the remaining deck, then move the top card to the bottom.
\end{quote}
In the case $n=4$, this is the shuffle driven by the weight $w$ above, defined with $\lambda = 0$
so that the identity permutation has zero probability. By Proposition~\ref{prop:interpolatingShuffle},
this shuffle lumps weakly to left cosets of $\Sym_{\{2,\ldots, n\}}$ inside $\Sym_n$, but, as seen in
the special case $n=4$, it does not lump strongly or exactly. 
Proposition~\ref{prop:interpolating} may be used to give many other such examples of infinite families
of weakly lumping left-invariant random walks, parametrised by degree, that do not lump strongly or exactly.

\subsubsubsection{1.2.5.~Weak lumping: idempotents compatible with a weight}
We have seen that the left-invariant random walk driven by the weight $w$ defined in~\eqref{eq:w123} 
weakly lumps
for distributions in $\C[G]\eta_T$. Correspondingly, as expected from Theorem~\ref{thm:mainGL} in the case
$e = \eta_T$, 
we have 
\begin{align*} \eta_T w (1-\eta_T) &= 0, \\ (\eta_T - \eta_H) w \eta_H &= 0. \end{align*}
Consider the weight
$w' = \eta_T w = \bigl(\mfrac{1}{2} \id  + \mfrac{1}{2} (2,3)\bigr) w$
corresponding to first swapping the middle two cards according to a coin-flip, and then shuffling
according to $w$ (defined with general $\lambda$). 
Since $\eta_T$ is an idempotent and $\eta_H \eta_T = \eta_H$,
the two displayed equations above hold replacing $w$ with $w'$.
Therefore $w'$ satisfies the conditions of Theorem~\ref{thm:mainGL} for $e = \eta_T$
and, by this theorem, the weight~$w'$ also weakly lumps 
for distributions in $\C[G] \eta_H$. Using the definition 
$w = (1-\lambda) \mfrac{1}{4} \Id + \lambda \bigl( \mfrac{1}{4} (1,4)(2,3) \bigr) + 
\mfrac{1}{4} (1,4,3) + \mfrac{1}{4} (1,4,2,3) \bigr)$
and the coset diagram in Figure~1, it is easy to see that
\[ w' = \eta_T w = (1-\lambda) \bigl( \mfrac{1}{8} \hskip-1pt \id + \mfrac{1}{8} (2,3) \bigr) +
\ \mfrac{\lambda}{8} \sum_{g \in H(1,2)} g \]
where the first two summands are $\mfrac{1-\lambda}{4}\eta_T$.
It easily follows that $w'$ satisfies the condition in Corollary~\ref{cor:strongExact}(i) for strong lumping.
This should be expected from the previous subsection, 
because~$w'$ includes the randomising effect of the choice of initial distribution
in $\C[G]\eta_T$. This example also shows
that weak lumping and stable lumping are still of interest, \emph{even in the case
when the left-invariant random walk lumps strongly}.

We use this example to illustrate the constructive
methods of \S\ref{sec: tests} in Examples~\ref{ex:LwConstruction} and~\ref{ex:JwConstruction}.

\subsection{Why irreducible weights?}\label{subsec:whyIrreducible}
Figure~2 shows a pentagon whose vertices are labelled by the residue classes modulo $5$ and whose front and back faces are marked F and B, respectively.
Let $\sigma$ be anticlockwise rotation by $2\pi/5$ and let $\tau$ be reflection through the vertical axis. These symmetries preserve
the position of the pentagon in the plane and generate the dihedral group of order $10$ with presentation
\[ G = \langle \sigma, \tau : \sigma^5 = \tau^2 = (\sigma \tau)^2 = 1 \rangle .\]
Let $H = \langle \tau \rangle$. Observe that the left coset $\sigma^i H$ 
consists of the permutations $\sigma^i$ and $\sigma^i\tau$ that move
the vertex labelled $-i$
to the top position. Thus the left-invariant random walk on $G$ 
lumps weakly on $G/H$ if and only if the sequence of observations of the label
at the top position forms a Markov chain. By Corollary~\ref{cor:HOrder2Or3}, if the driving weight
is irreducible, this holds if and only if the chain lumps exactly or strongly.
The weakly lumping irreducible weights are therefore classified by Corollary~\ref{cor:strongExact}.
For example, the uniform weight on $\sigma H = \{\sigma, \sigma \tau \}$ lumps strongly; the uniform weight
on $H \sigma = \{\sigma, \sigma^{-1}\tau \}$ lumps exactly, and the uniform weight on $\{ \sigma \tau, \sigma^{-1} \tau \}$ lumps
strongly \emph{and} exactly; the latter is expected because the support is a conjugacy class of~$G$;
as we explain in the literature survey in \S\ref{subsec:earlierWork} below, this follows from the main theorems of either
\cite{BW} or \cite{DiaconisRamSimper}.

\begin{figure}

\begin{center}
\includegraphics[page=2]{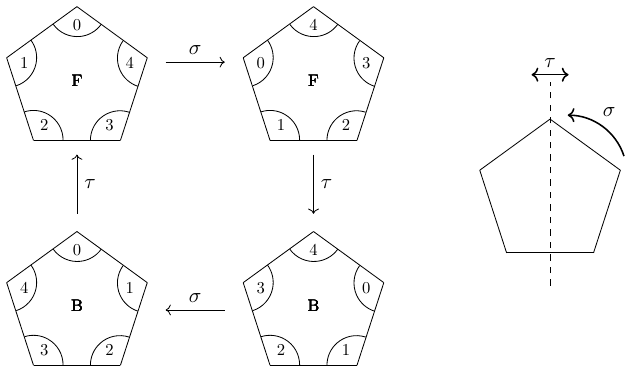}
\end{center}

\caption{The pentagon in its fixed position on the plane showing the anticlockwise rotation
$\sigma$ by $\pi/5$, the reflection $\tau$ and the sequence of symmetries $\sigma, \tau, \sigma, \tau$.
The rotation $\sigma$ acts on the vertex labels as the $5$-cycle $(0,4,3,2,1)$ if the front face is uppermost
and as the $5$-cycle $(0,1,2,3,4)$ if the back face is uppermost.}
\end{figure}

\begin{example}\label{ex: reducible dihedral}
We take the weight $1 + \sigma \in \C[G]$. 
If the initial distribution $\alpha$ is supported on~$\langle \sigma \rangle$, respectively $\langle \sigma \rangle \tau$,
then at every time, $\sigma$ permutes the labels of the pentagon by the $5$-cycle $(0,4,3,2,1)$, respectively $(0,1,2,3,4)$.
(This is shown visually in Figure 2.)
In the first case the induced process on the top label is the Markov chain on $\{0,1,2,3,4\}$ in which the steps $j \mapsto j-1$ mod $5$
and $j \mapsto j$ are equally likely; the same holds in the second case replacing $j-1$ with $j+1$. Therefore $\MC(\alpha, w)$ lumps weakly to $G/H$.
Conversely, suppose that $\alpha \bigl(\langle \sigma \rangle\bigr) > 0$ and $\alpha\bigl(\langle \sigma \rangle \tau\bigr) > 0$.
The sequence of observations of the label at the top position is then of one of the two forms
\begin{align*} & i,  \ldots, i, i + 1, \ldots, i + 1, i +2, \ldots, \\
               & i,  \ldots, i, i - 1, \ldots, i - 1, i - 2, \ldots \end{align*} 
where the first label seen other than $i$ is $i-1$ if  at time $0$ face F was uppermost
and $i+1$ if  at time $0$ face B was uppermost.
Writing $Y_t$ for the label at the top position at time $t$, 
and noting that it is possible that $Y_4 = 0$ and $Y_5 = 1$, we have
$\P[Y_6 = 0 \mid Y_5 = 1, Y_4 = 0] = 0$ whereas $\P[Y_6 = 0 \mid Y_5 = 1] > 0$.
Therefore $\MC(\alpha, 1+ \sigma)$ lumps weakly if and only if $\alpha$ is supported
on a single coset of~$\langle \sigma \rangle$.
\end{example}

This example shows that, in the reducible case, the set of initial distributions $\alpha$ such that $\MC(\alpha,w)$ 
lumps weakly need not be convex. This rules out any routine generalization of our main result, Theorem~\ref{thm:mainGL}, beyond the irreducible case.
See also Example~\ref{periodic example} for another illustration of the more complicated behaviour
seen in the reducible case. 
We believe this case is of considerable interest and plan to study it further in a sequel to this paper.

\subsection{Earlier work}\label{subsec:earlierWork}
For two excellent surveys of 
the  vast literature on  left-invariant random walks on finite groups
we refer the reader to  Diaconis~\cite{DiaconisBook} and Saloff-Coste \cite{SaloffCoste}.
In particular, by \cite[Proposition 2.3]{SaloffCoste} the left-invariant random walk on a finite group $G$ 
driven by a weight $w$ is irreducible if and only if the support of $w$ generates $G$, and is aperiodic if and only if
the support of $w$ is not contained in a coset of a proper normal subgroup of $G$. 
As we saw in~\S\ref{subsec:shuffles} the left-invariant random walk on a symmetric group  models repeated 
shuffles of a deck of cards. Notable examples where representation theory has been used 
to analyse the mixing times of shuffles include the $r$-top-to-random shuffle
studied in \cite{BritnellWildonBell}, \cite{DiaconisFillPitman}, \cite{FulmanCardShufflingTensorProducts} 
and \cite{Phatarfod},
and the riffle shuffle, studied in \cite{AldousDiaconis} and \cite{BayerDiaconis}. 
The recent book \cite{DiaconisFulman} is an excellent introduction to this area.
The left-invariant random walks on general linear groups,
or more broadly, finite groups of Lie type, have also been extensively studied: see 
\cite[Theorems 1.8, 1.9]{LiebeckShalev}
and \cite[\S 9.4]{SaloffCoste}
for further references.

\subsubsection*{Weak lumping in general}
We refer the reader to \S 2 for a comprehensive survey of weak lumping
including a characterisation of weak lumping due to Gurvits and Ledoux \cite{GL} that
is the basis of all our main results. Earlier introductory accounts include 
\cite[\S 6.4]{KS} and \cite[\S 2.4]{Pang}.

\subsubsection*{Weak lumping of the top-to-random shuffle}
Applying the RSK correspondence (see for instance \cite{FultonYT}) 
to a permutation $\sigma \in \Sym_n$ one obtains a pair of standard tableaux $(S, T)$ of the same shape.
For example, 
$\sigma$ is a non-trivial top-to-random shuffle
if and only if $S$ and $T$ have shape $(n-1,1)$ and the unique entry in the second row of the insertion tableau $S$
is $2$; the entry in the second row of the recording tableau~$T$
is then the position of the top card after the shuffle. 
In \cite[Theorem 3.1]{FulmanCardShufflingTensorProducts}, Fulman proves that after $t$ steps of the top-to-random shuffle,
starting at the identity, the distribution of the RSK shape of the permutation agrees with the probability
distribution after~$t$ steps of a random walk on partitions of $n$ defined using
Plancherel measure, started at $(n)$. Since the RSK shape of the identity is~$(n)$,
this might suggest the top-to-random shuffle 
lump strongly to partitions by taking RSK shapes, but
by Proposition 7.2 in \cite{BritnellWildonBell}, this is \emph{not} the case. 
However, in \cite[Theorem 3.9]{Pang}, Pang showed that 
taking RSK shapes in a modified version of the shuffle \emph{does} lump weakly
on a subspace of distributions (as in Corollary~\ref{cor: GL certificate theorem}) that are constant 
on sets of permutations having equal insertion tableau $S$. This leads to a conceptual proof via weak lumping of Fulman's result,
which, as Fulman remarks on page 12 of \cite{FulmanCardShufflingTensorProducts} initially seemed `quite mysterious'.

\subsubsection*{Strong lumping to right cosets}
As we saw in \S\ref{subsec:shuffles}, the left-invariant random walk lumps strongly on right cosets for any subgroup
$H$ of $G$. In \cite[Appendix~1]{AssafDiaconisSoundararajan} (where left- and right- are swapped relative to this paper)
the authors prove this fact and  describes the transition
matrix of the lumped walk. This fact is used in \cite[Example 2.8]{Pang} to show that the
Markov chain on $\F_2^n$ driven by flipping a position chosen uniformly at random lumps strongly
by deleting any single bit and in \cite[Example 2.9]{Pang} to give a shuffling example in which the lumped
chain tracks the position of a chosen card.
In \cite[Example 2.11]{Pang} the case $G = \Sym_n$ and $H = \Sym_r \times \Sym_{n-r}$ is used
to show that the $r$-top-to-random shuffle lumps strongly
by recording the \emph{positions} of the $r$ cards beginning at the top of the deck.

\subsubsection*{Strong lumping to left and double cosets}
As we explain in \S\ref{sec:doubleCosets}, 
the \emph{double cosets} of a finite group~$G$ for subgroups~$H$ and $K$ are the subsets $HxK$ for $x \in G$. 
In \cite{DiaconisRamSimper},
Diaconis, Ram and Simper prove
that if the weight $w$ is conjugacy-invariant, that is $w(g) = w(t^{-1}gt)$ for all $g$, $t\in G$,
then the left-invariant random walk lumps strongly to the double
cosets $H \backslash G / K$ for any subgroups $H$ and $K$. 
Theorem~1.2 in \cite{DiaconisRamSimper} characterises
the stationary distribution of the lumped chain and gives a good bound on the mixing time.
Later, in \cite{BW}, Britnell and Wildon 
gave a necessary and sufficient condition
for the left-invariant random walk to lump strongly to the double cosets $H \backslash G / K$. 
Their Corollary~1.2 implies that
a left-invariant random walk lumps strongly to $H \backslash G / H$  if and only if it lumps
strongly to the left cosets $G/H$; by Corollary~\ref{cor:strongExact}(i) this holds if and only if $w(gH)$
is constant for left cosets $gH$ in the same double coset.

\subsubsection*{Notable examples of strong lumping to left and double cosets}
The $r$-random-to-top shuffle is the shuffle of a deck of $n$ cards
in which $r$ cards, chosen uniformly at random, 
are moved to the top of the pack, while maintaining
their relative order. As we saw in the special case $r=1$
in \S\ref{subsec:shuffles}, it lumps strongly to the left cosets of $\Sym_r \times \Sym_{n-r}$;
the lumped Markov chain has states corresponding to
the $\binom{n}{r}$ possible sets of labels of the top $r$ cards in the deck, and in fact it always transitions to a uniform random state. 
In \cite[\S 3.2]{BW} it is shown that 
a family of shuffles generalizing the $r$-random-to-top shuffle
lumps strongly to the double cosets of $\Sym_r \times \Sym_{n-r}$ in $\Sym_n$; 
the lumped random walks are reversible and have some remarkable spectral properties.
The time reversal of the $r$-random-to-top shuffle
is the $r$-top-to-random shuffle, in which 
the top $r$ cards from the pack are moved to $r$ positions in the pack chosen uniformly at random,  again maintaining the relative order of the $r$ moved cards and the relative order of the other $n-r$ cards.
By Theorem~\ref{thm:mainTimeReversal},
the $r$-top-to-random shuffle
lumps exactly to the left cosets of $\Sym_r \times \Sym_{n-r}$ 
in $\Sym_n$, and lumps exactly to the double cosets of $\Sym_r \times \Sym_{n-r}$ in $\Sym_n$. 
These shuffles are notable because the weights are not conjugacy
invariant (except in the trivial case $r=n$), and so the full power of
the results from
\cite{BW} and Theorem~\ref{thm:mainTimeReversal} is required. 

We refer the reader to~\cite{DiaconisRamSimper} for a wealth of further examples of strong lumping to double
cosets of random walks driven by conjugacy invariant weights: 
particularly notable is Example~2.3, 
that the random walk on a finite group of Lie type driven by
multiplication by such a weight lumps strongly to the double cosets of a Borel subgroup,
and so induces a random walk on its Weyl group.

\subsubsection*{Weak lumping of other left-invariant random walks}
Besides the result of Pang already mentioned, where the lumping is to partitions, we know of 
no substantial examples in the literature of weak lumpings of left-invariant random walks
that are not also strong lumpings. In particular, the present paper is the first to study weak lumpings of the left-invariant random walk to left cosets; 
as our  extended examples in \S\ref{subsec:shuffles} and \S\ref{subsec: abelian example} show,
the theory we develop is broadly applicable.
We hope this paper will lead to further interesting examples of weak lumping.

\subsection{Outline}
The outline of this paper is as follows.
\begin{itemize}
\item In \S\ref{sec: probabilityPreliminaries} 
we present the necessary background from probability theory. In particular 
Theorem~\ref{thm: GL characterisation of WL} and
Corollary~\ref{cor: GL certificate theorem} give a short self-contained
proof of a characterisation of weak lumping due to Gurvits and Ledoux \cite{GL}.
\item In \S\ref{sec:algebraicPreliminaries} we give the necessary background from representation theory
including a primer on induced representations and the Wedderburn decomposition.
\item In \S\ref{sec:doubleCosets} we collect basic results on the double coset decomposition of groups
including a special case of Mackey's restriction formula.
\item In \S\ref{sec: a characterisation of WL} we use Theorem~\ref{thm: GL characterisation of WL}
to show that weak lumping of the right-invariant random walk is controlled by the behaviour of left ideals
of~$\C[G]$ of the form $\C[G]e$ where $e \in \Eb{H}$ and to prove
Theorem~\ref{thm:mainGL}, Corollary~\ref{cor:GLmodules},
Corollary~\ref{cor:LTest} and Theorem~\ref{thm: main testDist}.
\item In \S\ref{sec: tests} we give efficient computational algorithms
for the tests in Corollary~\ref{cor:LTest} and Theorem~\ref{thm: main testDist}, illustrated by the running
example in~\S\ref{subsec:shuffles}.
\item In \S\ref{sec: all WL weights} we make a detailed study of the weak lumping algebras $\Theta(e)$ 
defined in~\eqref{eq:Theta}
and prove Proposition~\ref{prop: irredundant}.
\item In \S\ref{sec:real} we prove Corollary~\ref{cor:real}, which
shows that only idempotents with real coefficients need to be considered
in Theorem~\ref{thm:mainGL} to obtain all weakly lumping weights.
\item In \S\ref{sec: proof of mainTransitionMatrices} we prove Theorem~\ref{thm:mainTransitionMatrices}.
\item In \S\ref{sec: time reversal} we state and prove Theorem~\ref{thm:DualityForMarkovChainLumping}, a new
result relating duality and time reversal for weak lumping of general finite time-homogeneous Markov chains, and use it to prove Theorem~\ref{thm:mainTimeReversal} and Corollary~\ref{cor:strongExact}.
\item In \S\ref{sec:real} we prove Corollary~\ref{cor:real}, which
shows that only idempotents with real coefficients need to be considered
in Theorem~\ref{thm:mainGL} to obtain all weakly lumping weights.
\item In \S\ref{sec:interpolating} we prove Proposition~\ref{prop:interpolating}; this generalizes part of the
shuffling example in \S\ref{subsec:shuffles} and gives a rich supply of left ideals for which the
left-invariant random walk lumps stably.
\item In \S\ref{sec:ThetaCharacterisation} we show how to interpret the conditions in Theorem~\ref{thm:mainGL}
and our other main results working double coset by double coset.
\item In \S\ref{sec:abelian} we finish by applying the results of the previous section to the case where
$H$ is abelian, obtaining a very sharp description of the weakly lumping weights in 
Proposition~\ref{prop:abelianThetaPerp} and proving Corollary~\ref{cor:abelian}.
 We finish with an extended example in which $G = \Sym_4$ and $H
= \langle (1,2,3,4) \rangle$.
\end{itemize}

\section{Background from Markov chain theory}
\label{sec: probabilityPreliminaries}

In this paper we are 
concerned with \emph{discrete time-homogeneous Markov chains} (DTHMCs) with finite state spaces.  
A DTHMC is specified by a state space $A$, a probability measure $\mu$ on $A$,  and a stochastic matrix $P$  whose rows and columns are indexed by the elements of $A$. We call $\mu$ the \emph{initial distribution}, $P$ the 
\emph{transition matrix}, and we denote this DTHMC by $X = \MC(\mu,P)$. It is a random variable $X = \left(X_t\right)_{t \in \mathbb{N}_0}$ with the properties that
\[\P[X_0 = x] = \mu(x), \quad \text{for each $x \in A$,}\]
and for any  $t \ge 1$ and any sequence $x_0, \dots, x_t \in A$,
\[
    \P[
    X_0=x_0,\dots,X_t=x_t
    ]
    = 
    \mu(x_0)
    P(x_0,x_1)
    \dots 
    P(x_{t-1},x_t)
    .
\]
It follows from this definition that for any $t \ge 0$, any $y \in A$, and any sequence $x_0, \dots, x_t$ such that $\P[X_0 = x_0, \dots, X_t = x_t] > 0$,  we have
\[
\P[X_{t+1} = y \mid X_0 = x_0, \dots, X_t = x_t] = \P[X_{t+1} = y \mid X_t = x_t] = P(x_t, y).
\]

In a predecessor to this paper, Britnell and Wildon \cite{BW}, the term `Markov chain' refers to a transition matrix, with the initial distribution left unspecified. This is how the term was used in the older probability literature, for example Kemeny and Snell \cite{KS}. In the recent probability literature, a `Markov chain' more often means a stochastic process with a countable state space. This is the usage here. 

We  follow the usual convention of probabilists 
that transition matrices act on the right on row vectors. Thus the matrix entry $P(x,y)$ or $P_{x,y}$ stands for the conditional probability that our Markov chain steps next to $y$ given that it is currently at $x$. For example, the left-invariant random walk driven by a weight $w$
has transition matrix  $P(x, y) = w(x^{-1}y)/w(G)$. If $w(G) = 1$ we say the weight is \emph{normalized}. The statement that a probability distribution $\mu$ is \emph{stationary} for $P$ is written algebraically as $\mu P = \mu$. When $\mu$ is stationary for $P$ and $X = \MC(\mu,P)$, we have for each $t \ge 0$ that $X_t \sim \mu$, meaning that $X_t$ is distributed according to $\mu$.

When $X = (X_t)_{t \geq 0}$ is a DTHMC with state space $A$ and $f: A~\to~B$ is a function, we will write $f(X)$ to mean the discrete time stochastic process $\left(f(X_t)\right)_{t \geq 0}$.  The process $f(X)$ is \emph{typically not a Markov chain}. This is neatly illustrated by the deterministic bottom-to-top shuffle of a deck of $n \ge 3$ cards that is initially uniformly distributed among the $n!$ possible orders. It is also possible for $f(X)$ to be a Markov chain without being time-homogeneous, as the following example shows.

\begin{example}[Inhomogeneous Markovian image process]\label{periodic example} Let $n > 2$ and let
$G = \langle \sigma, \tau \mid \sigma^n, \tau^2,  (\sigma\tau)^2 = 1\rangle$ be the dihedral group of order $2n$, 
as seen in \S\ref{subsec:whyIrreducible} in the case $n=5$.  
Let $(X_t)_{t \ge 0}$ be the left-invariant random walk on $G$ driven by the weight $w = \frac{1}{2}(\tau + \sigma \tau)$, with initial distribution $\alpha$. This weight is irreducible, but the transition matrix of the walk is periodic. Let~$H$ be the order $2$ subgroup $\langle \tau\rangle$. We get an inhomogeneous Markov process $(X_t H)_{t \ge 0}$ on $G/H$ if the initial distribution $\alpha$ is either concentrated on $\langle \sigma \rangle$ or concentrated on $\langle \sigma \rangle \tau$. In either case, the transition matrices of $(X_t H)_{t \ge 0}$ form an alternating sequence. For any other
initial distribution $\alpha$, the process $(X_t H)_{t \ge 0}$ is not a Markov chain.
\end{example}

\begin{definition}[Weak lumping with a given initial distribution]
    When $X = \MC(\mu,P)$ and $f(X)$ is a DTHMC, we say that $X$ \emph{lumps weakly} 
    under $f$, and we say that $P$ \emph{lumps weakly under $f$ starting in distribution} $\mu$.
\end{definition} 
\begin{definition}[Weak lumpability]
    A transition matrix $P$ on a state space $A$ is \emph{weakly lumpable under $f$} if there exists an initial distribution $\mu$ on $A$ such that $f(\MC(\mu,P))$ is a DTHMC.
\end{definition}
If the map $f$ is understood in the context, for example, when it is the quotient map from $A = G$ to $B = G/H$, we may also talk about weak lumping \emph{to} $B$ or \emph{on} $B$ instead of weak lumping under $f$. We avoid using the terms \emph{lumping}, \emph{lumpable}, and \emph{lumpability} on their own, because in some of the literature
(for instance \cite{BW}) this has been used to refer to the concept of strong lumping
in the sense of Definition~\ref{defn:strongLumping} below.

\subsection{The Gurvits--Ledoux characterisation of weak lumping of time homogeneous
Markov chains}\label{SS:GL}

Gurvits and Ledoux \cite{GL} characterised weak lumping of DTHMCs with finite state spaces using linear algebra. 
They also considered higher order Markov chains and hidden Markov chains with probabilistic output functions. In their notation, transition matrices act on the left on column vectors. To establish our preferred notation and to keep our paper self-contained, we now give a rapid exposition of some results from \cite{GL} which we will use in our study of weak lumping of left-invariant random walks from $G$ to $G/H$. We also introduce a new notion of \emph{stable} weak lumping and develop some of its properties.

Let $A$ and $B$ be finite sets and $f : A\to B$ a surjective function. Then $A$ is partitioned into the non-empty \emph{lumps} $f^{-1}(b)$ for $b \in B$. 
The linear map $F: \mathbb{R}^A \to \mathbb{R}^B$ induced by $f$ is
defined on the canonical bases of $\R^A$ and $\R^B$ by 
$F(e_a) = e_{f(a)}$. Acting on the right on row vectors,~$F$ is represented by the matrix defined by
$F_{a, b} = 1$ if $f(a) = b$ and $0$ otherwise. 
For any vector subspace $V$ of $\mathbb{R}^A$, we define 
\begin{equation}\label{eq:VcircDefn}
    V^\circ = V \cap \ker F.
\end{equation}
The space $V^\circ$ consists of those vectors in $V$ such that the sum of the coordinates over the lump $f^{-1}(b)$ vanishes for each $b \in B$.
We also define linear endomorphisms of $\mathbb{R}^A$ which act on the right on standard basis vectors ($e_a$ for $a \in A$) by 
\[
    e_a\Pi_b = 
    \begin{cases} 
        e_a & \text{if $f(a) = b$,} \\ 
        0   & \text{otherwise.}
    \end{cases}
\]
Thus $\Pi_b$ is the projection onto the direct summand indexed by $b$ in the direct sum
\[
    \mathbb{R}^A = \bigoplus_{b \in B} \mathbb{R}^{f^{-1}(b)}.
\]
Since we are using row vectors, the matrices $P$, $F$ and $\Pi_b$ act on the right. When $V$ is a linear subspace of $\mathbb{R}^A$ and $M$ is a linear map with domain $\mathbb{R}^A$, we write 
\[
    VM = \{vM: v \in V\}.
\]
We also use the above notations adapted to complex vector spaces.   

\begin{definition}\label{def: GL space} If $V$ is a real linear subspace of $\R^A$ or a complex linear subspace of $\C^A$ then we call $V$ a \emph{Gurvits--Ledoux space} if it satisfies $VP \subseteq V$ and $V\Pi_b \subseteq V$ for every $b \in B$.
\end{definition}
Every Gurvits--Ledoux space $V$ satisfies $V = \bigoplus_{b \in B} V\Pi_b$.
\begin{definition}\label{def: minimal GL space}
    For $P\in\Mat(\R^A)$ a stochastic matrix and $\alpha \in \R^A$ a probability distribution, let $V(f,P,\alpha)$ be the minimal vector space $V \subseteq \mathbb{R}^A$ such that
    \begin{defnlist} 
        \item[(a)] $\alpha \in V$,
        \item[(b)] $V P \subseteq V$,
        \item[(c)] $V\Pi_{b} \subseteq V$ for all $b\in B$.
    \end{defnlist}    
\end{definition}
This definition makes sense because each of conditions (a)--(c) is closed under intersection and $V = \R^A$ satisfies (a)--(c). Note that $V(f,P,\alpha)$ is the minimal real Gurvits--Ledoux space containing $\alpha$. It is straightforward to check that the minimal complex Gurvits--Ledoux space containing $\alpha$ is $\C \otimes_\R V(f,P,\alpha)$. 

Readers who wish to compare our exposition with that of Gurvits and Ledoux \cite{GL} will find that they take $A = \{1, \dots, N\}$ and $B = \{1, \dots, M\}$, their map $\varphi$ is our $f$, and the space we have called $V(f,P,\alpha)$ is called $\mathcal{CS}(\alpha,\Pi_.,P)$ there. We shall not comment further on the translation between \cite{GL} and our notation, but simply state their results in our notation. 

 We have $\ker F = \bigoplus_{b \in B} (\ker F)\Pi_b$, so 
\begin{equation}\label{eq: Vcirc is a direct sum}
    V(f,P,\alpha)^\circ 
    =
    \bigoplus_{b \in B} \bigl( (V(f,P,\alpha)\Pi_b) \cap \ker F \bigr)
    =
    \bigoplus_{b \in B} V(f,P,\alpha)^\circ \,\Pi_b.
\end{equation}
Considering $X = \MC(\alpha, P)$, we say that a finite sequence $(b_0, \dots, b_t)$ of elements of $B$ is 
an \emph{$\alpha$-possible sequence} if $\P[f(X_0) = b_0, \dots, f(X_t) = b_t] > 0$.  For any such sequence, let $C(\alpha;b_0, \dots, b_t)$ be the conditional distribution of $X_t$ given $f(X_0) = b_0, \dots, f(X_t) = b_t$, thought of as a row vector in $\R^A$.
The vector space $V(f,P,\alpha)$  is the minimal linear subspace of $\R^A$ that contains $C(\alpha; b_0, \dots, b_t)$ for every $\alpha$-possible sequence $(b_0, \dots, b_t)$.  

Notice that $f(X)$ is a time-homogeneous Markov chain if and only if the conditional distribution of $f(X_{t+1})$ given $f(X_0) = b_0, \dots, f(X_t) = b_t$ depends only on the value $b_t$,
and is given by the same function for all times $t$.  In other words, $f(X)$ is a time-homogeneous Markov chain if and only if 
\[
    C(\alpha; b_0, \dots, b_t)PF = C(\alpha; b_0', \dots, b'_{t'})PF
\]
whenever $(b_0, \dots, b_t)$ and $(b_0', \dots, b'_{t'})$ are $\alpha$-possible sequences and~$b_t = b'_{t'}$. 
For each $b \in B$, let
\[
    S(\alpha,b) = \bigl\{C(\alpha; b_0, \dots, b_t): t \ge 0,\, b_t = b, \,(b_0, \dots, b_t) \;\text{$\alpha$-possible}\bigr\}.
\]
Observe that 
\[
    V(f,P,\alpha)\Pi_b 
    = 
    \bigl\langle S(\alpha,b)\bigr\rangle.
\]
where, as ever, angled brackets denote the linear span of a set of vectors.  Hence
\[
    V(f,P,\alpha)^\circ\, \Pi_b = V(f,P,\alpha)\Pi_b \cap \ker F  = \bigl\langle \mu - \nu : \mu, \nu \in S(\alpha,b)\bigr\rangle.
\]
We see that the following are equivalent: 
\begin{itemize}
    \item $f(X)$ is a time-homogeneous Markov chain,
    \item for all $b \in B$, and for all $\mu, \nu \in S(\alpha, b)$, we have $(\mu-\nu)PF = 0$, 
    \item for all $b \in B$, we have $V(f,P,\alpha)^\circ\, \Pi_b PF = \{0\}$, and
    \item $V(f,P,\alpha)^\circ PF = 0$.
\end{itemize}
The third and fourth conditions above are equivalent by~\eqref{eq: Vcirc is a direct sum}. Finally, 
\[
    V(f,P,\alpha)^\circ P \subseteq V(f,P,\alpha)P \subseteq V(f,P,\alpha),
\]
so we have 
\[
    V(f,P,\alpha)^\circ PF = 0 
    \iff
    V(f,P,\alpha)^\circ P \subseteq V(f,P,\alpha)^\circ.
\]
We have proved the following result of Gurvits and Ledoux.

\begin{theorem}[Gurvits and Ledoux {\cite[Corollary 9]{GL}}]{\,}\label{thm: GL characterisation of WL}
The Markov chain $\MC(\alpha,P)$ lumps weakly under $f$ if and only if
     $V(f,P,\alpha)^\circ P \subseteq V(f,P,\alpha)^\circ$.
\end{theorem}

Now consider some non-empty set $\mathcal{A} \subset \R^A$ of probability vectors. We ask whether the image process $f(\MC(\alpha,P))$ is a DTHMC for every $\alpha \in \mathcal{A}$, with a transition matrix that does not depend on the choice of $\alpha$ from $\mathcal{A}$.

\begin{definition}
For $P\in\Mat(\R^A)$ a stochastic matrix and $\mathcal{A} \subset \mathbb{R}^A$ a non-empty set of probability vectors, $V(f,P,\mathcal{A}) = \sum_{\alpha \in \mathcal{A}} V(f,P,\alpha)$.
\end{definition}
It is straightforward to see that $V(f,P,\mathcal{A})$ is the minimal Gurvits--Ledoux subspace of $ \mathbb{R}^A$ that contains every element of $\mathcal{A}$.
\begin{theorem}[Gurvits and Ledoux {(\cite[Corollary 11]{GL})} ]
\label{thm: GL same image transition matrix}
For a non-empty set $\mathcal{A}$ of probability vectors in $\R^A$, the following are equivalent:
\begin{thmlistEL}\item[\emph{(a)}] there exists a stochastic matrix $Q$ from $B$ to $B$ such that for every $\alpha \in \mathcal{A}$, $f(\MC(\alpha,P))$ is a time-homogeneous Markov chain with transition matrix $Q$, 
 \item[\emph{(b)}] $V(f,P,\mathcal{A})^\circ P \subseteq V(f,P,\mathcal{A})^\circ$, 
\item[\emph{(c)}] $V(f,P,\mathcal{A})^\circ P F = 0$.
\end{thmlistEL}
\end{theorem}
\begin{proof}
Since $V(f,P,\mathcal{A})P \subseteq V(f,P,\mathcal{A})$, (b) and (c) are equivalent. Condition (a) implies that $V(f,P,\alpha) \subseteq \ker(PF - FQ)$ for each $\alpha \in \mathcal{A}$, so $V(f,P,\mathcal{A}) \subseteq \ker(PF-FQ)$. Since $V(f,P,\mathcal{A})^\circ F = 0$ by definition, we deduce $V(f,P,\mathcal{A})^\circ PF = 0$. Thus (a) implies $(c)$.  Now assume (c) and let $\alpha \in \mathcal{A}$. We have $V(f,P,\alpha)^\circ \subseteq V(f,P,\mathcal{A}
)^\circ$ so $V(f,P,\alpha)^\circ PF = 0$ and $V(f,P,\alpha)^\circ P \subseteq V(f,P,\alpha)^\circ$ and, by Theorem~\ref{thm: GL characterisation of WL}, $\MC(\alpha,P)$ lumps weakly under $f$. Let $Q^\alpha$ be the transition matrix of $f(\MC(\alpha,P))$, which is defined only on the set 
$$\{f(x): (\alpha P^m)_x > 0 \text{ for some } m \ge 0\}.$$  
To complete the proof that (c) implies (a) we must show that the $y$-rows of $Q^\alpha$ and $Q^\beta$ agree, for any $\alpha, \beta \in \mathcal{A}$ and any $y \in B$ such that row $y$ is defined for both $Q^\alpha$ and $Q^\beta$. Choose $m \ge 0$ such that $v = \alpha P^m \Pi_y \neq 0$, and $n \ge 0$ such that $w = \beta P^n \Pi_y \neq 0$. Then $v$ and $w$ are both supported on the lump $f^{-1}(y)$, and $w(f^{-1}(y))v - v(f^{-1}(y)) w \in V(f,P,\mathcal{A})^\circ$, so 
\begin{eqnarray*}0 &=& \bigl(  w\bigl(f^{-1}(y)\bigr)v - v\bigl(f^{-1}(y)\bigr) w\bigr) PF 
\\ &=& w\bigl( f^{-1}(y) \bigr) vFQ^\alpha - v\bigl( f^{-1}(y)\bigr) w F Q^\beta \\ &=& w\bigl(f^{-1}(y)\bigr)v
\bigl(f^{-1}(y)\bigr) e_y (Q^\alpha - Q^\beta).\end{eqnarray*}
Hence the $y$-rows of $Q^\alpha$ and $Q^\beta$ coincide. It follows that (c) implies (a).
\end{proof}

The matrix $Q$ in the above theorem may not be uniquely determined, because there may be some element $b \in B$ such that for any $\alpha \in \mathcal{A}$ the Markov chain $\MC(\alpha,P)$ never visits $f^{-1}(b)$. 

We shall also make use of the following consequences of Theorem~\ref{thm: GL same image transition matrix}.
\begin{corollary}[Algorithm to test for weak lumping of $\MC(\alpha,P)$ under $f$]\label{cor: algorithmic weak lumping test}
Begin by setting $V_0 = \langle \alpha \Pi_b: b \in B \rangle$. Then, inductively for $\tn = 0,1,2,\dots$, define
$$V_{\tn+1} = V_{\tn} + \sum_{b \in B} V_{\tn}P\Pi_b.$$
Let $k = \inf \{ \tn \ge 0: V_{\tn+1} = V_\tn\}$. Then $k \le |A|-1$, $V_n = V_k$ for all $\tn > k$, and 
$V_k = V(f,P,\alpha)$. Hence $\MC(\alpha,P)$ lumps weakly under $f$ if and only if $V_k^\circ P F = \{0\}$.  
\end{corollary} 
\begin{proof}
 The assertion that $k \le |A|-1$ follows from the fact that $\dim V_0 \ge 1$, $V_{\tn+1} \supseteq V_\tn$ and therefore $ \dim V_{\tn+1} \ge \dim V_\tn$, for each $\tn \ge 0$, and $V_\tn \subseteq \mathbb{R}^A$ so $\dim V_\tn \le |A|$ for all $\tn$. 
 So $$k = \inf \{\tn \ge 0: \dim V_{\tn+1} = \dim V_\tn\} \le |A|- \dim V_0 \le |A| - 1.$$
 From the inductive definition of $V_{\tn+1}$ in terms of $V_\tn$ it is clear that $V_{k+1} = V_k$ implies $V_\tn = V_k$ for all $\tn > k$.  Each $V_\tn$ is the direct sum of its $\Pi_b$-projections, and each $V_\tn$ contains $\alpha$. Moreover, $$V_k P \subseteq \sum V_k P \Pi_b \subseteq V_{k+1} = V_k$$ so $V = V_k$ is a vector space satisfying conditions (a)--(c) of Definition~\ref{def: minimal GL space}. To show that $V_k$ is the minimal such space, check by induction on $\tn$ that any such $V$ must contain $V_\tn$ for every $\tn \ge 0$, and in particular $V \supseteq V_k$.  
\end{proof}
\begin{definition}[Stable lumping]\label{def: stable} Let $P\in\Mat(\R^A)$ be a stochastic matrix and let $V$ be any real vector subspace of $\R^A$ or complex vector subspace of $\mathbb{C}^A$ such that
    \begin{thmlist}
        \item[(a)] $V$ contains at least one probability vector.
        \item[(b)] $V P \subseteq V$,
        \item[(c)] $V\Pi_{b} \subseteq V$ for all $b\in B$,
        \item[(d)] $V^\circ P F = 0$ or equivalently $V^\circ P \subseteq V^\circ$, where $V^\circ := V \cap \ker F$.
    \end{thmlist}
     Then we say that $P$ \emph{lumps weakly under $f$ with stable space $V$}, or more briefly that $P$ \emph{lumps stably for} $V$. 
\end{definition}
If $\MC(\alpha,P)$ lumps weakly under $f$ then $P$ lumps stably for the real vector space $V(f,P,\alpha)$. The same matrix $P$ may also lump stably for other spaces $V \subseteq \R^A$ such that $V(f,P,\alpha) \subseteq V$. If~$P$ lumps stably for a complex vector space $V \subseteq \C^A$ then $P$ also lumps stably for the real vector space $\mathbb{R}^A \cap V$. 
On the other hand, if $P$ lumps stably for a real vector space $V \subseteq \R^A$ then~$P$ also lumps stably for the complexification $V \otimes_{\R} \C$. Our reason for considering complex vector spaces in this paper is that,
as remarked in the introduction,
 we wish to make use of results in representation theory that hold for the group algebra $\C[G]$, but do not hold
  for $\R[G]$,
 because~$\mathbb{R}$ is not algebraically closed. 

\begin{remark}\label{rem: stable spaces and stable ideals} We show below (see Proposition~\ref{prop:lumpsStably}) that Definition~\ref{def: stable} is compatible with Definition~\ref{defn:lumpsStably} in the case where $P = P_w$ is the transition matrix of a left-invariant random walk on a finite group $G$ driven by a weight $w$, $f$ is the natural map $G \rightarrow G/H$, and the space $V \subseteq \mathbb{C}^G$ is a left ideal $\C[G]e$ of $\C[G]$, where $e \in \Eb{H}$. That is, $w$ lumps stably for $\C[G]e$ in the sense of Definition~\ref{defn:lumpsStably} if and only if $P_w$ lumps stably for $V$ in the sense of Definition~\ref{def: stable}. 
\end{remark}

The next result is a corollary of Theorem~\ref{thm:  GL characterisation of WL}; a partial converse to it is Lemma~\ref{lem:probcharStable}.
\begin{corollary}[Stable spaces as certificates of weak lumping]\label{cor: GL certificate theorem}{\,}
Suppose $P$ lumps weakly under $f$ with (real or complex) stable space $V$, and let $\alpha$ be any probability vector in $V$. Then $X = \MC(\alpha,P)$ lumps weakly under $f$. Moreover, for each $t \ge 0$ the conditional distribution of $X_t$ given $f(X_0), \dots, f(X_t)$ always lies in $V$. 
\end{corollary}
\begin{proof}
 Conditions (b) and (c) in Definition~\ref{def: stable} imply that $V(f,P,\alpha) \subseteq V$, and (d) implies that $(V \cap \ker F)P \subseteq \ker F$, so $$V^\circ(f,P,\alpha)P =  \bigl(V(f,P,\alpha) \cap \ker F\bigr)P \subseteq (V \cap \ker F)P \subseteq \ker F.$$
 We also have from the definition of $V(f,P,\alpha)$ that $V(f,P,\alpha) P \subseteq V(f,P,\alpha)$, so
 $V^\circ(f,P,\alpha) P \subseteq V^\circ(f,P,\alpha)$, which by Theorem~\ref{thm: GL characterisation of WL} implies that $\MC(\alpha,P)$ lumps weakly under $f$.
 
 For any $\alpha$-possible sequence $(b_0, \dots, b_t)$,  the conditional distribution of $X_t$ given $f(X_0) = b_0, \dots, f(X_t) = b_t$ is the normalisation of the vector $v$, where $v = \alpha \Pi_{b_0}$ if $t = 0$, and if $t=1$ then
 $ v = \alpha\Pi_{b_0}(P\Pi_{b_1})\dots(P\Pi_{b_t})$. By induction using (b)  and (c), we have $v \in V$.
\end{proof}
\subsection{The irreducible case}
\begin{lemma}\label{lem:useErgodicThm}
 If $P \in \mathrm{Mat}(\R^A)$ is an irreducible stochastic matrix and $\mu$ is its unique stationary distribution then $\mu \in V(f,P,\alpha)$ for every probability distribution $\alpha$ on $A$. Moreover, if $P$ lumps stably for $V$, then $\mu \in V$. 
\end{lemma}
\begin{proof}
 By the ergodic theorem for irreducible Markov chains (see for instance \cite[Theorem 5.1.2(b)]{KS}), 
 we have $\frac{1}{n}\sum_{i=0}^{n-1} \alpha P^i \to \mu$ as $n \to \infty$. The terms in this convergent sequence belong to $V(f,P,\alpha)$ by properties (a) and~(b) in Definition~\ref{def: minimal GL space}. Since $V(f,P,\alpha)$ is a finite-dimensional vector space, it is topologically closed, hence $\mu \in V(f,P,\alpha)$. 
  If $P$ lumps stably for~$V$ then we may pick a probability vector $\alpha \in V$ and then $V(f,P,\alpha) \subseteq V$ so $\mu \in V$.
\end{proof}

\begin{corollary}[Weak lumpability of irreducible transition matrices]\label{cor: weak lumpability of irreducible transition matrices}
 If $P \in \mathrm{Mat}(\R^A)$ is an irreducible stochastic matrix and $\mu$ is its unique stationary distribution
 then $P$ is weakly lumpable under $f:A \to B$ if and only if $\MC(\mu,P)$ lumps weakly under $f$.
 \end{corollary}
 \begin{proof}
 If $\MC(\mu,P)$ lumps weakly under $f$ then $P$ is weakly lumpable under~$f$, by the definition of weak lumpability. For the converse, suppose that $\MC(\alpha,P)$ lumps weakly under $f$. 
 Then $\mu \in V(f,P,\alpha)$ by Lemma~\ref{lem:useErgodicThm}.  It follows that $V(f,P,\mu) \subseteq V(f,P,\alpha)$ and hence $V^\circ(f,P,\mu) \subseteq V^\circ(f,P,\alpha)$ and therefore $V^\circ(f,P,\mu)P \subseteq \ker F$. Since also $$V^\circ(f,P,\mu)P \subseteq V(f,P,\mu)P \subseteq V(f,P,\mu),$$ we find
 $$V^\circ(f,P,\mu)P \subseteq V(f,P,\mu) \cap \ker F = V^\circ(f,P,\mu)$$ which by Theorem~\ref{thm: GL characterisation of WL} implies that $\MC(\mu,P)$ lumps weakly under $f$. 
 \end{proof}
 We remark that Corollary~\ref{cor: weak lumpability of irreducible transition matrices} was known long before the work of Gurvits and Ledoux: see for example \cite[Theorem 4.6.3]{KS}. We believe that the following observation is new.

\begin{lemma}[Lattice of stable spaces for a weakly lumpable irreducible transition matrix]\label{L: lattice of stable spaces}
 Let $P \in \mathrm{Mat}(\R^A)$ be an irreducible stochastic matrix that is weakly lumpable
under $f :A \rightarrow B$, with unique stationary distribution $\mu$.  The set of real subspaces of $\R^A$ for which $P$ lumps stably is a lattice under intersection and sum, with bottom element $V(f,P,\mu)$. Likewise the set of complex subspaces of $\C^A$ for which $P$ lumps stably is a lattice with bottom element $V(f,P,\mu) \otimes_{\R} \C$.
\end{lemma}
\begin{proof}
The proof of Corollary~\ref{cor: weak lumpability of irreducible transition matrices} shows that when $P$ is irreducible, every~$V$ for which $P$ lumps stably must contain $\mu$ and hence all of $V(f,P,\mu)$, and if~$V$ is complex, then $V(f,P,\mu) \otimes_{\R} \C \subseteq V$.

Next, consider two spaces $U$ and $V$ (both real or both complex) such that~$P$ lumps stably for $U$ and also for $V$.  Let $V = U \cap V$. Then $\mu \in U$ and $\mu \in V$ so $\mu \in U \cap V$. We have $VP \subseteq UP \subseteq U$ and likewise $VP \subseteq V$ so $VP \subseteq V$. For any $b \in B$, $V\Pi_b \subseteq U\Pi_b \subseteq U$ and likewise $V\Pi_b \subseteq V$ so $V\Pi_b \subseteq V$. Finally, $V^\circ  = U \cap V \cap \ker F  = U^\circ \cap V^\circ$ so $V^\circ PF \subseteq U^\circ PF = 0$. We have shown that $P$ lumps stably for $V$.

Now let $W = U + V$. Then $\mu \in W$ and  $WP  = UP + V P \subseteq U + V = W$. 
We have $$W = U + V = \bigoplus_{b \in B} U \Pi_b  + \bigoplus_{b \in B} V \Pi_b = \bigoplus_{b \in B} (U \Pi_b + V \Pi_b)$$ so $$ W\Pi_b = U\Pi_b +V\Pi_b \subseteq U + V = W$$
hence $W = \bigoplus_{b \in B} W\Pi_b$. Since $P$ is irreducible and $f$ is surjective, for each $b\in B$ we have $\mu(f^{-1}(b)) > 0$. By Lemma~\ref{lem:useErgodicThm} we have $\mu \in U$ and $\mu \in V$, and so $U\Pi_b = \langle \mu\Pi_b \rangle \oplus U^\circ \Pi_b$, for each $b \in B$, and likewise for $V$. Hence $$W\Pi_b = (U^\circ\Pi_b + V^\circ\Pi_b) \oplus \langle \mu\Pi_b\rangle,$$ and therefore $W^\circ = U^\circ + V^\circ$. Finally, $W^\circ PF = U^\circ PF + V^\circ PF = 0+0$. We have shown that~$P$ lumps stably for $W$.
\end{proof} 
It follows that for any irreducible $P$ that is weakly lumpable under $f$, there exists a unique maximal space $V_{\max}(f,P)$ for which $P$ lumps stably. It contains $V(f,P,\alpha)$ as a subspace for every initial distribution $\alpha$ such that $\MC(\alpha,P)$ lumps weakly under $f$.

 \begin{corollary}[Initial distributions compatible with an irreducible weight]\label{cor: GL good initial distributions}
Let $P \in \mathrm{Mat}(\R^A)$ be an irreducible stochastic matrix with unique stationary distribution $\mu$. Suppose $\MC(\mu,P)$ lumps weakly under the surjective map $f: A \to B$. Let $\mathcal{I}$ be the set of probability distributions $\alpha$ such that $\MC(\alpha,P)$ lumps weakly under $f$. Then $\mathcal{I}$ is equal to the convex polytope
$\Delta \cap V_{\max}(f,P)$, where $\Delta$ is the simplex of probability vectors in $\R^A$. Moreover, there exists a unique stochastic matrix $Q \in \mathrm{Mat}(\R^B)$ that serves as a transition matrix for $f(\MC(\alpha, P))$ for every $\alpha \in I$. This $Q$ is given by
\begin{equation}\label{eq: pseudo-aggregation} Q(i,j) = \frac{1}{\mu(f^{-1}(i))}\sum_{x \in f^{-1}(i)}\;\sum_{y \in f^{-1}(j)} \mu(x)P(x,y).
\end{equation}
 \end{corollary}
\begin{proof}
The characterisation of $\mathcal{I}$ is a straightforward corollary of Lemma~\ref{L: lattice of stable spaces} and Corollary~\ref{cor: GL certificate theorem}.
The existence of a single transition matrix $Q$ from~$B$ to $B$ which serves as a transition matrix for $f(\MC(\alpha,P))$ for every $\alpha \in \mathcal{I}$ follows from Theorem~\ref{thm: GL same image transition matrix}, by taking $\mathcal{A} = \mathcal{I}$ and noting that then $\mathcal{I} \subset V(f,P,\mathcal{I}) \subseteq V_{\max}(f,P)$ so condition (c) of Theorem~\ref{thm: GL same image transition matrix} is satisfied. Because $P$ is irreducible and $f$ is surjective, there is in fact a unique transition matrix for $f(\MC(\alpha,P))$, for each $\alpha \in \mathcal{I}$. The expression~\eqref{eq: pseudo-aggregation} for $Q$ in terms of $P$ and $\mu$ is obtained by calculating $$\P[f(X_1) = j \mid f(X_0) = i] = \frac{\P[f(X_0) = i \text{ and 
 } f(X_1) = j]}{\P[X_0 = i]}.$$
 Note that by the irreducibility of $P$ and the surjectivity of $f$, the denominator $\P[X_0 = i] = \mu(f^{-1}(i))$ is non-zero for each $i \in B$. 
\end{proof}

Let us emphasise that in the irreducible case the transition matrix for
the lumped process $f(\MC(\alpha, P))$ is independent of the initial distribution~$\alpha$. In the reducible case this may fail, as we saw in Example~\ref{ex: reducible dihedral}. It can even fail in the case where $\MC(\alpha,P)$ lumps weakly 
under~$f$ for every initial distribution $\alpha$. An example of this is given by $A = \{1,2,3,4\}$, $B = \{1,3,4\}$, $f(1) = f(2) = 1$, $f(3) = 3$, $f(4) = 4$ and 
$$ P = \left(\begin{matrix} 0 & 0 & 2/3 & 1/3\\ 0 & 0 & 1/3 & 2/3\\ 0 & 0 & 1 & 0\\ 0 & 0 & 0 & 1 \end{matrix}\right).$$

Returning to the setting of Corollary~\ref{cor: GL good initial distributions}, it follows that $V_{\max}(f,P)$ is the maximal vector subspace $V \subseteq \mathbb{R}^A$ with the properties
 \begin{defnlist}
 \item[(1)] $V P \subseteq V$,
 \item[(2)] $V \Pi_b \subseteq V$ for each $b \in B$, and
 \item[(3)] $V \subseteq \ker(PF - FQ)$.
 \end{defnlist}
The space $V_{\max}(f,P)$ may  therefore be computed by the following linear algebra algorithm.
To begin, set 
\[ V = \bigoplus_{b \in B} \left\{v \in \mathbb{R}^{f^{-1}(b)}: v(PF - FQ) = 0\right\}.\] 
Note that this $V$ satisfies condition (3) above.
While $VP \not\subseteq V$, replace $V$ by
$$ V \cap \bigoplus_{b \in B} \left\{v \in \mathbb{R}^{f^{-1}(b)} : vP \in V \right\}.$$
The dimension of $V$ cannot increase under this operation, and if it does not decrease then $$V = \bigoplus_{b \in B} \left\{v \in \mathbb{R}^{f^{-1}(b)} : vP \in V \right\}  $$ and hence $V$ satisfies conditions (1)--(3) above, i.e.~$P$ lumps stably for $V$. If $W$ is any space for which $P$ lumps stably, then throughout the algorithm we have $W \subseteq V$. Therefore the final $V$ is equal to $V_{\max}$. By the assumption that $\MC(\mu,P)$ lumps weakly under $f$, we have $\langle \Pi_b\mu: b \in B\rangle \subseteq V_{\max}$, hence $\dim V_{\max} \ge |B|$. Therefore the algorithm terminates after no more than $|A|-|B|$ replacement steps. 

We end with a new result that we shall use in the proof of Proposition~\ref{prop:probcharrealIdeals}.

\begin{lemma}[Probabilistic characterisation of stable lumping]\label{lem:probcharStable}
Let $P \in \mathrm{Mat}\left(\R^A\right)$ be an irreducible stochastic matrix, let $f: A \to B$ be a surjection, and let $V \subseteq \R^A$ be a linear subspace containing at least one probability vector. Suppose that for every probability vector $\alpha \in V$ the Markov chain $X = \MC(\alpha,P)$ satisfies both
\begin{thmlistE}
\item $X$ lumps weakly under $f$, and
\item for $t \ge 0$, the conditional distribution of $X_t$ given $f(X_0), \dots, f(X_t)$ always lies in $V$. 
\end{thmlistE}
Then $P$ lumps weakly under $f$ with stable space $V$.
\end{lemma}
\begin{proof}
By hypothesis, $V$ contains at least one probability vector, say  $\alpha$, so condition (a) of Definition~\ref{def: stable} is satisfied. We may apply condition (ii) for $X = \MC(\alpha,P)$ and average over the distribution of $(f(X_0), \dots, f(X_t))$, to see that the distribution of $X_t$ lies in $V$ for every $t$. By considering the mean of the distributions of $X_0, \dots, X_t$ and taking the limit as $t \to \infty$, and using that $V$ is closed, we find that the unique stationary distribution $\mu$ for~$P$ also lies in $V$. 
Now take an arbitrary $\beta \in V$. For sufficiently small $\epsilon$ we have a strictly positive probability vector $\alpha = (\mu + \epsilon \beta)/(1+\epsilon\sum_{a \in A} \beta_a)$, to which we may apply condition (ii). Looking at the distribution of $X_1$ where $X = \MC(\alpha,P)$, we find that $\alpha P \in V$. Since $\mu P = P$, we deduce $\beta P \in V$, verifying condition (b). 
Considering the conditional distribution of $X_0$ given $f(X_0) = b$ (which is well-defined for each $b \in B$ because $\alpha$ has strictly positive coordinates) we deduce that $(\mu + \epsilon \beta)\Pi_b \in V$ for each $b \in B$. This holds for all sufficiently small $\epsilon$, so $\beta \Pi_b \in V$, verifying condition~(c).
Finally, let $\gamma \in V^\circ$, so $\gamma F = 0$ and in particular $\sum_{a \in A} \gamma_a = 0$. Then $\alpha^+ = \mu + \epsilon \gamma$ and $\alpha^- = \mu - \epsilon \gamma$ are probability vectors in $V$ for sufficiently small $\epsilon$, so by condition (i) and Corollary~\ref{cor: GL good initial distributions} we have $\alpha^+, \alpha^- \in V_{\max}(f,P)$. We have $\alpha^+F = \alpha^-F$ by construction, so $\alpha^+ - \alpha^- = 2\epsilon \gamma \in V_{\max}(f,P)^\circ$ and hence $\gamma PF = 0$, verifying condition (d). 
\end{proof}

\subsubsection{Earlier work of Rubino and Sericola}
The set $\mathcal{I}$ described in Corollary~\ref{cor: GL good initial distributions} was first characterised by Rubino and Sericola in \cite{RS1, RS2} by an algorithm that computes the set of extreme vertices of $\mathcal{I}$. 
As far as we know, this algorithm may have exponential worst-case complexity because the number of extreme points of a polytope can be exponentially large in the dimension and the number of faces of the polyhedron. Indeed, the dual version of McMullen's upper bound theorem 
says that a convex polyhedron of dimension $d$ defined by $N$ linear inequalities may have at most \[ \binom{N - \lceil d/2\rceil}{\lfloor d/2\rfloor}  + \binom{N - \lfloor d/2 \rfloor - 1}{\lceil d/2 \rceil} \] vertices, and this is sharp; see \cite[\S5.4 and \S5.5]{Matousek}.
In addition, Rubino and Sericola gave an explicit set of $\bigl(|B| + |B|^2 + \dots + |B|^{|A|}\bigr)$ equations in the entries of $\alpha$ and $P$, linear in $\alpha$ and polynomial of degree at most $|A|$ in $P$, such that $\MC(\alpha,P)$ lumps weakly under $f$ if and only if all of these equations are satisfied by $(\alpha, P)$. The set of such pairs $(\alpha, P)$ is therefore a semi-algebraic set. More precisely, it is the intersection of an affine algebraic variety with a product of simplices.

\subsection{Strong lumping}
Strong lumping is a special case of weak lumping, defined by a simple algebraic condition on the transition matrix.
\begin{definition}\label{defn:strongLumping} A transition matrix $P$ on $A$ \emph{lumps strongly} under a map $f: A \to B$ if whenever $f(a) = f(a')$, for every $b \in B$ we have
$$ \sum_{x \in f^{-1}(b)} P(a,x) = \sum_{x \in f^{-1}(b)} P(a', x). $$
\end{definition}
This property is also known as \emph{Dynkin's condition}.  If $P$ lumps strongly under $f$, then for every probability distribution $\alpha$ on $A$, the time-homogeneous Markov chain $\MC(\alpha,P)$ lumps weakly under $f$. In fact, when Dynkin's condition holds then $P$ lumps stably for $V = \mathbb{R}^A$. Indeed, conditions (a)--(c) of Definition~\ref{def: stable} hold trivially, and for condition (d) we must check that $(\ker F) P F = 0$. Since $\ker F$ is spanned by vectors of the form $e_a - e_{a'}$ where $f(a) = f(a')$, this is equivalent to Dynkin's condition. Now apply Corollary~\ref{cor: GL certificate theorem}.

For the case of a left-invariant random walk on a finite group $G$, a simple necessary and sufficient condition for strong lumping to $G/H$ was given by Britnell and Wildon \cite{BW}; see Corollary~\ref{cor:strongExact}.

\subsection{Exact lumping}
Exact lumping is another special case of weak lumping. It is the opposite extreme: in strong lumping the stable space $V$ certifying weak lumping is as large as possible, satisfying $V = \R^A$, while in exact lumping the stable space is as small as possible, 
satisfying $V^\circ = 0$, where $V^\circ$ is as defined in~\eqref{eq:VcircDefn}. 
\begin{definition}\label{defn:exactLumping}
Let $\alpha \in \mathbb{R}^A$ be a probability distribution and $P$ a transition matrix from $A$ to $A$. Then we say that $\MC(\alpha, P)$ \emph{lumps exactly} under $f: A \to B$ if $V(f,P,\alpha)^\circ = 0$. In the case where $P$ is irreducible, with unique stationary distribution $\mu$, we say that $P$ \emph{lumps exactly} under $f$ when $\MC(\mu,P)$ lumps exactly under $f$.
\end{definition}
Exact lumping implies weak lumping, by Theorem~\ref{thm: GL characterisation of WL}, because $$V(f,P,\alpha)^\circ = 0 \implies  V(f,P,\alpha)^\circ \, P \subseteq V(f,P,\alpha)^\circ.$$

\begin{lemma}\label{lem:EquivalentConditionsForExactLumping}
Consider the stationary distribution $\mu = \lim_{n \to \infty} \frac{1}{n} \sum_{i=0}^{n-1}\alpha P^i$. The following are equivalent:
\begin{thmlist} 
    \item[\emph{(a)}] $\MC(\alpha,P)$ lumps exactly under $f$,
    \item[\emph{(b)}]  $\dim(V(f,P,\alpha)\Pi_b) \le 1$ for all $b \in B$,
    \item[\emph{(c)}]  $V(f,P,\alpha)=\bigoplus_{b \in B} \langle \mu \Pi_b \rangle$.
\end{thmlist}
\end{lemma}
\begin{proof}
    Suppose for a contradiction that (a) holds but (b) does not, that is to say $V(f,P,\alpha)^\circ = 0$ and there exist linearly independent vectors $v$ and $v'$ in $V(f,P,\alpha)\Pi_b$. Then $v\Pi_b = c e_b$ and $v' \Pi_b = c'e_b$ for some $c,c' \in \mathbb{R}$. So either $c'=c = 0$, in which case $v$ is a non-zero vector in $V(f,P,\alpha)$, or $c'v - cv'$ is a non-zero vector in $V(f,P,\alpha)$, a contradiction. Thus (a) implies (b). 
    
    To show that (b) implies (c), note $\mu \in V(f,P,\alpha)$, so $V(f,P,\alpha) \supseteq \bigoplus_{b \in B} \langle \mu \Pi_b\rangle$. Now for each lump $f^{-1}(b)$ that is accessible from some point in the support of $\alpha$, we have a one-dimensional space $V(f,P,\alpha)\Pi_b$, which must be $\langle \mu\Pi_b \rangle$, and for each inaccessible lump we have $V(f,P,\alpha)\Pi_b = \langle \mu\Pi_b \rangle = 0$.

    To show that (c) implies (a), let $v=\sum_{b\in B} c_b \mu\Pi_b$ for some coefficients $(c_b)_{b\in B}$ and suppose that $v \in \ker F$. Then 
    \[
        0
        =
        vF
        =
        \sum_{b\in B} c_b \mu\Pi_b F 
        = 
        \sum_{b\in B} c_b
        \Big(
        \sum_{a\in f^{-1}(b)}{\mu(a)}
        \Big)
        e_b
        . 
    \]
    Hence for each $b\in B$, we have \smash{$c_b\big(\sum_{a\in f^{-1}(b)}{\mu(a)}\big)=0$}, and 
    so we have either  \smash{$\sum_{a\in f^{-1}(b)}{\mu(a)}=0$} or $c_b=0$. In the former case, $\mu(a)=0$ for every $a\in f^{-1}(b)$ because $\mu$ is a non-negative vector. Hence $v = \sum_{b \in B} c_b \mu\Pi_b = 0$, as required.
\end{proof}
\begin{corollary}\label{cor:exactLumping}
Let $P$ be an irreducible transition matrix and $\mu$  its unique stationary probability distribution. Then the following are equivalent:
\begin{thmlistES}
\item[$\bullet$] $\MC(\alpha,P)$ lumps exactly under $f$, 
\item[$\bullet$] $\MC(\mu,P)$ lumps exactly under $f$ and $\alpha \Pi_b \propto \mu \Pi_b$ for each $b \in B$,
\item[$\bullet$] $\alpha$ belongs to the linear span $\langle \sum_{x \in f^{-1}(b)} \mu(x) e_x : b \in B\rangle$ and this span is preserved by right-multiplication by $P$.
\end{thmlistES}
\end{corollary}

\begin{proof} 
This is immediate from Lemma~\ref{lem:EquivalentConditionsForExactLumping}.
\end{proof}

Thus in the irreducible case, exact lumping means that if the Markov chain is started in its stationary distribution, then at each later time $t$, conditional on the history of lumps $f(X_0) = b_0, \dots, f(X_t) = b_t$, the conditional distribution of $X_t$ on its lump $f^{-1}(\{b_t\})$ is proportional to the restriction of the stationary distribution to that lump.

It appears that the term `exact lumping' was coined in 1984 by Schweitzer \cite{Schweitzer} for the special case where $\alpha$ is the uniform distribution on $A$ and $\alpha$ is stationary. It is used in this sense in several later papers, for example \cite{Buchholz, MarinRossi}. Our definition of exact lumping extends this to more general stationary distributions, and does not require the uniform distribution to be stationary. This more general notion already appeared in 1976 in Kemeny and Snell \cite[Thm.~6.4.4]{KS} (without the name `exact') as a sufficient condition for weak lumping in the case where $P$ is irreducible and aperiodic and $\alpha$ is the unique stationary distribution of $P$. Our definition does not require irreducibility or aperiodicity.  

Exact lumping is very closely related to the well-known Pitman--Rogers condition. Suppose $\MC(\alpha,P)$ lumps exactly under $f$. Since $V(f,P,\alpha)$ is spanned by vectors with non-negative entries, there exists a sub-stochastic matrix $U$ with rows indexed by $B$ and columns indexed by $A$ such that the row of $U$ indexed by $b$ is the unique probability vector in $V(f,P,\alpha)\Pi_b$, if $\dim(V(f,P,\alpha)\Pi_b) = 1$, and is $0$ if $\dim(V(f,P,\alpha)\Pi_b) = 0$. Let $Q$ be the transition matrix of the lumped process $f(\MC(\alpha,P))$. Then we have
\begin{equation}
\label{eq: Pitman-Rogers}
UP = QU.
\end{equation}
This equation is an \emph{algebraic intertwining}, and it is often called the Pitman--Rogers condition, after \cite{PR}. The matrix $U$ is often referred to as the \emph{link} matrix. Note that~\eqref{eq: Pitman-Rogers} does not explicitly mention the initial distribution.  

Conversely, suppose that $P$ is a stochastic matrix from $A$ to $A$ and that~\eqref{eq: Pitman-Rogers} holds for some matrices $U$ and $Q$, where $Q$ is a stochastic matrix from $B$ to $B$ and $U$ is a matrix from $B$ to $A$ whose rows are either $0$ or probability vectors, with the property that $U_{b,a} \neq 0 \implies f(a) = b$. Let $\alpha$ be any probability vector that is a convex combination of the non-zero rows of $U$. (Note this requires at least one row of $U$ to be non-zero, ruling out the trivial solution $U = 0$ of~\eqref{eq: Pitman-Rogers}.) Then $V(f,P,\alpha)$ is a subspace of the row span of $U$ 
and so $\MC(\alpha,P)$ lumps exactly under~$f$. Moreover, $Q$ serves as a  transition matrix for the lumped process $f(\MC(\alpha,P))$. Thus the Pitman--Rogers condition implies exact lumping for suitable initial distributions. In the case where $P$ is irreducible, the link matrix $U$ is necessarily stochastic. In fact Pitman and Rogers \cite{PR} 
introduced their intertwining equation as a sufficient condition for weak lumping in the context of continuous time Markov chains; the theory in that case is very similar, with the transition matrices $P$ and $Q$ replaced by generator matrices.

\begin{example}\label{Ex:exactLumpingToLeftCosets}
In the context of a left-invariant random walk on a group~$G$ driven by an irreducible weight $w$, lumping to $G/H$, the only stationary distribution is the uniform distribution~$\eta_G$, and so exact lumping may be defined in terms of $V = \langle b\eta_H: bH \in G/H\rangle$. We have $V^\circ = 0$ and $V = \bigoplus_{bH \in G/H} V \Pi_b$ so the only non-trivial condition to check is  $VP \subseteq P$. This corresponds to the case $e = \eta_H$ in Theorem~\ref{thm:mainGL}. See Corollary~\ref{cor:strongExact} for a simpler characterisation of exact lumping in this case.
\end{example}

\subsection{The reversible case}
A stationary Markov chain $X = \MC(\alpha,P)$ with state space $A$ is said to be \emph{reversible} if for all $x,y \in A$ we have $$\alpha(x)P(x,y) = \alpha(y)P(y,x).$$ This implies that for every $n \ge 0$ the sequence $(X_0, \dots, X_n)$ has the same joint distribution as its time reversal $(X_n, \dots, X_0)$. It also means that the chain can be extended to be indexed by times in $\mathbb{Z}$, and this extended chain is equal in law to its own time reversal. The theory of weak lumping simplifies greatly in the reversible case:
\begin{theorem}[Burke and Rosenblatt (1958) {\cite[Thm. 1] {BR}}]\label{thm:BurkeRosenblatt}
 Let $P$ be a stochastic matrix from $A$ to $A$ and let $\alpha$ be a stationary distribution for $P$ such that $\alpha(x) > 0$ for all $x \in A$. Suppose that $\MC(\alpha,P)$ is reversible and lumps weakly under $f: A \to B$.  Then $f$ is a strong lumping of $\MC(\alpha,P)$.
\end{theorem}
A close inspection of the proof given in \cite{BR} shows that the assumption of full support may be weakened to assuming that $\alpha$ assigns positive weight to each lump of at least two elements. 

In \S\ref{sec: time reversal} we will prove a new duality result, Theorem~\ref{thm:DualityForMarkovChainLumping}, which  relates the weak lumping of a stationary finite Markov chain to the weak lumping of its time reversal, under the mild assumption that the stationary distribution has full support. A special case (Corollary~\ref{cor:ExactStrongDuality}) is that for such stationary Markov chains,  time reversal exchanges strong lumping and exact lumping. 
Applying this to the reversible case, we obtain the following corollary of Theorem~\ref{thm:BurkeRosenblatt}.

\begin{corollary}\label{cor:ReversibleLumpingsAreExact}
Under the conditions of Theorem~\ref{thm:BurkeRosenblatt}, $f$ is an exact lumping of $\MC(\alpha,P)$.
\end{corollary}

\subsection{Probabilistic consequences of strong and exact lumping} 
We finish with two results that are not logically required but serve to illuminate the definitions of strong
and exact lumping.
We maintain our usual notation in which $f :A \rightarrow B$ is a surjective function.

\begin{proposition}
\label{prop:condindepstrong}
    Suppose that $X=\MC(\alpha,P)$ lumps strongly under $f$. 
    Then for all $t\geq0$, $X_t$ and $(f(X_{t+1}),f(X_{t+2}),\dots)$ are conditionally independent given $f(X_t)$.
\end{proposition}
\begin{proof} 
For every $y,b \in B$, let $Q(y,b)$ denote the common value of 
\[ \sum_{a \in f^{-1}(b)} P(x,a) \] 
over all $x \in f^{-1}(y)$.
The conditional distribution of $(f(X_{t+1}), f(X_{t+2}), \dots)$ given $f(X_t)$ is determined by its finite-dimensional marginals. From Dynkin's condition, it is easy to show by induction over $n$ that, 
for each $t \ge 0$, each $x_t \in A$ such that $\P[X_t = x_t] > 0$ and each $n \ge 1$, we have
\[
\P[f(X_{t+1})=y_{t+1},\dots,f(X_{t+n})=y_{t+n} 
        \mid 
         X_t=x_t] = \prod_{i=1}^n Q(y_{t+i-1},y_{t+i}), 
\]         
where $y_t = f(x_t)$. This suffices to demonstrate the conditional independence because the right-hand side 
depends on $x_t$ only through $f(x_t)$.
\end{proof}

\begin{proposition}
\label{prop:condindepexact}
    Suppose that $X=\MC(\alpha,P)$ lumps exactly under $f$. Then for all $t\geq0$, $X_t$ and $(f(X_0),\dots,f(X_{t-1}))$ are conditionally independent given $f(X_t)$.
\end{proposition}
\begin{proof}
 Let $V = V(f,P,\alpha)$. Let $b_0, \dots, b_t$ be any sequence of lumps such that $\P[X_0 = b_0, \dots, X_t - b_t] > 0$, and let $\beta$ be the conditional distribution of  $X_t$ given $f(X_0) = b_0, \dots, f(X_t) = b_t$. From the construction of $V$ we have $\beta \in V$. Since $\beta$ is supported on $f^{-1}(b_t)$, we have $\beta \in V\Pi_{b_t}$. Since $\dim(V\Pi_b) \le 1$ by Lemma~\ref{lem:EquivalentConditionsForExactLumping}, $\beta$ is the unique probability vector in $V\Pi_{b_t}$. In particular, $\beta$ is a function of $b_t$ alone, which yields the claimed conditional independence.
\end{proof}

\section{Background from character theory}\label{sec:algebraicPreliminaries}
We refer the reader to the textbooks
of Benson \cite{BensonI}, Dornhoff \cite{Dornhoff} and Isaacs \cite{Isaacs} 
for further background on representation
theory and characters; we use the latter two as our basic references in this section. See also \cite{JamesLiebeck}
for an excellent introduction. 
Recall that throughout $G$ is a fixed finite group. 

\subsection{Group algebras}\label{subsec:groupAlgebra}
The group algebra $\C[G]$ is, by definition, the $|G|$-dimensional vector space 
of all formal linear combinations $\sum_{g \in G} \beta(g) g$ of the group elements,
with coefficients $\beta(g) \in \C$. The multiplication is defined by bilinear extension
of the group multiplication: thus
\[ \bigl( \sum_{x \in G} \beta(x) x \bigr) \bigl( \sum_{y \in G} \gamma(y) y \bigr) =
\sum_{g \in G} \bigl( \sum_{x \in G} \beta(x) \gamma(x^{-1}g) \bigr) g. \]
We identify weights on $G$ with non-zero elements of $\C[G]$ having non-negative
real coefficients and 
probability measures on $G$ with non-negative elements of $\C[G]$ whose coefficient
sum is~$1$. We say that a weight with this property is \emph{normalized}.
As the following lemma shows, this makes $\C[G]$
the natural setting for computing with random walks on $G$.

\begin{lemma}\label{lemma:stepIsMultiplication}
In the left-invariant random walk on $G$ driven by a normalized weight $w$, if $X_t \sim \alpha$
then $X_{t+1} \sim \alpha w$.
\end{lemma}

\begin{proof}
By definition of the walk
\[ \P[X_{t+1} = y] = \sum_{x \in G} \P[X_{t+1} = y | X_t = x] \alpha(x) 
= \sum_{x \in G} \alpha(x) w(x^{-1}y) = (\alpha w)(y) \]
as required.
\end{proof}

It follows that if $w$ is a normalized weight and $X_0 \sim \alpha$ then $X_t \sim \alpha w^t$
for each $t \ge 0$. 

\subsection{Representations, modules and characters}
Let $V$ be a finite-dim\-ensional $\C$-vector space. A \emph{representation} of $G$ is a homomorphism
$\rho : G \rightarrow \GL(V)$ from $G$ into the general linear group of invertible
linear maps on $V$. We may abuse notation and refer to this representation as $V$.
The group algebra $\C[G]$ then acts on~$V$ on the left by
\[ \Big( \sum_{g \in G} w(g) g\Big) v = \sum_{g \in G} w(g) \rho(g) v. \]
Thus $V$ becomes a left $\C[G]$-module. (See \cite[\S 1]{Dornhoff} \cite[Definition 1.3]{Isaacs}
for the definition of modules for an algebra.) Conversely, 
given a left $\C[G]$-module~$V$, there is a corresponding representation $\rho : G \rightarrow \GL(V)$ defined by letting $\rho(g)$ be the linear transformation by which $g$ acts on $V$.
It will often be convenient to pass between these two equivalent languages.
The character of a representation $\rho : G \rightarrow \GL(V)$, or the corresponding $\C[G]$-module,
is the function $\chi_V : G \rightarrow \C$
defined by $\chi_V(g) = \tr \rho(g)$, where $\tr$ denotes trace.
Note that $\chi_V(1) = \dim V$.

\begin{example}\label{ex:basicRepresentations}
The
\emph{trivial representation} of $G$ defined by $\rho(g) = (1) \in \GL_1(\C)$ for each $g \in G$
has character $\triv_G$, defined by $\triv_G(g) = 1$ for each $g\in G$.
The \emph{regular representation of $G$} is the representation afforded by the left
action of $G$ on $\C[G]$. Thus in the canonical basis of $\C[G]$ of group elements,
each $g \in G$ acts as a $|G| \times |G|$-permutation matrix.
It has character~$\phi_G$ defined by $\phi_G(1) = |G|$ and
$\phi_G(g)  = 0$ for each non-identity $g \in G$.
\end{example}

The regular representation is the case $\Omega = G$ of the following construction.

\begin{example}[Permutation representations]\label{ex:permutationRepresentations}
Suppose that $G$ acts on the left on a set~$\Omega$.
The \emph{permutation representation of $G$ on $\Omega$}, denoted $\C[\Omega]$, has underlying vector space
with canonical basis $\{ v_\omega : \omega \in \Omega \}$ and action defined by
$\rho(g) v_\omega = v_{g\omega}$. If $G$ is a subgroup of $\Sym_\Omega$ then we say that $\C[\Omega]$
is the \emph{natural} representation of $G$.
\end{example}

Since we work with left modules, in the action of $\Sym_\Omega$
on $\Omega$ we compose permutations from right to left; thus $(gh)(\omega) = g(h(\omega))$
for $g, h \in \Sym_\Omega$. This convention is in force for the examples in this section only.

\subsubsection*{Irreducible representations}
A \emph{subrepresentation} of a representation $\rho : G \rightarrow \GL(V)$ is 
a subspace $W$ of $V$
such that $\rho(g) W \subseteq W$ for each $g \in G$.
A representation $V$ is \emph{irreducible} if it does not contain a
non-trivial proper subrepresentation.
Representations of finite groups over $\C$ are completely reducible, 
meaning that one can always write $V = \bigoplus_{W} W$ where each subrepresentation $W$
is irreducible. (See \cite[Theorem 3.1]{Dornhoff} or \cite[Definition~1.7, Theorem~1.9]{Isaacs}.) 
We write $\Irr(G)$ for the finite set of irreducible representations of $G$
up to isomorphism, and also for the set of their characters; this creates no ambiguity in practice.

We now begin a running example using the symmetric group $\Sym_3$.
By the general theory (see \cite[\S 4]{James}), the irreducible
representations of $\Sym_n$ are canonically labelled by the partitions of $n$.
Here we give an \emph{ad hoc} construction.

\begin{example}[Irreducible representations of $\Sym_3$] \label{eg: irreps of Sym3}
The natural representation of the symmetric group $\Sym_3$ on $\langle v_1, v_2, v_3 \rangle$
decomposes as $$\langle v_1 + v_2 + v_3 \rangle \oplus \langle v_1 - v_3, v_2-v_3 \rangle.$$
The first summand affords the trivial representation of $\Sym_3$ and the second an irreducible
$2$-dimensional representation.
The $1$-dimensional sign representation, in which $g \in G$ acts as $\sgn(g) \in \{+1,-1\}$
is the remaining irreducible representation of $\Sym_3$.
The corresponding modules are $S^{3}$, $S^{21}$ and $S^{111}$, respectively;
we have $S^{3} \cong \C$ (the trivial representation) and $S^{111} \cong \sgn$ (the sign representation).
\end{example}

By definition $\C[G]$-modules $U$ and $W$ are \emph{isomorphic}
if there is an invertible linear map $T : U \rightarrow W$ such that $g T(u) = T(g u)$ for all $u \in U$
and $g \in G$. This holds if and only if $U$ and $V$ have the same character. 
The character of $U \oplus W$ is $\chi_U + \chi_W$. 
The number of times an irreducible representation $U$ appears as a summand in a
direct sum decomposition of $V$ is $\langle \chi_V, \chi_U\rangle$, where the 
inner product on characters of $G$ is
defined by~\eqref{eq:innerProduct}. Thus
\begin{equation}\label{eq:characterInnerProduct}
\langle \chi_V, \chi _W \rangle = \frac{1}{|G|}\sum_{g\in G}  \overline{\chi_V(g)}\chi_W(g).
\end{equation}

\begin{example}[Regular representation of $\Sym_3$]\label{eg: regular rep of S3}
The characters of $\Sym_3$ of the irreducible $\C[\Sym_3]$-modules constructed
in Example~\ref{eg: irreps of Sym3} are
$\chi_{S^3}(g) = 1$, 
$\chi_{S^{21}}(g) = |\mathrm{Fix}(g)| - 1$ and $\chi_\mathrm{sgn}(g) = \sgn(g)$
for each $g \in \Sym_3$. Note $\chi_{S^3} = \triv_{\Sym_3}$.
The regular character of $\Sym_3$ 
decomposes as
\[
\phi_{\Sym_3} = \chi_{S^3} + 2 \chi_{S^{21}} + \chi_{S^{111}}.
\]
This is generalized by the Wedderburn decomposition seen in \S\ref{subsec:WedderburnDecomposition} following.
\end{example}

As already seen from the regular representation, $\C[G]$ is itself a left $\C[G]$-module.
Moreover, a left $\C[G]$-submodule of $\C[G]$ is simply a left ideal in $\C[G]$. Indeed, both are defined as subspaces $L\subseteq\C[G]$ with the property $gL \subseteq L$ for all $g\in G$.

\subsection{Wedderburn decomposition}\label{subsec:WedderburnDecomposition}
Modules for algebras
are defined by generalizing the constructions already seen for the group algebra $\C[G]$:
see for instance \cite[\S 1]{Dornhoff}. For our purposes, besides group algebras, the only example we need is
$\Mat_d(\C)$, the algebra of $d \times d$ matrices.
By a basic result (see for instance \cite[Theorem 2.18(a)]{Dornhoff}), $\Mat_d(\C)$ has a unique irreducible 
module up to isomorphism, namely the space of column vectors $\C^n$.
Moreover, as a left $\Mat_d(\C)$-module,
\[ \Mat_d(\C) = W_1 \oplus \cdots \oplus W_d \]
where $W_i$ is the left ideal of $\Mat_d(\C)$
of matrices zero except in column $i$. Each $W_i$
is isomorphic as a $\Mat_d(\C)$-module to the irreducible module $\C^n$.
If $V$ is a $d$-dimensional vector space we write $\Mat(V)$ for $\Mat_d(\C)$;
then $V$ is, up to isomorphism, the unique irreducible module for $\Mat(V)$.

\begin{proposition}[Wedderburn decomposition]
\label{prop:Wedderburn}
The group algebra $\C[G]$ admits a decomposition
\[ \C[G] \cong \bigoplus_{V\in\Irr(G)} \isoblock{V} \]
as a left $\C[G]$-module and as an algebra,
where the sum ranges over a set of representatives of the irreducible representations of $G$.
Moreover the isomorphism may be chosen so that if, in one direct sum decomposition of
$\C[G]$ as a left $\C[G]$-module, the $\dim V$\! summands isomorphic to the irreducible $\C[G]$-module
$V$\! are $V_1 \oplus \cdots \oplus V_{\dim V}$, then
for each $i$,
the image of $V_i$ in $\isoblock{V}$ is the left ideal of $\isoblock{V}$
of matrices that are zero except in column $i$.
\end{proposition}

\begin{proof}
See \cite[Theorem 3.2]{Dornhoff}; this is proved using Theorem 2.18 earlier in \cite{Dornhoff},
from which the `moreover' part is clear. \end{proof}

A more condensed proof suitable for experts is given in \cite[Theorem 1.3.4]{BensonI}.
The proposition is also proved in \cite[Theorem 1.15]{Isaacs}, our reference for character theory, 
but it takes some work to deduce the version
stated above from the subsequent remarks.

\begin{example}\label{eg: Wedderburn iso Sym3}
By Proposition~\ref{prop:Wedderburn}, the     
Wedderburn decomposition of $\Sym_3$ into algebra summands is
    \begin{align*}
        \C[\Sym_3] &\cong \Mat_1(\C) \oplus \Mat_2(\C) \oplus \Mat_1(\C)\\
        &\cong \Mat(S^{3}) \oplus \Mat(S^{21}) \oplus \Mat(S^{111})
    \end{align*}
where, by the `moreover' part,
$\Mat(S^{21}) \cong S^{21} \oplus S^{21}$ as a left $\C[\Sym_3]$-module.
    A Weddernburn isomorphism, as in this proposition, therefore identifies $\C[\Sym_3]$ with the algebra of $4\times 4$ block matrices
    \[
    \left(
    \begin{matrix}
        \star & 0 & 0 & 0 \\
        0 & \star&\star & 0\\
        0 & \star&\star & 0\\
        0 & 0 & 0 & \star
    \end{matrix}
    \right).
    \]
    In what comes, we favour the more compact diagrams below, which show a choice of Wedderburn decomposition as in the  `moreover' part.
    \[
    \raisebox{-\mbaseline}{\includegraphics[page=3]{AllPictures.pdf}}.
    \]
\end{example}
\begin{example}\label{eg: Wedderburn iso Sym4}
    The symmetric group $\Sym_4$ has five irreducible representations, labelled by the partitions of $4$. 
    The Wedderburn decomposition is
    \begin{align*}
        \C[\Sym_4] &\cong \Mat_1(\C) \oplus \Mat_3(\C) \oplus \Mat_2(\C) \oplus \Mat_3(\C) \oplus \Mat_1(\C)\\
        &\cong \Mat(S^{4}) \oplus \Mat(S^{31}) \oplus \Mat(S^{22}) \oplus \Mat(S^{211}) \oplus \Mat(S^{1111}).
    \end{align*}
    We can therefore identify $\C[\Sym_4]$ with the algebra of $10\times 10$ block matrices of 
    the form
    \[
    \raisebox{-\mbaselinemed}{\includegraphics[page=4]{AllPictures.pdf}}.
    \]
   We continue this example in Example~\ref{eg: induction Sym3 to Sym4}.
\end{example}

\subsection{Idempotents}\label{subsec:idempotents}
As we stated in the introduction, an element $e \in \C[G]$ is \emph{idempotent} if $e^2 = e$.
It is a basic result that
if $L$ is a left ideal in $\C[G]$ (or equivalently,
a left $\C[G]$-submodule of $\C[G]$) 
there exists an idempotent
$e \in L$ such that $L = \C[G] e$.
Such an idempotent can be constructed by choosing a linear
projection $\pi : \C[G] \rightarrow L$ and then taking its `average' 
\smash{$\overline{\pi} = \frac{1}{|G|} \sum_{g \in G} g^{-1} \pi g$}. It is routine to check that $\overline{\pi}$
is also a projection onto~$L$, 
that~\smash{$\overline{\pi}$} commutes with the action of $G$,
and hence that the map $\overline{\pi}$ agrees with $x \mapsto x e$ where
\smash{$e = \frac{1}{|G|} \sum_{g \in G} g^{-1} \pi(g) \in L$}
is the image  of $\id_G$ under $\overline{\pi}$. Thus $e$ is a suitable idempotent. 
We remark that $e$ is not in general unique: 
see Example~\ref{ex:Sym3PrimitiveIdempotents}; in fact
$e$ is unique if and only if it is a sum of distinct centrally primitive idempotents
in the sense defined below.

\begin{example}\label{ex:trivialIdempotent}
    The trivial representation of $G$ 
    is isomorphic to the
    left ideal of $\C[G]$ spanned by the idempotent $\eta_G$, defined earlier to be $|G|^{-1} \sum_{g\in G} g$.
\end{example}

\begin{lemma}[Idempotents versus characters] \label{lem:IdempotentsVSCharacters}
    Let $e\in \C[G]$ be an idempotent and let $W$ be a $\C[G]$-module.
    Then $\dim eW =  \langle \chi_{\C[G]e}, \chi_W\rangle$.
\end{lemma}

\begin{proof} 
By complete reducibility 
we may assume that $\C[G] e$ is isomorphic to the
irreducible $\C[G]$-module $U$. It follows easily from the Wedderburn decomposition that
$e U$ is one-dimensional, and $e V = 0$ if $V$ is an irreducible $\C[G]$-module
not isomorphic to $U$. Therefore
$\dim eW$ is the number of simple modules isomorphic to $U$ in a direct
sum decomposition of $W$ into simple modules; this is the right-hand side.
\end{proof}

By  \cite[Theorem~2.12]{Isaacs} the \emph{centrally primitive idempotent} for an 
irreducible $\C[G]$ module $V$ with
character $\chi$ is
\begin{equation}\label{eq:centralPrimitiveIdempotent}
e_\chi = \frac{|\chi(1)|}{|G|} \sum_{g \in G} \chi(g^{-1}) g.
\end{equation}
(The case where $V$ is the trivial module was seen in Example~\ref{ex:trivialIdempotent}.)
The image of $e_\chi$ in the Wedderburn decomposition is zero except in
the block $\isoblock{V}$, where it is the identity matrix. 
It follows easily that 
$e_\chi L = L e_\chi =  L\cap\isoblock{V}$ for any left ideal $L$ of $\C[G]$.
In general a matrix
 $e\in\Mat_d(\C)$ is an idempotent if and only if it is conjugate to a matrix of the form $\diag(1,...,1,0,...,0)$.

\begin{example}[Centrally primitive idempotents of $\Sym_3$]\label{ex:Sym3CentralIdempotents}
 The centrally primitive idempotents for $\Sym_3$ are
$e_3 = \frac{1}{6} \sum_{g \in \Sym_3}(g)$, $e_{21} = \frac{2}{3} - \frac{1}{3}(1,2,3) - \frac{1}{3}(1,3,2)$
and $e_{111} = \frac{1}{6} \sum_{g \in \Sym_3} \sgn(g) g$, where $(1,2,3)$ and $(1,3,2)$ are the two $3$-cycles
in $\Sym_3$.
In any chosen Wedderburn isomorphism they correspond to the identity matrices in the relevant 
matrix blocks:

\vspace*{-12pt}
\[
\includegraphics[page=5]{AllPictures.pdf}
    \]
\end{example}

\vspace*{-6pt}
A not necessarily central idempotent $e$ is \emph{primitive} if it cannot be expressed
as a sum $f + f'$ with $f$ and $f'$ idempotents and $ff' = ff' = 0$. 
In the previous example, the centrally primitive idempotents $e_3$ and $e_{111}$ are primitive, since the left ideals that they generate are one-dimensional. Under the Wedderburn decomposition we have

\vspace*{-12pt}
\[ 
e_{21} \longmapsto \raisebox{-\mbaseline}{\includegraphics[page=6]{AllPictures.pdf}}
\]

\vspace*{-6pt}\noindent
which shows that $e_{21}$ is \emph{not} primitive. We give an explicit decomposition of $e_{21}$
working in $\C[\Sym_3]$
in Example~\ref{ex:Sym3PrimitiveIdempotents} below.

\subsection{Restriction and induction}\label{subsec:restrictionAndInduction}
Fix throughout a subgroup $H$ of~$G$. We shall often use the restriction
and induction functors relating $\C[G]$-modules to $\C[H]$-modules
as defined for modules and characters in \smash{\cite[\S 9]{Dornhoff}} and for characters
in \cite[Ch.~5]{Isaacs}. 

\begin{samepage}
\begin{definition}[Restricted and induced modules]\label{defn:restrictedAndInducedModules}{\ }
\begin{defnlist}
\item 
The \emph{restriction of a $\C[G]$-module $W$ to $H$}, denoted $W\res^G_H$, is the $\C[H]$-module
with the same underlying vector space as $W$, but the action defined only on $H$.
\item The \emph{induction of a $\C[H]$-module $U$ to $G$}, denoted $U \ind_H^G$, is
defined to be $\C[G] \otimes_{\C[H]} U$ where
$\C[G]$ is regarded as a right $\C[H]$-module by right multiplication.
\end{defnlist}
\end{definition}
\end{samepage}

In (b), the space $\C[G] \otimes_{\C[H]} U$ may be defined as the quotient of $\C[G] \otimes U$
by the subspace spanned by all $gh \otimes u - g \otimes hu$ for $x \in G$, $h \in H$ and $u \in U$.
Thus the relation $gh \otimes u = x \otimes hu$ holds in $\C[G] \otimes_{\C[H]} U$. This vector space
is a $\C[G]$-module with action defined by linear extension
of $k (g \otimes u) = kg \otimes u$ for $k \in G$.
In many cases one can avoid thinking about the technical construction using
tensor products and instead employ the following characterisation.

\begin{proposition}[Characterisation of induced modules]\label{prop:inductionCharacterisation}
Let $U$ be a $\C[H]$-module. The following are equivalent for a $\C[G]$-module~$V$:
\begin{thmlist}
\item $V \cong U \ind_H^G$;
\item $V$ has a $F[H]$-submodule $X$ isomorphic to $U$ such that $X$ generates~$V$ as a $\C[G]$-module
and $\dim V = [G : H]\dim X$; 
\item $V$ has a $F[H]$-submodule $X$ isomorphic to $U$ and there is a vector space decomposition 
$V = \bigoplus_{g \in G/H} gX$.
\end{thmlist}
\end{proposition}

\begin{proof}
For the equivalence of (i) and (ii), see \cite[page 56, Corollary 3]{Alperin}.
By dimension counting one sees that that (ii) implies (iii) and the converse is obvious.
\end{proof}

\begin{example}\label{ex:inductionCharacterisation}
Let $\Omega = G/H$. The permutation module $\C[\Omega]$ of $G$ acting on the cosets of $H$, as defined
in Example~\ref{ex:permutationRepresentations}, has as a canonical basis $\{v_{bH} : b \in G/H \}$.
Observe that $X = \langle v_H \rangle$ affords the trivial representation $\C$ of $H$ and that
$\dim \C[\Omega] = |G/H|$. Therefore by Proposition~\ref{prop:inductionCharacterisation}
we have $\C[\Omega] \cong \C\ind_H^G$. 
\end{example}

In particular, if $U$ is a left ideal of $\C[H]$ then 
by Proposition~\ref{prop:inductionCharacterisation}(iii) the left ideal of $\C[G]$ generated by $U$,
namely \smash{$\bigoplus_{g \in G/H} gU$}, is isomorphic to \smash{$U \ind_H^G$}.
Such ideals are ubiquitous in this work. Given a left coset $bH \in G/H$,
Let $\pi_{bH} : \C[G] \rightarrow \C[H]$ denote the projection map
defined by 
\[ \pi_{bH} \bigl( \sum_{g \in G} x(g) g \bigr) = \sum_{g \in bH} x(g) g.\]

\begin{definition}\label{defn:inducedIdeal}
We say that an ideal $L$ of $\C[G]$ containing $\pi_H(L)$
is an \emph{induced ideal} from~$H$ to $G$.
\end{definition}

We typically omit `from $H$ to $G$' as it is clear from context.

\begin{example}\label{ex:permutationIdeal}
Observe that $\langle \eta_H \rangle$ is a left ideal of $\C[H]$ affording 
the trivial representation of $H$.
Therefore, by the remark before Definition~\ref{defn:inducedIdeal}, 
\[ \langle \eta_H \rangle \Ind_H^G \cong \C[G]\eta_H
 = \langle b \eta_H : b \in G /H \rangle.\]
Working directly from Definition~\ref{defn:restrictedAndInducedModules}(ii)
one could instead show that $g \otimes \eta_H \rightarrow g \eta_H$ defines a $\C[G]$-isomorphism
$\C[G] \otimes_{\C[H]} \langle \eta_H \rangle \cong \C[G] \eta_H$; this is a routine, but somewhat
technical, check.
\end{example}

The following proposition justifies the name `induced ideal' and generalizes the features
seen in the previous example.

\begin{proposition}\label{prop:inducedIdealCharacterisation}
Let $L$ be a left ideal of $\C[G]$ containing $\pi_H(L)$. 
Setting $U = \pi_H(L)$, we have
\begin{thmlist}
\item $U$ is a left ideal of $\C[H]$;
\item $L = \C[G] U$;
\item $L = \bigoplus_{b \in G/H} bU$;
\item there is an isomorphism of left $\C[G]$-modules \smash{$L \cong U\ind_H^G$};
\item there exists an idempotent $e \in \C[H]$ such that $L = \C[G]e$.
\end{thmlist}
\end{proposition}

\begin{proof}
Since $L$ is a left ideal of $\C[G]$ we have $\C[H] U \subseteq L$ 
and since $\C[H]$ is closed under multiplication by elements of $H$, we have $\C[H]U \subseteq \C[H]$.
Therefore $\C[H]U \subseteq L \cap \C[H] = U$, proving~(i). 
Since $L$ is a left ideal of $\C[G]$,
$L$ contains $\bigoplus_{b \in G/H} bU$, and since 
\[ L \cap b \hskip1pt \C[H] \subseteq b (L \cap \C[H]) = bU, \]
it follows that
$L = \bigoplus_{b \in G/H} bU$. This proves (ii) and (iii); now (iv) follows from Proposition~\ref{prop:inductionCharacterisation}(iii). Finally by 
the remark at the start of \S\ref{subsec:idempotents}, applied to the left ideal
$U$ of $\C[H]$, there exists an idempotent $e \in \C[H]$ such that $U = \C[H] e$. It now
follows easily from (iii) that $L = \C[G]e$.
\end{proof}

\begin{example}[Primitive idempotents of $\Sym_3$]\label{ex:Sym3PrimitiveIdempotents}
We remarked after
Example~\ref{ex:Sym3CentralIdempotents}
that the centrally primitive idempotents $e_3 = \eta_{\Sym_3}$ and $e_{111} =
\frac{1}{6} \sum_{g \in \Sym_3} \sgn(g) g$
are primitive. Let $H = \Sym_{\{2,3\}}$. The natural action of $\Sym_3$ on $\{1,2,3\}$ corresponds to 
the action of $\Sym_3$ on the left cosets of~$H$.
(More formally, the map $g H \mapsto g(1)$ is a permutation isomorphism.)
By Example~\ref{ex:permutationIdeal},
the natural permutation representation $\langle v_1, v_2,v_3 \rangle$ for $\Sym_3$
(seen earlier in Example~\ref{eg: irreps of Sym3}) is isomorphic to the left ideal $\C[G] \eta_H$,
by an isomorphism satisfying $v_1 \mapsto \eta_H$. 
We saw earlier that $\langle v_1, v_2,v_3 \rangle \cong S^3 \oplus S^{21}$ where $S^3$
is the trivial module. Therefore subtracting $\eta_{\Sym_3}$ from $\eta_H$
we obtain
\[ f = \eta_H \!-\! \eta_{\Sym_3}
=
 \mfrac{1}{3} \!\id_{\Sym_3} \!-\! \ \mfrac{1}{6} (1,2,3) -\! \ \mfrac{1}{6} (1,3,2) 
-\! \ \mfrac{1}{6} (1,2) -\! \ \mfrac{1}{6} (1,3) +\! \mfrac{1}{3}
(2,3) . \]
Since $\eta_{\Sym_3}$ is central, $f$ is an idempotent such that $\C[\Sym_3]f \cong S^{21}$.
Taking $f' = e_{21} - f$ gives an explicit decomposition of the central
primitive idempotent $e_{21}$ into primitive idempotents, as promised
after Example~\ref{ex:Sym3CentralIdempotents}. 
It is worth noting that there is nothing canonical about this decomposition:
indeed $f$ may be replaced with any conjugate $u fu^{-1}$ for a unit $u \in \C[G]$.
The images of $\eta_H$, $\eta_G$ and $f = \eta_H - \eta_G$
are shown below,
with respect to the Wedderburn decomposition 
$\C[\Sym_3] = \C[\Sym_3] \eta_G + \bigl( \C[\Sym_3]f + \C[\Sym_3]f' \bigr) + \C[\Sym_3]e_{111}$:
 \[       
 \includegraphics[page=7]{AllPictures.pdf}
    \]
\end{example}

Definition~\ref{defn:restrictedAndInducedModules} extends
to characters and representations in the obvious way:
let $\chi_W\res^G_H$ be the character of~$W\res^G_H$
and let $\chi_U \ind_H^G$ be the character of~$U\ind_H^G$.

\begin{example}
    \label{eg: induction Sym3 to Sym4}
    Let $H = \Sym_3$ and $G = \Sym_4$. Refer to Examples \ref{eg: irreps of Sym3}, \ref{eg: Wedderburn iso Sym3} and \ref{eg: Wedderburn iso Sym4}. We have
    
    \smallskip
    \begin{defnlist}
        \item[(1)] $\chi^{3}\ind_{\Sym_3}^{\Sym_4} = \chi^4 + \chi^{31},$
        \item[(2)] $\chi^{21}\ind_{\Sym_3}^{\Sym_4} = \chi^{31}+ \chi^{22} + \chi^{211}$, and 
        \item[(3)] $\chi^{111}\ind_{\Sym_3}^{\Sym_4} = \chi^{211}+ \chi^{1111}$. 
    \end{defnlist}
    
    \smallskip\noindent
    In each case the character of $\Sym_4$ is obtained by adding a box to the relevant partition of $3$,
    in all possible ways. (See \cite[Ch.~9]{James} for the general result.) 
    Corresponding to the Wedderburn decomposition 
    \[ \C[\Sym_3] \cong S^3 \oplus S^{21} \oplus S^{21} \oplus S^{111} \]
    from Example~\ref{eg: Wedderburn iso Sym3}, there is a choice of Wedderburn isomorphism for
     $\C[\Sym_4]$ such that
    \[
    \C[\Sym_4] \cong \C[\Sym_3]\Ind_{\Sym_3}^{\Sym_4} \cong
    \raisebox{-16pt}{\includegraphics[page=8]{AllPictures.pdf}}
    \]
    where the four summands are the modules induced from the summands of $\C[\Sym_3]$ displayed above
    and the blocks are ordered from top-left to bottom-right $4$, $31$, $22$, $211$, $1111$.
    This isomorphism is chosen so that tensoring by the sign representation $S^{1111}$ corresponds to rotating diagrams by a half-turn; note that $S^{31} \otimes \sgn \cong S^{211}$ and $S^{22} \otimes \sgn \cong S^{22}$.
        We return to this example in Example~\ref{eg: exact lumping algebra Sym3 Sym4}.
\end{example}

We end this subsection with a fundamental result relating induced and restricted 
modules and characters. Below the subscripts $G$ and $H$ indicate the group relevant to the
character inner product defined in~\eqref{eq:characterInnerProduct}. 

\begin{proposition}[Frobenius reciprocity]\label{prop:FrobeniusReciprocity}
If $U$ is a $\C[H]$-module and $W$ is a $\C[G]$-module then
\[ \bigl\langle \chi_U \Ind_H^G, \chi_W \bigr\rangle_G = 
\bigl\langle \chi_U, \chi_W \Res^G_H \bigr\rangle_H. \]
\end{proposition}

\begin{proof}
See \cite[Theorem 9.4(c)]{Dornhoff}, \cite[Lemma 5.2]{Isaacs} for proofs with minimal prerequisites,
or, for an 
elegant and conceptual proof using the tensor-hom adjunction,
\cite[Proposition 2.8.3]{BensonI}.
\end{proof}

\subsection{Borel and parabolic subalgebras}
We continue with some basic results on subalgebras of $\Mat_d(\C)$
needed in the proof of Proposition~\ref{prop: irredundant}. Proofs are included
to make the article self-contained and to introduce some ideas relevant to the proof of this proposition.
Recall that a subalgebra $P$ of $\Mat_d(\C)$ is \emph{parabolic} 
if there is a chain of subspaces
$\C^d \supset V_1 \subset \ldots \supset V_r \supset 0$
such that $P = \{ T \in \Mat_d(\C) : T(V_i) \subseteq V_i \text{ for $1 \le i \le r$} \}$.
In this case we write
\[ P = \Stab( \C^d \supset V_1 \supset \ldots \supset V_r \supset 0 ). \]
Observe that if $\MX$ is an invertible matrix then
\begin{equation}
\label{eq:parabolicConjugacy}
\MX P \MX^{-1} = \Stab\bigl( \C^d \supset \MX(V_1) \supset \ldots \supset \MX(V_r) \supset 0 \bigr) .
\end{equation}
Stated slightly informally, the following lemma is that parabolic subalgebras are self-normalizing.

\begin{lemma}\label{lemma:parabolicSelfNormalizing}
Let $P$ be a parabolic subalgebra of $\Mat_d(\C)$. Let $\MX$ be an invertible matrix
in $\Mat_d(\C)$. If $\MX P \MX^{-1} = P$ then $\MX \in P$.
\end{lemma}

\begin{proof}
Let $P = \Stab( \C^d \supset V_1 \supset \ldots \supset V_r \supset 0)$ be a parabolic subalgebra
of $\Mat_d(\C)$.
Observe that $V_1$ is the greatest proper subspace of $\C^d$ preserved by $P$, and inductively,
$V_i$ is the greatest proper subspace of $V_{i-1}$ preserved by $P$, for each $i$.
Thus $P$ determines the chain of subspaces $\C^d \supset V_1 \supset \ldots \supset V_r \supset 0$.
It now follows from~\eqref{eq:parabolicConjugacy} that
$\MX P\MX ^{-1} = P$ if and only if $\MX (V_i) = V_i$ for each $i$, or equivalently,
if and only if $\MX$ is an invertible matrix in $P$. The lemma follows.
\end{proof}

We define the \emph{standard Borel subalgebra} of $\Mat_d(\C)$ 
to be its subalgebra of invertible lower triangular matrices. 
Equivalently, if $C_i$ is the subspace of~$\C^d$ of column vectors zero in their top $d - i$ positions,
then the standard Borel subalgebra is
$\Stab(\C^d \supset C_{d-1} \supset \ldots \supset C_{1} \supset 0 )$. 
We say that a subalgebra of $\Mat_d(\C)$ is \emph{Borel} if it is conjugate 
by an invertible matrix to the standard Borel. 
We say that a parabolic subalgebra of $\Mat_d(\C)$ is \emph{standard}
if it contains the standard Borel subalgebra.

\begin{lemma}\label{lemma:standardParabolic}
A parabolic subalgebra is standard if and only if it is equal to 
$\Stab(\C^d \supset C_{d_1} \supset \ldots \supset C_{d_r} \supset 0)$
for some $d > d_1 > \ldots > d_r > 0$.
\end{lemma}

\begin{proof}
The `if' direction is clear from the equivalent definition of the standard Borel
algebra just given. For the `only if' direction, let $P = \Stab( \C^d \supset
V_1 \supset \ldots \supset V_r \supset 0)$ be a standard parabolic
subalgebra. Let $B$ be its subgroup
of invertible lower triangular matrices. Observe that the orbits of $B$ on
$\C^d$ are $\{0\}$ and 
$\mathcal{O}_i$ for $1 \le i \le d$,
where 
\[ \mathcal{O}_i = \{v \in \C^d : v_1 = \ldots = v_{i-1} = 0, v_i \not=0 \}.\]
If $V$ is a subspace of $\C^d$ such that $P(V) \subseteq V$
then since $B(V) = V$, the subspace $V$ is a union of orbits of these orbits.
Applying this observation to each $V_i$ in turn
we find that $V_i = C_{d_i}$ where $d_i = \dim V_i$, for each $1 \le i \le r$.
Hence $P = \Stab(\C^d \supset C_{d_1} \supset \ldots \supset C_{d_r} \supset 0)$
as required.
\end{proof}

\begin{proposition}\label{prop:parabolicConjugacy}
Each parabolic subalgebra of $\Mat_d(V)$ is conjugate to a unique standard parabolic.
\end{proposition}

\begin{proof}
Let $P = \Stab(\C^d \supset V_1 \supset \ldots \supset V_r \supset 0)$ be a parabolic subalgebra
of $\Mat_d(\C)$. Starting with $V_r$ and working backwards, one may
construct an invertible matrix $\MX$ such that
$\MX(V_i) = C_{\dim {V_i}}$ for each $1 \le i \le r$. By~\eqref{eq:parabolicConjugacy} 
we have $\MX P \MX^{-1} = \Stab(\C^d \supset C_{\dim V_1} \supset \ldots \supset C_{\dim V_r} \supset 0)$.
By the `if' direction of Lemma~\ref{lemma:standardParabolic}, $\MX P\MX^{-1}$ is a standard parabolic.
By~\eqref{eq:parabolicConjugacy}, the dimensions of the $V_i$ are preserved by
conjugacy. Therefore $\MX P \MX^{-1}$ is the unique standard parabolic conjugate to $P$.
\end{proof}

\subsection{Right idealizers}\label{subsec:rightIdealizers}
We finish our algebraic background
with a result on idealizers needed in the proof of Theorem~\ref{thm:mainGL}.
Given a left ideal $L$ in $\C[G]$ its
\emph{right idealizer} $\RId_{\C[G]}(L)$ 
is the largest subspace~$W$ of $\C[G]$ 
such that $L$ is a right ideal in~$W$. Note that since $L$ is closed under multiplication,
$W$ contains $L$. In symbols,
\[
\RId_{\C[G]}(L) = \{w \in \C[G] : Lw \subseteq L\}.
\]
\vspace{14pt}
\begin{lemma}\label{lemma:RId}
Let $e\in\C[G]$ be an idempotent, let $L = \C[G]e$ be a left ideal of $\C[G]$.
Then $\RId_{\C[G]}(L) = \C[G]e + (1-e)\C[G]$. Moreover, the Wedderburn component of 
$\RId_{\C[G]}(L)$ in the block $\Mat(V)$ is (up to a choice of isomorphism) the parabolic subalgebra
$\Stab(\C^d \supset C_{r} \supset \C^d)$

\vspace*{-12pt}
\[
\includegraphics[page=9]{AllPictures.pdf}\vspace*{-6pt}
\]
where $r$ is the multiplicity of the irreducible module $V$ in $L$ and $d = \dim(V)$.
\end{lemma}

After this lemma, the formula
\(
\RId_{\C[G]}(L) = \C[G]e + (1-e)\C[G]
\)
can be represented (on each Wedderburn block) as
\[
\includegraphics[page=10]{AllPictures.pdf}
\]
\begin{proof}[Proof of Lemma~\ref{lemma:RId}]
    By the Wedderburn decomposition in Proposition~\ref{prop:Wedderburn}, we can write $e = \sum_{V \in \Irr(G)} e_V$ in a unique way, where $e_V$ is the part of $e$ supported on the Wedderburn block $\isoblock{V}$. Moreover, since each Wedderburn block is an algebra, we deduce that $e_V$ is an idempotent for every $V\in\Irr(G)$.
    Up to a choice of Wedderburn isomorphism, we can identify each $e_V$ with the diagonal matrix $\diag(1,...,1,0,...,0)$ with exactly $r = \langle \chi_{\C[G]e}, \psi_V\rangle$ ones.
    This gives
    
    \vspace*{-18pt}
    \[
    \C[G]e = \raisebox{-\mbaseline}{\includegraphics[page=11]{AllPictures.pdf}}
    \quad\quad \text{and}\quad\quad
    (1-e)\C[G]e = \raisebox{-\mbaseline}{\includegraphics[page=12]{AllPictures.pdf}}.
    \]
    
    \vspace*{-6pt}\noindent
    Checking the claims is now a linear algebra exercise.
\end{proof}

\begin{remark}\label{remark: RId is normaliser}
    Note that the right idealizer of a left ideal and its normalizer are closely related. Indeed, let $L$ be a left ideal. Then working in the group $\C[G]^\times$ of invertible elements of $\C[G]$, we have
    \[ N_{\C[G]^\times}L = \{w\in\C[G]^\times :  w^{-1}Lw = L\}    
    = \{w\in\C[G]^\times :  Lw = L\} 
    = \RId_{\C[G]}(L) \cap \C[G]^\times \]
where the second equality holds because $L$ is a left ideal
and the third because $Lw = L$ if and only if $Lw \subseteq L$ for invertible
elements $w$.
\end{remark}

\section{Double cosets}\label{sec:doubleCosets}
In this section we collect some basic results on double cosets, giving a short algebraic
proof of a key `averaging' lemma and a special case of Mackey's restriction formula.
 Fix throughout this section a finite group $G$ and subgroups $T$ and $H$ of $G$.
By definition, the double coset $TxH$ is the set $\{ t x h : t \in T, h \in H\}$.

\subsection{Counting results}
It is clear that $TxH$ is a union of left cosets of $H$, and also
a union of right cosets of $T$. The different expressions for elements of $TxH$ all come from the equation
\begin{equation} txh = tsx s' h \label{eq:doubleCosetAllForms} \end{equation}
where $s \in T \cap xHx^{-1}$ and so $s' = x^{-1}s^{-1} x \in xTx^{-1} \cap H$.
It follows that
\begin{align*}
TxH &= \{txh : t \in T, h \in (x^{-1}Tx \cap H) \backslash H \} \\ 
    &= \{txh': t \in t / (T \cap xHx^{-1}), h \in H \}. 
\end{align*}
Thus $TxH$ is a disjoint union of the right cosets $Txh$ for $h$ in a set of representatives
for the right cosets of $x^{-1}T x \cap H$ in $H$, and also a disjoint
union of the left cosets $txH$ for $t$ in a set of representatives for the left cosets
of $T \cap xHx^{-1}$ in $T$. This is shown diagrammatically below,
using the identity as one coset representative in each case.

\begin{center}
\includegraphics[page=13]{AllPictures.pdf}
\end{center}
By~\eqref{eq:doubleCosetAllForms}, each box has $|T \,\cap\, xHx^{-1}| 
= |x^{-1}Tx \,\cap\, H|$ different elements
of $HxH$. 
Denoting this common value by $r$, there are $|T|/r$ rows and $|H|/r$ columns.
By counting elements we obtain the equation $r |T||H|/r^2 = |TxH|$, or equivalently
\begin{equation}
|TxH| \, |x^{-1}Tx \cap H| = |H|\,|T|. 
\label{eq:doubleCosetCountTxH} 
\end{equation}
As an immediate application of~\eqref{eq:doubleCosetCountTxH} we prove
the following key `averaging' lemma.
Recall from the start of \S\ref{subsec:mainResults} that if $K \le G$ is a subgroup
then $\eta_K = |K|^{-1}\sum_{k \in K}k \in \C[G]$, where $\C[G]$ is the group
algebra defined in~\S\ref{subsec:groupAlgebra}.

\begin{samepage}
\begin{lemma}\label{lemma:averaging}
For $w \in \C[G]$ we have
\begin{thmlistEL}
\item $\eta_T w$ is constant on each right coset $Tg$ in $TxH$ and its common value on $Tx$
is $w(Tx)/|T|$;
\item $w \eta_H$ is constant on each left coset $gH$ in $HxH$ and its common value on $xH$ is $w(xH)/|H|$;
\item $\eta_T w \eta_H$ is constant on $TxH$ and its common value on the double coset is $w(TxH)/|TxH|$.
\end{thmlistEL}
\end{lemma}
\end{samepage}

\begin{proof}
We have $(\eta_T wv)(x) = |T|^{-1}\sum_{t \in T} w(tx)$;
this is $w(Tx)/|T|$, as required for (i). The proof of (ii) is dual.
By (i) and (ii) each weight in $\eta_T \C[G] \eta_H$ is constant on $TxH$.
Let the common value of $\eta_T w  \eta_H$ be $c$. By (i), then (ii), then~\eqref{eq:doubleCosetCountTxH},
we have 
\[ \begin{split}
    c = \frac{(w \eta_H)(Tx) }{|T|}= \frac{1}{|T||H|} \sum_{h \in H} w(Txh) 
    = \frac{|x^{-1}Tx \cap H|}{|T||H|} w(TxH) = \frac{w(TxH)}{|TxH|} \end{split} \]
as required.
\end{proof}

We outline an alternative proof of (iii) that probability-minded readers
will find very intuitive:
one can sample uniformly at random from $TxH$  by choosing $t \in T$ and $h \in H$
according to the distributions $\eta_T$ and $\eta_H$ and then taking $txh$.
Hence the expected value of $w$ on $TxH$, namely $w(TxH)/|TxH|$,
is the common value of $\eta_T w \eta_H$ on~$TxH$.

\bigskip 
\subsection{Mackey's rule for the trivial representation}
We have seen that the 
orbits of $T$ acting on the left on the set $G/H$  of left cosets are each of the form $\{txH : t \in T\}$,
and that the stabiliser of the distinguished orbit 
representative $xH$ is 
$T\cap xHx^{-1}$. 

\begin{lemma}[Mackey's rule for the trivial representation] 
    \label{lemma:MackeyTrivial} Let $\C$ be the trivial representation of $H$. There is an isomorphism of left $\C[H]$-modules
    \[
    \C\Ind_H^G\Res_H \cong \bigoplus_{x \in H \backslash G / H} \C\Ind_{H \cap xHx^{-1}}^H,
    \]
    where the sum ranges over a set of double coset representatives for $H\backslash G/H$.
\end{lemma}

\begin{proof}
Taking $T = H$ in the remark before the lemma, it follows that
the $\C[H]$-permutation module of $H$ acting transitively on its orbit $\{ hxH : h \in H\}$ of left cosets
is $\C\Ind_{H \cap xHx^{-1}}^H$.
\end{proof}

This result is generalized by Lemma~\ref{lemma:MackeyOnHxH} below.

\section{A characterisation of weak lumping of left-invariant random walks}
\label{sec: a characterisation of WL}

Recall that in our standing notation $H$ is a fixed subgroup of the finite group $G$.
In this section we work throughout with an irreducible
weight $w$ and characterise the initial distributions $\alpha$
such that the left-invariant random walk driven by $w$ with initial distribution $\alpha$
lumps weakly to the left cosets $G/H$, in the sense of Definition~\ref{defn:lumpsWeakly}(a).

\subsection{The minimal Gurvits--Ledoux space and minimal Gur\-vits--Ledoux left
ideal 
}\label{subsec:minimalGLideal}
Recall from before Definition~\ref{defn:inducedIdeal}
that if $bH \in G/H$ then
$\pi_{bH} : \C[G] \rightarrow \C[G]$ is the projection map onto the coset $bH$, defined by
$\pi_{bH} \bigl( \sum_{g \in G} x(g) g \bigr) = \sum_{g \in bH} x(g) g$.
By Definition~\ref{def: minimal GL space}, the minimal complex Gurvits--Ledoux
vector space  for the left-invariant random walk driven by $w$
started at the probability distribution $\alpha \in \C[G]$
is the intersection of all vector subspaces~$V$ of $\C[G]$ satisfying
\begin{axioms}{V}
\setcounter{enumi}{-1}
    \item $\alpha\in V$, \label{axiom V0}
    \item $Vw \subseteq V$, \label{axiom V1}
    \item $\pi_{bH}(V) \subseteq V$ for all $bH\in G/H$. \label{axiom V2}
\end{axioms}
Note this intersection is well-defined as  $\C[G]$ satisfies \ref{axiom V0}, \ref{axiom V1} and \ref{axiom V2} 
and each of the conditions is closed under intersection.
We denote this minimal vector space by $V_{\alpha, w}$.  
To compare with the definitions in \S\ref{SS:GL}, $V_{\alpha,w}$ can be identified with 
$\C\otimes_{\R}V(f,P,\alpha)$ where $P$ is the transition matrix of the left-invariant walk driven by $w$ and $f$ is the left coset mapping $G \to G/H$. Note that $V_{\alpha,w}$ is a complex vector subspace of the group algebra $\C[G]$.

Recall that if $T$ is a non-empty subset of $G$ then $\eta_T \!=\! |T|^{-1} \sum_{t \in T} t \in \C[G]$.

\begin{lemma}\label{lem: irreducible case}
Let $w$ be an irreducible weight.
If $V$ is a subspace of $\C[G]$ satisfying \ref{axiom V0}, \ref{axiom V1} and \ref{axiom V2} then
$\eta_G \in V$.
\end{lemma}
\begin{proof}
Since $w$ is irreducible, the unique stationary distribution of the left-invariant random walk is $\eta_G
\in \C[G]$. As in Definition~\ref{def: minimal GL space}, $V$ is a complex Gurvits--Ledoux space containing a probability vector $\alpha$, so $V \supseteq \C \otimes_{\R}V(f,P,\alpha)$. By Lemma~\ref{lem:useErgodicThm}, $\eta_G \in V(f,P,\alpha)$. \end{proof}

\begin{proposition}\label{prop: irreducible case}
If $w$ is an irreducible weight then $\LwAsV \subseteq V_{\alpha,w}$.
\end{proposition}

\begin{proof}
If $V$ is a subspace of $\C[G]$ satisfying
\ref{axiom V0}, \ref{axiom V1} and \ref{axiom V2}
\emph{for $\alpha$} then, by Lemma~\ref{lem: irreducible case},~$V$ satisfies \ref{axiom V0} \emph{for $\eta_G$}, and hence $V$ contains the minimal complex
Gurvits--Ledoux space $V_{\eta_G, w}$. The proposition follows by
taking the intersection over all such $V$.
\end{proof}

\begin{lemma}
\label{lemma:LwIsInduced}
The subspace $\LwAsV$ is 
a left ideal of $\C[G]$. Moreover,
    $\pi_H(\LwAsV)$ is a left ideal of $\C[H]$ and $\LwAsV = \C[G]\bigl( \pi_H (\LwAsV) \bigr)$
    is an induced ideal.
\end{lemma}
\begin{proof}
For readability, let $L = \LwAsV$.
To show that $\LwAsL$ is a left ideal it suffices
to show that $k^{-1}\LwAsL \supseteq \LwAsV$ for each $k \in G$, since by multiplying
by $k$ we then obtain $\LwAsL \supseteq k\LwAsL$. Setting $g = k^{-1}$
we therefore check that $g\LwAsL$ satisfies \ref{axiom V0}--\ref{axiom V2}. We may
then deduce 
$g\LwAsL \supseteq \LwAsL$ by minimality of~$\LwAsL$:
    \begin{defnlistE} 
        \item[(V0)] $g\eta_G = \eta_G \in \LwAsL$ by \ref{axiom V0} for $\LwAsL$;
        \item[(V1)] $(g\LwAsL) w \subseteq g\LwAsL$ since $\LwAsL w \subseteq \LwAsL$ by \ref{axiom V1} for $\LwAsL$;
        \item[(V2)] since $\pi_{bH} g = g \pi_{g^{-1}bH}$, we have
        $\pi_{bH}(g\LwAsL) = g\pi_{g^{-1}bH}(\LwAsL) \subseteq g\LwAsL$
 by \ref{axiom V2} for $\LwAsL$.
    \end{defnlistE}
Hence $\LwAsL$ is a left ideal of $\C[G]$.
Again by \ref{axiom V2} for $\LwAsL$, we have $\pi_H(\LwAsL) \subseteq \LwAsL$.
Therefore by Proposition~\ref{prop:inducedIdealCharacterisation}(ii), $\LwAsL = \C[G]\bigl( \pi_H(\LwAsL) \bigr)$.
\end{proof}

Thus $V_{\eta_G, w}$, which \emph{a priori} was merely a subspace of $\C[G]$,
is in fact an induced ideal in the sense
of Definition~\ref{defn:inducedIdeal}. 

\subsection{Application of the Gurvits--Ledoux characterisation to prove Corollary~\ref{cor:LTest}}
\label{subsec:applicationGL}
Let $\C^{G/H}$ be the vector space of complex valued functions on $G/H$
and let $\Lambda : \C[G] \to \C^{G/H}$ be the linear map induced by the canonical quotient
map $G \mapsto G/H$ defined by $g \mapsto gH$.
\begin{lemma}\label{lem:kerLambda}
$\ker \Lambda = \C[G](1-\eta_H)$.
\end{lemma}
\begin{proof}
We have $x \in \ker \Lambda$ if and only 
if $\sum_{g \in bH} x(g) = 0$ for each $b \in G$,
and so if and only if $x \in \bigoplus_{b \in G/H} b \hskip1pt \C[H] (1-\eta_H) = \C[G](1-\eta_H)$.
\end{proof}
Given a vector subspace $V$ of $\C[G]$, let $V^\circ$ denote the subspace $V \cap \ker \Lambda$, in accordance with \eqref{eq:VcircDefn}.  We introduce a further property:
\begin{axioms}{V}
\setcounter{enumi}{2}
    \item $V^\circ w\subseteq V^\circ$. \label{axiom V3}
\end{axioms}

\smallskip\noindent
Since $\ker \Lambda$ and $V_{\alpha,w}$ have bases consisting of real vectors, $V_{\alpha,w}^\circ$ also has a basis consisting of real vectors, which span the real vector space $V(f,P,\alpha)^\circ$. Moreover, $V(f,P,\alpha)^\circ P \subseteq V(f,P,\alpha)^\circ$ if and only if $V_{\alpha,w}^\circ w \subseteq V_{\alpha,w}^\circ$. Thus in this setting the conclusion of 
Theorem \ref{thm: GL characterisation of WL} may be written as follows:
\[
\text{$\MC(\alpha,w)$ lumps weakly to $G/H$ if and only if $V_{\alpha,w}$ satisfies \ref{axiom V3}.}
\tag{$\star$}
\]
We use  ($\star$) to prove Corollary~\ref{cor:LTest}. The following lemmas are required.

\begin{lemma}\label{lem: weak lumpability of irreducible weights}
Let $w$ be an irreducible weight. 
If \ref{axiom V3} holds for $V_{\alpha,w}$ then \ref{axiom V3} holds 
for $\VOrLw$.
\end{lemma}
\begin{proof}
We have 
$\VOrLw^\circ w \subseteq \VOrLw w \subseteq \VOrLw$.
By Proposition~\ref{prop: irreducible case} we have 
$\VOrLw\subseteq V_{\alpha,w}$. Hence
\[ \VOrLw^\circ w = (\VOrLw \cap \ker \Lambda)w \subseteq (V_{\alpha,w} \cap \ker \Lambda)w  =  V_{\alpha,w}^\circ w \subseteq V_{\alpha,w}^\circ \subseteq \ker \Lambda \,\]
and so $\VOrLw^\circ w \subseteq \VOrLw \cap \ker \Lambda = \VOrLw^\circ$.
\end{proof}

\begin{lemma}\label{lemma:Vcirc}
Let $V$ be a subspace of $\C[G]$ satisfying \ref{axiom V2}. Suppose that
$\eta_G \in V$.
Then $V^\circ = V(1-\eta_H)$.
\end{lemma}

\begin{proof}
By \ref{axiom V2}, $V = \bigoplus_{b \in G/H} V \,\cap\, b \hskip1pt \C[H]$
where $V \cap b \hskip1pt \C[H] = \pi_{bH}(V)$. By \ref{axiom V0} and \ref{axiom V2} 
we have $\pi_{H}(\eta_G) = \eta_H \in V$.
Given $u \in b \hskip1pt \C[H]$ we have $ u \eta_H \in \langle b \eta_H \rangle$,
hence each subspace $\pi_{bH}(V)$ is closed under right multiplication
by the idempotent $1-\eta_H$.
Now using $V^\circ = V \cap \ker \Lambda$ 
and Lemma~\ref{lem:kerLambda} we have
\[ V^\circ = \Bigl(\, \bigoplus_{b \in G/H} \pi_{bH}(V) \Bigr) \cap
\Bigl(\, \bigoplus_{b \in G/H} b \hskip1pt \C[H](1-\eta_H) \Bigr)
= \bigoplus_{b \in G/H} \pi_{bH}(V) (1-\eta_H). \]
Therefore $V^\circ = V(1-\eta_H)$, as required. \end{proof}

\begin{lemma}\label{lemma: ker Lambda is Ann etaH}
If $L$ is an induced left ideal from $H$ to $G$ containing $\eta_G$ then $L^\circ = L (1-\eta_H)$.
\end{lemma}

\begin{proof}
By Proposition~\ref{prop:inducedIdealCharacterisation}(iii), $L$
satisfies (V2). Now apply Lemma~\ref{lemma:Vcirc}.
\end{proof}

For ease of notation, from now on we shall write $L_w$ for $V_{\eta_G, w}$; note that
by Lemma~\ref{lemma:LwIsInduced}, $L_w$ is a left ideal of $\C[G]$, so this is consistent
with our usual notational conventions.

\newcounter{tmp}
\setcounter{section}{1}

\setcounter{tmp}{\value{theorem}}
\setcounter{theorem}{\value{LTest}}

\begin{corollary}[Weak lumping test for a weight]
Let $w$ be an irreducible weight. The following are equivalent:
\begin{thmlist}
\item The left-invariant random walk driven by $w$ lumps weakly to $G/H$; 
\item $\MC(\eta_G, w)$ lumps weakly to $G/H$;
\item $L_w(1-\eta_H) w\eta_H = 0$;
\item The left-invariant random walk driven by $w$ lumps to $G/H$ with stable ideal $L_w$.
\end{thmlist}
\end{corollary}
\setcounter{theorem}{\value{tmp}}
\setcounter{section}{5}

\begin{proof}[Proof of Corollary~\ref{cor:LTest}]
 Suppose that (i) holds, so there exists
a starting distribution $\alpha$ such that $\MC(\alpha, w)$ lumps weakly to $G/H$.
Then, by ($\star$),
the minimal Gurvits--Ledoux vector space~$V_{\alpha, w}$ satisfies \ref{axiom V3}.
By Proposition~\ref{prop: irreducible case}, using our hypothesis
that $w$ is irreducible,~$\VOrLwInProof$ is contained in $V_{\alpha, w}$.
By Lemma~\ref{lem: weak lumpability of irreducible weights}, $\VOrLwInProof$ also satisfies \ref{axiom V3}.
Hence, by the `if' direction of ($\star$), $\MC(\eta_G, w)$ lumps weakly to $G/H$,
proving~(ii). 

By the `only if' direction of ($\star$),
(ii) implies in particular that $\VOrLwInProof$ satisfies \ref{axiom V3}, that is,
$\VOrLwInProof^\circ w \subseteq \VOrLwInProof$. 
By Lemma~\ref{lemma:LwIsInduced}, $\VOrLwInProof$ is an induced ideal and so
by Lemma~\ref{lemma: ker Lambda is Ann etaH}, $\VOrLwInProof^\circ = \VOrLwInProof(1-\eta_H)$.
Therefore $\VOrLwInProof(1-\eta_H) w \subseteq \VOrLwInProof(1-\eta_H)$.
In particular 
\[ \VOrLwInProof(1-\eta_H) w \eta_H \subseteq \VOrLwInProof(1-\eta_H) \eta_H = 0,\]
where the final equality holds because $(1-\eta_H)\eta_H = 0$. This proves (iii)
and shows that (iii) is equivalent to $\VOrLwInProof^\circ w \subseteq \VOrLwInProof^\circ$.
This is condition (d) in Definition~\ref{def: stable}, and conditions (a), (b) and~(c) in this definition
hold by (V0), \ref{axiom V1}, \ref{axiom V2} for $\VOrLwInProof$. Therefore (iii) implies (iv), namely that
the left-invariant random walk driven by $w$ lumps stably for the left ideal $\VOrLwInProof$.
In particular, the left-invariant random walk lumps weakly when started at $\eta_G$, which implies (i) on taking $\alpha = \eta_G$.
\end{proof}

\subsection{\texorpdfstring{The maximal Gurvits--Ledoux ideal $J_w$}{}}\label{subsec:Jw}
Playing an equally important role to the induced left ideal $L_w = V_{\eta_G, w}$,
which by Proposition~\ref{prop: irreducible case} should be thought of as the \emph{minimal}
Gurvits--Ledoux space for $\eta_G$ and $w$, 
we will shortly 
define the \emph{maximal} Gurvits--Ledoux ideal $J_w$. 
The following two lemmas are implied by
Lemma~\ref{L: lattice of stable spaces} but for the reader's convenience we give short proofs here in the language of group algebras.    
Recall property \ref{axiom V2} of a subspace $V$ of $\C[G]$ is that $\pi_{bH}(V) \subseteq V$ for all $bH \in G/H$.

\begin{lemma}\label{lemma:sumCirc}
Let $V$ and $W$ be subspaces of $\C[G]$ satisfying
\ref{axiom V2} and both containing $\eta_G$. Then $(V + W)^\circ = V^\circ + W^\circ$.
\end{lemma}

\begin{proof}
By Lemma~\ref{lemma:Vcirc} we have $V^\circ = V(1-\eta_H)$ and $W^\circ = W(1-\eta_H)$,
and since $V+W$ also satisfies \ref{axiom V2} and contains $\eta_G$, we also have
$(V+W)^\circ = (V+W)(1-\eta_H)$. The lemma follows.
\end{proof}

Recall that \ref{axiom V1} and \ref{axiom V3} are the properties that $Vw \subseteq V$ and $V^\circ w\subseteq V^\circ$,
respectively; \ref{axiom V2} has just been used.

\begin{lemma}\label{lem: maximal GL module exists}
Properties \ref{axiom V1}, \ref{axiom V2} and \ref{axiom V3} are closed under addition of vector subspaces that contain $\eta_G$.
\end{lemma}
\begin{proof}
Let $V$ and $W$ be vector subspaces of $\C[G]$ satisfying \ref{axiom V1}--\ref{axiom V3}.
Then $(V+W)w = Vw + Ww \subseteq V + W$ giving \ref{axiom V1} and $\pi_{bH}(V + W) = \pi_{bH}(V) + \pi_{bH}(W)
\subseteq V + W$, giving \ref{axiom V2}. Finally \ref{axiom V3} holds by Lemma~\ref{lemma:sumCirc}.
\end{proof}

By Corollary~\ref{cor:LTest} and ($\star$) in \S\ref{subsec:applicationGL},
if the left-invariant random walk lumps weakly to $G/H$ in the sense of Definition~\ref{defn:lumpsWeakly}(b),
then $L_w$ (the new name for $V_{\eta_G,w}$) satisfies \ref{axiom V1}, \ref{axiom V2}, and~\ref{axiom V3}.
We can thus consider the unique maximal subspace of $\C[G]$ satisfying these properties,
which we denote $J_w$.
\begin{lemma}
\label{lemma:JwIsInduced}
    Suppose that the irreducible weight $w$ lumps weakly to $G/H$.
    The maximal subspace $J_w$ of $\C[G]$ containing $\eta_G$ and satisfying properties \ref{axiom V1}, \ref{axiom V2}, and \ref{axiom V3}
     is a left ideal of $\C[G]$. Moreover,
    $\pi_H(J_w)$ is a left ideal of $\C[H]$ and $J_w = \C[G]\bigl( \pi_H(J_w) \bigr)$ is 
    an induced ideal.
\end{lemma}
\begin{proof}
We check in a very similar way to the proof of Lemma \ref{lemma:LwIsInduced} that if $g \in G$ then $gJ_w$ satisfies conditions \ref{axiom V1}, \ref{axiom V2}, and \ref{axiom V3}. Then, by maximality of $J_w$ it follows that $gJ_w \subseteq J_w$, and hence $J_w$ is a left ideal. We leave checking \ref{axiom V1}
and \ref{axiom V2} to the reader. To check \ref{axiom V3}, note that $\ker \Lambda = \C[G](1-\eta_H)$
is a left ideal of $\C[G]$ and so
\[ g(V^\circ) = g(V \cap \ker \Lambda) = gV \cap g \ker \Lambda = gV \cap \ker \Lambda = (gV)^\circ. \]
The end is exactly as in the earlier proof: by \ref{axiom V2} we have $\pi_H(J_w) \subseteq J_w$ and
so by Proposition~\ref{prop:inducedIdealCharacterisation}(ii), $J_w = \C[G]\bigl( \pi_H(J_w) \bigr)$. 
\end{proof}

We call $J_w$ the \emph{maximal Gurvits--Ledoux ideal for $w$}. As a complex vector space, $J_w = \C \otimes_R V_{\max}(f,P,\alpha)$.
We now use $J_w$ to prove
Theorem \ref{thm: main testDist}. 
See~\S\ref{sec: tests} for its algorithmic counterpart.

\setcounter{tmp}{\value{theorem}}
\setcounter{section}{1}
\setcounter{theorem}{\value{Jthm}}
\begin{theorem}[Weak lumping test for a distribution]
    Let $w\in\C[G]$ be an irreducible weakly lumping weight. 
    For each probability distribution $\alpha$, the Markov chain
    $\MC(\alpha,w)$ lumps weakly to $G/H$ if and only if $\alpha\in J_w$. 
\end{theorem}
\setcounter{section}{5}
\setcounter{theorem}{\value{tmp}}
\begin{proof}
    Suppose $\MC(\alpha, w)$ lumps weakly to $G/H$. Consider the minimal space $V_{\alpha,w}$ satisfying 
    \ref{axiom V0}, \ref{axiom V1}, and \ref{axiom V2}. By $(\star)$ in \S\ref{subsec:applicationGL},
   it also satisfies~\ref{axiom V3}. By maximality, $J_w$ contains $V_{\alpha,w}$ 
   and so  $\alpha \in J_w$.
    Conversely, suppose that
    $\alpha \in J_w$. That is, $J_w$ satisfies \ref{axiom V0}. By definition,  $J_w$ satisfies \ref{axiom V1},
    \ref{axiom V2}, and \ref{axiom V3}. Since $J_w = \C \otimes_{\R} V_{\max}(f,P,\alpha)$, if $\alpha \in J_w$ is a probability vector then $\alpha \in V_{\max}(f,P,\alpha)$ and hence, by Corollary~\ref{cor: GL certificate theorem},
    $\MC(\alpha,w)$ lumps weakly to $G/H$.
\end{proof}
\setcounter{theorem}{\value{tmp}}

\subsection{Gurvits--Ledoux ideals: the general case and the proof of Theorem~\ref{thm:mainGL}}
We have seen that the minimal Gurvits--Ledoux vector space $L_w$ and
the maximal Gurvits--Ledoux vector space $J_w$ controlling weak lumping
of a weight $w$ to $G/H$ are induced left ideals of $\C[G]$.
This motivates the following definition.

\begin{samepage}
\begin{definition}\label{defn:GLideal}
     Let $L$ be a left ideal of $\C[G]$. We say $L$ is a \emph{Gurvits--Ledoux ideal for~$w$} if
    \begin{axioms}{L}
    \setcounter{enumi}{-1}
        \item $\eta_G\in L$, \label{axiom L0}
        \item $Lw \subseteq L$, \label{axiom L1} and
        \item $L$ is an induced ideal from $H$ to $G$.
         \label{axiom L2}
    \end{axioms}
    We say that a Gurvits--Ledoux ideal for $w$ is \emph{weakly lumping} if
    \begin{axioms}{L}
    \setcounter{enumi}{2}
        \item $L^\circ w \subseteq L^\circ$,\label{axiom L3}
    \end{axioms}
  We denote by $L_{\alpha,w}$ the minimal Gurvits--Ledoux ideal containing the distribution $\alpha$.
\end{definition}
\end{samepage}

We immediately justify the term `weakly lumping' in this definition.

\begin{proposition}\label{prop:lumpsStably}
Let $L$ be a Gurvits--Ledoux ideal for the weight $w$. The left-invariant
random walk driven by $w$ lumps weakly
to $ G/H$ with stable space
$L$ if and only if $L^\circ w \subseteq L^\circ$.
\end{proposition}

\begin{proof}
We use the basic fact proven in Lemma~\ref{lemma:stepIsMultiplication} that right-multiplication by $w$ is right-multiplication by the transition matrix $P$ of the left-invariant random walk driven by $w$. Conditions (a), (b), and (c) of Definition~\ref{def: stable} are satisfied when $V$ is taken to be any Gurvits--Ledoux ideal $L$. Condition~(d) reduces to $L^\circ w \subseteq L^\circ$.
\end{proof}
\begin{corollary}
If there exists a Gurvits--Ledoux ideal $L$ for the weight $w$ such that $L^\circ w \subseteq L^\circ$, then the left-invariant random walk driven by $w$ lumps weakly to $G/H$.
\end{corollary}
\begin{proof} This follows from Proposition~\ref{prop:lumpsStably} and Corollary~\ref{cor: GL certificate theorem}.
\end{proof}

It is now natural to ask how $L^\circ$ can be computed. This has an appealing answer in terms of idempotents.
Observe that if $L$ is a Gurvits--Ledoux ideal for $w$
then by Proposition~\ref{prop:inducedIdealCharacterisation}(v) and \ref{axiom L2} there
exists an idempotent $e \in \C[H]$ such that $L = \C[G]e$. By \ref{axiom L0}, $L$ 
contains $\eta_G$,
and so $\eta_G e = \eta_G$. Applying
the projection map~$\pi_H$ we obtain
  $\eta_H e = \eta_H$.
Hence $e \in \Eb{H}$, the set of idempotents of $\C[H]$ such that $\eta_H e = \eta_H$.
Now by Lemma~\ref{lemma: ker Lambda is Ann etaH},
\[
\qquad L^\circ = L(1-\eta_H) = \C[G]e(1-\eta_H) = \C[G](e - e\eta_H) = \C[G](e - \eta_H).
\]
Thus $L^\circ$ is concretely described by idempotent multiplication.
Similarly, the following lemma translates conditions 
\ref{axiom L1} and \ref{axiom L3} into
the language of idempotents.

\begin{lemma}\label{lemma:circ}
Let $L = \C[G]e$ be an induced left ideal of $\C[G]$, for some $e \in \Eb{H}$. 
For $w \in \C[G]$ we have
\begin{thmlistE}
\item $Lw \subseteq L$ if and only if $ew(1-e) = 0$;
\item $L^\circ w \subseteq L^\circ$ if and only if $(e-\eta_H)w(1-e+\eta_H) = 0$.
\end{thmlistE}
Moreover \emph{(i)} and \emph{(ii)} are together equivalent to
\begin{thmlistE}\setcounter{thmlistcntE}{2}
\item $ew(1-e) = 0$ and $(e-\eta_H)w \eta_H = 0$.
\end{thmlistE}
\end{lemma}

\begin{proof}
Observe that if $f$ is an idempotent then $\C[G]f = \{ x \in \C[G]: x(1-f) = 0\}$.
Applying this with $f=e$ and then $f = e-\eta_H$ (using Lemma~\ref{lemma: ker Lambda is Ann etaH})
we get
\[
Lw \subseteq L \!\iff\! \C[G]ew \subseteq \C[G]e \!\iff\! \C[G]ew(1-e) = 0 \!\iff\! ew(1-e) = 0 
\]
proving (i)
and very similarly that $L^\circ w \subseteq L^\circ$ if and only if
$(e-\eta_H)w(1-e+\eta_H) = 0$
proving~(ii). Now multiplying $ew(1-e) = 0$ on the left by $1-\eta_H$
we get $(e-\eta_H)w (1-e) = 0$. Therefore, given that $ew(1-e) = 0$,
the conditions $(e-\eta_H)w(1-e+\eta_H) = 0$ and $(e-\eta_H)w\eta_H$
are equivalent. This proves (iii).
\end{proof}

The following further
remarks connect Definition~\ref{defn:GLideal} with the definitions and results presented
so far.

\begin{remarklist}
\item[(1)] If $w$ is an irreducible weight then, as in Lemma~\ref{lem:useErgodicThm}, then it follows from
\ref{axiom L1} that $\eta_G \in L_{\alpha, w}$. Thus in the irreducible case \ref{axiom L0} could be replaced by the condition that $L$ contains some probability vector.

\item[(2)] By \ref{axiom V2}, the subspace $V_{\alpha, w}$ 
of $\C[G]$ originally defined in \S\ref{subsec:minimalGLideal} is closed under the projection maps.
Therefore \[ \quad V_{\alpha, w} = \bigoplus_{b \in G/H} \pi_{bH} (V_{w,\alpha})
= \bigoplus_{b \in G/H} b \pi_H (b^{-1} V_{w,\alpha}) \] and it follows that
the left ideal of $\C[G]$ generated by $V_{\alpha, w}$
is $\C[G] U$ where $U = \sum_{b \in G/H} \pi_H(b^{-1}V_{w,\alpha})$.
It is therefore an induced ideal and so $L_{\alpha, w}$
is simply the smallest ideal containing the 
subspace $V_{\alpha,w}$.
\item[(3)]
By Lemma~\ref{lemma:LwIsInduced}, the minimal Gurvits--Ledoux space $V_{\eta_G, w}$ 
is an induced left ideal of $\C[G]$, and so we have 
\[\quad V_{\eta_G, w} = L_w = L_{\eta_G, w} \]
where the second equality holds by definition.
(Note the first equality is the change in notation introduced after 
the proof of Lemma~\ref{lemma: ker Lambda is Ann etaH}, so our usage of $L_w$ is consistent throughout.)
\item[(4)]
By Lemma~\ref{lemma:JwIsInduced},
the maximal Gurvits--Ledoux space $J_w$ 
is an induced left ideal and since, by its definition in Lemma~\ref{lemma:JwIsInduced}, it satisfies (V3),
the left-invariant random walk driven by $w$ lumps stably for the ideal~$J_w$.
\end{remarklist}

We can now prove Theorem \ref{thm:mainGL} and Corollary \ref{cor:GLmodules}, which we restate below.

\setcounter{tmp}{\value{theorem}}
\setcounter{section}{1}
\setcounter{theorem}{\value{mainGL}}
\begin{theorem}    Let $w$ be an irreducible
    weight on $G$ and let $\alpha$ be a distribution on $G$, both thought as elements of $\C[G]$.
    Then $\MC(\alpha, w)$ lumps weakly to $G/H$ if and only if there exists an idempotent $e\in\Eb{H}$ such that
    \begin{thmlistE}
    \item $\alpha\in\C[G]e$,
    \item $ e w (1-e) = 0$,
    \item $ (e - \eta_H) w \eta_H = 0$. 
    \end{thmlistE}
    In this case, for any $t\ge0$, the conditional distribution of $X_t$ given the sequence of cosets $X_0H$, ..., $X_{t-1}H$ always belongs to $\C[G]e$. 
\end{theorem}
\setcounter{section}{5}
\setcounter{theorem}{\value{tmp}}

\begin{proof}
Suppose that $\MC(\alpha,w)$ lumps weakly to $G/H$.
Recall that $J_w$ is the maximal Gurvits--Ledoux ideal for $w$ defined in  
\S\ref{subsec:Jw}. By Theorem~\ref{thm: main testDist}, proved at the end of \S\ref{subsec:Jw},
we have $\alpha \in J_w$, giving~(i).
As seen at the start of \S\ref{subsec:idempotents}, there exists an idempotent
$e \in \C[H]$ such that $J_w = \C[G]e$. Since $\eta_G \in J_w$ by \ref{axiom L0},
we have $e \in \Eb{H}$. By Lemma~\ref{lemma:circ} and \ref{axiom L2} and \ref{axiom L3},~$e$ satisfies (ii) and~(iii).

Conversely, suppose that there is an idempotent $e \in \Eb{H}$ 
satisfying (i), (ii) and (iii). 
Set $L = \C[G]e$. By (i), $L$ contains $\alpha$.
By Lemma~\ref{lemma:circ}(ii) and (iii),
 $Lw \subseteq L$ and $L^\circ w \subseteq L^\circ$. 
 Therefore by Proposition~\ref{prop:lumpsStably}, 
$\MC(\alpha, w)$ lumps weakly to $G/H$ with stable ideal $L$.
In particular $\MC(\alpha, w)$ lumps weakly.

The final claim restates the analysis preceding Theorem \ref{thm: GL characterisation of WL}.
\end{proof}

\setcounter{tmp}{\value{theorem}}
\setcounter{section}{1}
\setcounter{theorem}{\value{GLmodules}}
\begin{corollary}
    The Markov chain $\MC(\alpha,w)$ lumps weakly to left cosets of $G/H$ if and only if $L_{\alpha,w}^\circ w \subseteq L_{\alpha,w}^\circ$.
\end{corollary}
\begin{proof}
    Suppose that $\MC(\alpha,w)$ lumps weakly to left cosets of $H$. Then 
    by Theorem~\ref{thm:mainGL},
    there exists an idempotent $e\in\Eb{H}$ satisfying conditions (i), (ii), and (iii) from 
    this theorem.
    By Lemma~\ref{lemma:circ}
    and by minimality, we have $L_{\alpha,w}\subseteq \C[G]e$. 
    In particular, by (iii), we have $(e-\eta_H) w \eta_H = 0$ and so
    $\C[G]e(1-\eta_H) w \eta_H = 0$ and thus $L_{\alpha,w}(1-\eta_H) w \eta_H = 0$. Again by Lemma~\ref{lemma:circ},
    this becomes \ref{axiom L3}, as desired.
    
    Conversely suppose that $L_{\alpha,w}^\circ w \subseteq L_{\alpha,w}^\circ$. Then an idempotent generator of $L_{\alpha,w}$ satisfies conditions (i), (ii), and (iii) of Theorem \ref{thm:mainGL}.
\end{proof}
\setcounter{section}{5}
\setcounter{theorem}{\value{tmp}}

\subsection{Corollaries for strong and exact lumping}
We summarise the results of this section for these two special cases.
Strong lumping is defined in Definition \ref{defn:strongLumping}.
Note that conditions (iii)--(vi) in both propositions are independent of $\alpha$.

\begin{proposition}[Characterisations of strong lumping]\label{prop:strongJwIsCG} 
Let $w$ be an irreducible weight. The following are equivalent:
\begin{thmlist} 
    \item the left-invariant random walk driven by $w$ lumps strongly to $G/H$;
    \item $\MC(\alpha, w)$ lumps weakly to $G/H$ for all initial distributions $\alpha$;
    \item $J_w = \C[G]$ and $J_w^\circ = \C[G](1-\eta_H)$; 
    \item $\C[G]$ is a weak lumping Gurvits--Ledoux ideal for $w$;
    \item $(\ker \Lambda) w \subseteq \ker \Lambda$;
    \item $(1-\eta_H)w\eta_H = 0$;
    \item For each $g \in G$, $w(hgH)$ is constant for $h \in H$.
    \end{thmlist}
\end{proposition}

\begin{proof}
As we explained after Definition~\ref{defn:strongLumping}, strong lumping
implies weak lumping starting at an arbitrary initial distribution. Hence (i) implies (ii).
In this case, by Theorem~\ref{thm: main testDist}, $J_w = \C[G]$, and then
$J_w^\circ = \C[G](1-\eta_G)$ by Lemma~\ref{lemma: ker Lambda is Ann etaH}, giving (iii).
It is clear from Definition~\ref{defn:GLideal} that $\C[G]$ is a Gurvits--Ledoux
ideal for any weight. By definition $\C[G]$ is weakly lumping for the weight~$w$
if and only if $\C[G]^\circ w \subseteq \C[G]^\circ$; this holds by \ref{axiom V3}
for $J_w  = \C[G]$. Hence (iii) implies~(iv). By definition, 
$\C[G]^\circ = \ker \Lambda$ so (v) is a restatement of (iv).
Now (v) implies (vi) by taking $e = 1$ in Lemma~\ref{lemma:circ}
and (vi) implies (vii) by Lemma~\ref{lemma:averaging}(ii).
Finally suppose that (vii) holds.
The Dynkin condition for strong lumping (see Definition~\ref{defn:strongLumping})
is that 
\[ \sum_{x \in bH} w(a^{-1}x) = \sum_{x \in bH} w({a'}^{-1}x) \]
for all left cosets $bH$, all $a, a' \in G$ such that $aH = a'H$, and all $x \in G$. Equivalently,
$w(a^{-1}bH) = w((ah)^{-1}bH)$ for all $a$, $b \in G$ and $h \in H$, and this holds
by (vii) since $(ah)^{-1}bH = h^{-1}a^{-1}bH$. Hence (vii) implies (i), completing the cycle.
\end{proof}

Exact lumping is defined in Definition~\ref{defn:exactLumping}.

\begin{proposition}[Exact lumping]\label{prop:exactLwIsCGeta} 
Let $w$ be an irreducible weight. The following are equivalent:
\begin{thmlist} 
    \item the left-invariant random walk driven by $w$ lumps exactly to $G/H$; 
    \item $\eta_H w \in \langle b \eta_H : b \in G /H \rangle$;
    \item $L_w = \C[G]\eta_H$ and $L_w^\circ = 0$;
    \item $\C[G]\eta_H$ is a weak lumping Gurvits--Ledoux ideal for $w$;
    \item $\C[G]\eta_H w \subseteq \C[G]\eta_H$;
    \item $\eta_Hw(1-\eta_H) = 0$;
    \item For each $g \in G$, $w(Hgh)$ is constant for $h \in H$.
    \end{thmlist}
Moreover if $\alpha$ is an initial distribution then
$\MC(\alpha, w)$ lumps exactly to $G/H$ if and only if one of these conditions
holds and, in addition, the restriction of $\alpha$ to each left coset $bH$ is proportional to $b\eta_H$.
\end{proposition}

\begin{proof}
By the equivalence of (a) and (c) in Lemma~\ref{lem:EquivalentConditionsForExactLumping},
(i) holds if and only if
$\bigoplus_{b \in G/H} \langle b \eta_H \rangle$ is preserved by right multiplication by $w$.
Thus (i) and (ii) are equivalent. 
Moreover, by Corollary~\ref{cor:exactLumping}, $\MC(\alpha, w)$ lumps exactly if and only
if (i) holds and the restriction of $\alpha$ to each left coset $bH$ is proportional to $b\eta_H$.
To complete the proof it 
suffices to prove that conditions (ii)--(vii) are equivalent.

If (ii) holds then, considering the definition of $V_{\eta_G, w}$
in \S\ref{subsec:minimalGLideal}, we have
\[ \bigoplus_{b \in G/H} \langle b \eta_H \rangle = V_{\eta_G, w}.\]
This space is by definition (see the change of notation after Lemma~\ref{lemma: ker Lambda is Ann etaH}),
the minimal Gurvits--Ledoux ideal $L_w$. Therefore (ii) implies that $L_w = \C[G]\eta_H$.
In this case, by Lemma~\ref{lemma:circ}, $L_w^\circ = \C[G]\eta_H (1-\eta_H) = 0$. Hence (ii) implies (iii).
Suppose that (iii) holds.
Then $L_w = \C[G]\eta_H$ satisfies $L_w w \subseteq L_w$ by \ref{axiom V1}
and clearly $\eta_G \in L_w$ and since $\eta_H \in E(H)$, Proposition~\ref{prop:inducedIdealCharacterisation}(v)
implies that $L_w$ is a Gurvits--Ledoux ideal in the sense of Definition~\ref{defn:GLideal}.
By (iii) we have $L_w^\circ = 0$, hence $L_w^\circ w \subseteq L_w^\circ$ and we have (iv). 
Part (v) simply restates that $\C[G]\eta_H w \subseteq \C[G] \eta_H$, so (iv) implies (v),
and since $\eta_H w \in \C[G]\eta_H$ if and only if $\eta_H w (1-\eta_H) = 0$, (v) implies (vi).
Since (vi) is equivalent to $\eta_H w = \eta_H w \eta_H$, Lemma~\ref{lemma:averaging}(i) and (iii)
applied with $T = H$ imply that $\eta_H w$ is constant on each right coset $Hg$ in a given double coset,
hence (vi) and (vii) are equivalent. Finally (vi) implies (ii) since 
$\langle b \eta_H : b \in G/H \rangle$ is the kernel of right multiplication by $1-\eta_H$.
\end{proof}

We remark that the equivalence of (i) and (vii) in Proposition~\ref{prop:strongJwIsCG} and Proposition~\ref{prop:exactLwIsCGeta} proves Corollary~\ref{cor:strongExact}. We later
deduce this result in a more conceptual way as a corollary of Theorem~\ref{thm:mainTimeReversal}:
see \S\ref{sec: time reversal}.

We end this section with a joint corollary of Theorem~\ref{thm:mainGL} and the two propositions above
that deals with cases when $H$ is very small.

\begin{corollary}\label{cor:HOrder2Or3}
    Let $H$ be a subgroup of $G$, and suppose that $|H|\le 3$. Let $w$ be an irreducible weight and $\alpha$ any probability distribution on $G$. If $\MC(\alpha,w)$ lumps weakly then it lumps either strongly or exactly.
\end{corollary}
\begin{proof}
     If $|H| = 1$ then every weight lumps both strongly and exactly to $G/H = G$.   The only possibility in  Theorem~\ref{thm:mainGL} is $e = \eta_H = \id_H$, 
     and conditions (vi) in Proposition~\ref{prop:strongJwIsCG} and
      Proposition~\ref{prop:exactLwIsCGeta} hold trivially.
    
    Now suppose $|H|=2$. Let $H = \langle h\rangle$. Then $\Eb{H} = \{\id_H, \eta_H\}$. 
    There are two possible cases in Theorem~\ref{thm:mainGL}. If $e = \eta_H$, then 
    the condition $ew(1-e) = 0$ gives condition (vi) of Proposition~\ref{prop:exactLwIsCGeta} so we have exact lumping. If instead $e = \id_H$ the condition $(e - \eta_H)w\eta_H = 0$ gives condition (vi) of Proposition~\ref{prop:strongJwIsCG} so we have strong lumping.
    
    Finally, suppose $|H|=3$. Again let $H = \langle h\rangle$. There are three primitive idempotents 
    \[
    \def\arraycolsep{1pt}
    \begin{array}{rcrcrcrcr}
        \eta_H &=& \frac{1}{3}( & 1 &+ & h &+ & h^2),\\[3pt]
        \xi_H &=& \frac{1}{3}( & 1 &+ & \zeta^2 h &+ & \zeta h^2),\\[3pt]
        \overline{\xi}_H &=& \frac{1}{3}( & 1 &+ & \zeta h &+ & \zeta^2 h^2)
    \end{array}\]
    where $\zeta$ is a complex primitive third root of unity. Thus $\zeta^2 = \overline{\zeta}$.
    Since the initial distribution~$\alpha$ and the weight $w$ take real values, 
    the minimal Gurvits--Ledoux ideal  $L_{\alpha,w}$  defined in Definition~\ref{defn:GLideal}
    is closed under complex conjugation. If $L_{\alpha,w} = \C[G]\eta_H$ then by Proposition~\ref{prop:exactLwIsCGeta}(iv), $\MC(\alpha,w)$ lumps exactly to $G/H$.  Otherwise $L_{\alpha,w}$ cannot be either of $\C[G](\eta_H + \xi_H)$ or $\C[G](\eta_H + \overline{\xi}_H)$ since these are exchanged by complex conjugation.
    The only remaining possibility is that 
    $L_{\alpha,w} = \C[G]$,
    and then  Proposition~\ref{prop:strongJwIsCG}(iv) implies
    that $\MC(\alpha,w)$ lumps strongly to $G/H$.
\end{proof}

\section{Tests for weak lumping}
\label{sec: tests}
Corollary \ref{cor: main testWeight} and Theorem \ref{thm: main testDist} are characterisations of weak lumpability of a weight $w$ and weak lumping of $\MC(\alpha,w)$ for a distribution $\alpha$.
They rely on the computation of the left ideals $L_w$ and $J_w$ given a weight $w$.
In this section, we provide two practical computational procedures to compute these left ideals of $\C[G]$. Each algorithm uses a nested sequence of left ideals of $\C[H]$ that eventually stabilises, in one case to $\pi_H(L_w)$ and in the other to $\pi_H(J_w)$, respectively. 

By Lemmas~\ref{lemma:LwIsInduced} and~\ref{lemma:JwIsInduced}, $L_w$ (which was earlier denoted
$V_{\alpha_G, w}$) and $J_w$ are induced ideals of $\C[G]$, and so satisfy
$L_w = \C[G]\bigl(\pi_H(L_w) \bigr)$ and $J_w = \C[G] \bigl( \pi_H(J_w) \bigr)$
by Proposition~\ref{prop:inducedIdealCharacterisation}(ii).
When the order of $H$ is small compared to that of $G$, computing ideals of $\C[H]$
and hence induced ideals of $\C[G]$ is significantly more efficient than computing vector subspaces of $\C[G]$; it is in this sense that our algorithms become more powerful than those of \cite{RS1} and \cite{GL}, for instance Corollary~\ref{cor: algorithmic weak lumping test} which gave an algorithm for determining whether a general DTHMC $\MC(\alpha,P)$ lumps weakly under $f: A \to B$. Magma \cite{Magma} code that implements the two algorithms for calculating $L_w$ and~$J_w$
 is available as part of the arXiv version
of this paper.

\subsection{Weak lumping test for a weight}
\label{subsec: testWeight}
The characterisation of weak lumpability provided by Corollary \ref{cor: main testWeight} relies on $L_{w}$, the minimal Gurvits--Ledoux ideal for $w$. We construct it algorithmically as follows. 

Start with the left ideal $M_0 = \C[H]\eta_H$. 
Define inductively
\[
M_n := \pi_H\left(\C[G]M_{n-1} + \C[G]\bigoplus_{bH\in G/H}\pi_{bH}(M_{n-1}w)\right).
\]
Since $L_{w}$ is a left ideal and satisfies \ref{axiom L1} and \ref{axiom L2}, we have $\C[G]M_n\subseteq L_{w}$. By construction $M_n\supseteq M_{n-1}$ so the sequence $(\dim(M_n))_{\ge0}$ is increasing in $n$, takes integer values, and is bounded above by $\dim\C[G]$. So it must stabilise: there exists $N$ such that $M_n = M_N$ for all $n\ge N$.
In each step, $\C[G]M_n$ is a left ideal of $\C[G]$ containing $\eta_G$, and thus $\C[G]M_N$ is too. This shows that $\C[G]M_N$ satisfies \ref{axiom L0}. Since $M_N$ is by definition
an ideal of~$\C[H]$, the ideal $\C[G]M_N$ is an induced ideal, as required by 
\ref{axiom L2}.
Since $\C[G]M_N = \C[G]\pi_H(M_N)$, we have
\[
\C[G]\bigoplus_{bH\in G/H}\pi_{bH}(M_{N}w) \subseteq \C[G] M_{N}.
\]
This gives \ref{axiom L1}. 
Since $\C[G]M_n$ is contained in $L_{w}$, we obtain $L_{w} = \C[G] M_N$ by minimality.

\begin{example}\label{ex:LwConstruction}
We take $G = \Sym_4$ and $H = \Sym_{\{2,3,4\}}$. 
In \S\ref{subsec:shuffles} we showed that the 
weight \smash{$w = (1-\lambda)  \Id + \mfrac{\lambda}{3} \bigl(
 (1,4)(2,3) + (1,4,3) + (1,4,2,3) \bigr)$} defined in~\eqref{eq:w123}
lumps weakly to $G/H$ with stable ideal $\C[G]\eta_T$, where $T = \Sym_{\{2,3\}}$.
We now use  Corollary~\ref{cor:LTest} to find the left ideal $L_w$
when $0 < \lambda < 1$ and deduce that $w$ does not lump
stably for any proper subideal of $\C[G]\eta_T$. 
Following the construction above, we set
$M_0 = \C[H] \eta_H = \langle \eta_H \rangle$.
Calculation shows that the normalized
projections to the left cosets $H$, $(1,2)H$, $(1,3)H$ and $(1,4)H$ 
of $\eta_H w$ are
\begin{align*}
\pi_H(\eta_H w) &= \eta_H \\ 
\pi_{(1,2)H} (\eta_H w) &= \mfrac{1}{2}(1,4,2) + \mfrac{1}{2}(1,4,3,2) \\
\pi_{(1,3)H} (\eta_H w) &= \mfrac{1}{2}(1,4,3) + \mfrac{1}{2}(1,4,2,3) \\
\pi_{(1,4)H} (\eta_H w) &= \mfrac{1}{2} (1,4) + \mfrac{1}{2}(1,4)(2,3)
\end{align*} 
and the ideal of $\C[G]$ generated by the projections is $\C[G]\eta_T$.
Therefore $M_1 = \C[H]\eta_T$.
Since we know that $w$ lumps weakly to $G/H$ with stable ideal $\C[G]\eta_T$, 
it follows by minimality of $L_w$ that $L_w = \C[G]\eta_T$.
Alternatively this can be checked by calculating directly that $M_3 = M_2$.
Hence  the minimal Gurvits--Ledoux space $L_w$ is $\C[\Sym_4]\eta_T$. 
A very similar calculation shows that if 
$w' = \eta_T w$ as earlier then $L_{w'} = \C[\Sym_4]\eta_T$,
and so the stable lumping ideals found in the earlier example were minimal in both cases.
\end{example}

\subsection{Weak lumping test for a distribution}
\label{subsec: testDist}
Let $w\in\C[G]$ be an irreducible weight, and assume that it lumps weakly on left cosets of~$H$.
The set of distributions $\alpha$ such that $\MC(\alpha,w)$ lumps weakly on left cosets of $H$ is $J_w$ by Theorem~\ref{thm: main testDist}. In this section, we provide a practical computational procedure to compute $J_w$. It may be compared with the general algorithm that we gave after Corollary~\ref{cor: GL good initial distributions} for computing the maximal Gurvits--Ledoux space $V_{\max}(f,P)$ associated to an irreducible stochastic matrix that lumps weakly under $f: A \to B$.

Let $A_0 = \C[H]$. We have $J_w \subseteq \C[G]A_0$. 
Of the conditions defined at the start of \S\ref{subsec:minimalGLideal}, the left ideal
 $\C[G]A_0$ satisfies conditions \ref{axiom V1} and \ref{axiom V2}, but not necessarily condition \ref{axiom V3}.
Also, note $L_w \subseteq \C[G]A_0$.
Define inductively $B_n$ so that $B_n^\circ$ is the largest subspace of $A_n^\circ$ such that 
\[
\C[G]B_n^\circ w \subseteq \C[G]A_n^\circ.
\]
The largest such subspace is well defined, since this property is closed under sum. It is a left ideal 
of $\C[H]$ since this property is closed under left $\C[H]$-multiplication. Moreover, $L_w^\circ w \subseteq L_w^\circ \subseteq \C[G]A_n^\circ$ and thus $L_w$ is a subideal of $\C[G]B_n$.
Define $A_{n+1}$ so that $A_{n+1}^\circ$ is the largest subspace of $B_n^\circ$ such that 
\[
\big(\C[G]A_{n+1}^\circ \oplus \C[G]\eta_H\big) w \subseteq \C[G]B_n.
\]
Since $\C[G]\eta_H w \subseteq L_w w \subseteq L_w \subseteq \C[G]B_n$,
the largest such space is well defined, and it is a left ideal of $\C[H]$. 
We therefore have a sequence of nested left ideals of $\C[G]$ given by
\begin{equation}
    \label{eq:nesting}
\C[G]A_0 \supseteq \C[G]B_0 \supseteq \C[G]A_1 \supseteq \C[G]B_1 \supseteq \cdots,
\end{equation}
which is bounded below by $J_w$, since $J_w$ is defined to be the largest space satisfiying \ref{axiom V1}--\ref{axiom V3}. 
Therefore, the sequence stabilises: there exists $N$ such that $A_n = B_n = A_N$ for all $n\ge N$. Moreover, 
every term in  \eqref{eq:nesting} is an induced ideal, and hence so is $A_N$. This shows that $\C[G]A_N$ satisfies \ref{axiom V2}. By construction, 
\[
\C[G]B_N^\circ w \subseteq \C[G]A_N^\circ
\quad\text{and}\quad
\C[G]A_N w \subseteq \C[G]B_N,
\]
which implies that $\C[G]A_N = \C[G]B_N$ satisfies \ref{axiom V1} and \ref{axiom V3}. We conclude that $\C[G]A_N = J_w$ by maximality of $J_w$.

\begin{example}\label{ex:JwConstruction}
Again we use the example from \S\ref{subsec:shuffles}, taking
$G = \Sym_4$ and $H = \Sym_{\{2,3,4\}}$ and the weight $w$, now in the uniform case with $\lambda = \mfrac{3}{4}$,
so $w = \mfrac{1}{4} \bigl( \Id + (1,4)(2,3) + (1,4,3) + (1,4,2,3) \bigr)$.
Following the construction above we take $A_0 = \mathbb{Q}[H]$ and find using computer algebra
that $B_0^\circ$ is
the $3$-dimensional left ideal of $\C[H]$ generated
by $1 - (2,4) - (3,4) + \eta_T$, where $T = \langle (2,3) \rangle \le H$. 
Noting that $\eta_H(1 - (2,4) - (3,4)) = -\eta_H$, it follows that
$B_0$ is the $4$-dimensional left ideal of $\C[H]$ generated by $1 - (2,4) - (3,4)$.
A similar computer algebra calculation now show that $A_1^\circ$ is $2$-dimensional, spanned by
\begin{align*}
\id_{\Sym_4}\! -(2,4,3)\!+\! (2,3) \!-\! (2,4) &= 2\id_{\Sym_4} \!+ (2,3,4)\! +\! (3,4) + \!2(2,3) \!- \!6\eta_H, \\
(2,3,4) \!- \!(2,4,3) \!-\! (2,4)\! +\! (3,4) &= \id_{\Sym_4} \!+ 2(2,3,4) \!+ 2(3,4) \! + \!(2,3)\! - \!
6\eta_H .
\end{align*} 
It follows that $1 + (2,3) \in A_1$, and so $A_1$ is the $3$-dimensional ideal $\C[H]\eta_T$.
We know that $w$ lumps weakly to $G/H$ with this stable ideal, so the algorithm stabilises
at this point: $B_1^\circ = A_1^\circ$, $B_1 = A_1$, $A_2^\circ = A_1^\circ$ and $A_2 = A_1$.
(Again this may be verified by computer algebra.)
We conclude that $J_w = \C[G]\eta_T$ and, given the previous example, that in this case
the minimal and maximal Gurvits--Ledoux ideals $L_w$ and $J_w$ coincide.
If we instead take the strongly lumping weight $w' = \eta_T w$ then 
calculation shows that $B_0^\circ = \C[H](1-\eta_H)$ and $B_0 = \C[H]$, and now
it is immediate that $A_1^\circ = \C[H](1-\eta_H)$ and $A_1 = \C[H]$. Hence the algorithm
stabilises one step sooner and we obtain $J_w = \C[G]$, as expected from Proposition~\ref{prop:strongJwIsCG}(iii).
\end{example}

\section{The structure of the set of all weakly lumping weights}
\label{sec: all WL weights}

By Definition~\ref{defn:lumpsWeakly}, a weight $w$ lumps weakly to $G/H$ if $\MC(\alpha, w)$ lumps
weakly to $G/H$ for some initial distribution $\alpha$.
In~\eqref{eq:Theta} we defined the sets
\[   
\Theta(e) =  \bigl\{ w \in \C[G]: ew(1-e) = 0, (e-\eta_H)w\eta_H = 0 \bigr\}. 
\]
As we mentioned after Theorem~\ref{thm:mainGL}, it follows easily from this theorem
that  the set of weakly lumping weights
is $\Delta \cap \Theta$ 
where 
\begin{equation} \label{eq: union Theta}
\Theta = \bigcup_{e\in\Eb{H}} \Theta(e). 
\end{equation}
and $\Delta \subseteq \R[G]$ is the simplex of probability distributions.
In this section we begin by studying the sets $\Theta(e)$, which turn out to have remarkable
algebraic properties. In particular, they  are subalgebras of $\C[G]$.
 We compute the dimension of each $\Theta(e)$ 
 in Corollary \ref{cor:dimTheta}. We continue by considering when the union in~\eqref{eq: union Theta} 
is redundant, in the sense that one of these sets is contained in another. We then prove Proposition~\ref{prop: irredundant}.

\subsection{Weak lumping algebras}
\label{subsec: weak lumping algebras}

Let $e \in \Eb{H}$ be an idempotent, and let $L = \C[G]e$.
Recall that $\Theta(e)$ is the set of weights $w$ for which $L = \C[G]e$ is a weakly lumping Gurvits--Ledoux ideal in the sense of Definition~\ref{defn:GLideal}.
By Lemma~\ref{lemma:circ}(i) and (ii), 
    \begin{align*}
    \Theta(e) &= \{w\in\C[G] : Lw \subseteq L,\  L^\circ w \subseteq L^\circ\}\\
    &= \{w\in\C[G] :  ew(1-e) = 0, \ (e-\eta_H)w(1-e+\eta_H) = 0\}.
    \end{align*}
The set $\{w\in \C[G] : Lw \subseteq L\}$ is the right idealizer $\RId_{\C[G]}(L)$ of $L$. 
By Lemma \ref{lemma:RId},
    \begin{equation}
        \label{eq: Theta as intersection of idealizers}
    \Theta(e) = \RId_{\C[G]}(L)\cap\RId_{\C[G]}(L^\circ),
    \end{equation}
is a subalgebra of $\C[G]$.
We call it the 
\emph{weak lumping algebra of $e$}.

\begin{lemma}\label{lem:ThetaAlgebra}
    Let $e \in \Eb{H}$ be an idempotent.
    The weak lumping algebra of $e$ is a parabolic subalgebra of $\C[G]$. It satisfies 
    \[
    \Theta(e) = (\C[G](e-\eta_H) + (1-e)\C[G]) \oplus \eta_H\C[G]\eta_H.
    \]
\end{lemma}

\begin{proof}
    Let $L = \C[G]e$ and consider $\Theta(e) = \RId_{\C[G]}(L) \cap \RId_{\C[G]}(L^\circ)$. 
    Since $e\in\Eb{H}$ there is an idempotent decomposition of the identity element of $\C[G]$:
    \[
    1 = (e - \eta_H) + \eta_H + (1 - e).
    \]
    Using  Proposition \ref{prop:Wedderburn} we may choose a Wedderburn isomorphism 
   such that, on  the block $\isoblock{V}$ corresponding 
   to the irreducible character $\chi_V \in \IrrC{G}$, the elements $e$ and $\eta_H$ are sent to the diagonal matrices
    \[
    e = \diag(
    \underbrace{1, ..., 1, 1, ..., 1}_{\langle\chi_{L},\chi_V\rangle} ,\,  0, ..., 0)
    \text{ and }
    \eta_H = \diag(\underbrace{0, ..., 0}_{\langle\chi_{L^\circ},\chi_V\rangle}\hskip-2pt, \underbrace{1, ..., 1}_{\langle\triv\uparrow_H^G,\chi_V\rangle} \hskip-3pt,\, 0, ..., 0).    
    \]
    We represent this diagrammatically by
    \[
    \includegraphics[page=14]{AllPictures.pdf}
    \]
    By Lemma \ref{lemma:RId}, both $\RId_{\C[G]}(L)$ and $\RId_{\C[G]}(L^\circ)$ are
    standard parabolics, and their intersection is again a standard parabolic:
    using~\eqref{eq: Theta as intersection of idealizers} the part of this intersection in
    the block $\Mat(V)$ is
    \[
    \Theta(e)\cap\isoblock{V} =\!\!\raisebox{-\mbaseline}{
    \includegraphics[page=15]{AllPictures.pdf}}
    \]
    The part of the right-hand side in the lemma in the block $\Mat(V)$ is
    \[ \big((\C[G](e - \eta_H) + (1-e)\C[G]) \oplus \eta_H\C[G]\eta_H\big) \cap \isoblock{V},\] 
    which diagrammatically becomes
    \[
    \includegraphics[page=16]{AllPictures.pdf}
    \]
    Since the diagrams agree, the lemma holds for the part of the weak lumping algebra $\Theta(e)$
    in the block $\Mat(V)$. The lemma follows by summing over all blocks in the Wedderburn decomposition.
\end{proof}

\begin{remark}\label{remark: Hecke in weak lumping algebras}
    A consequence of the formula of Lemma \ref{lem:ThetaAlgebra} is that the subalgebra $\eta_H\C[G]\eta_H$,
    shown diagrammatically by the second summand above, is a direct summand common to all weak lumping algebras. This subalgebra can be identified with the set of $H$-bi-invariant functions on $\C[G]$; this is the Hecke algebra
    seen in Theorem~\ref{thm:mainTransitionMatrices} and its proof in \S\ref{subsec:HeckeAlgebras}.
\end{remark}

\begin{corollary}\label{cor:dimTheta}
Let $e \in \Eb{H}$ be an idempotent and let $L = \C[G]e$.
For each irreducible character $\psi\in\IrrC{G}$, define $a_\psi = \langle \chi_{L^\circ}, \psi\rangle$,
$c_\psi = \langle \triv_H\Ind_H^G, \psi\rangle$, and
$d_\psi = \dim \psi = \psi(1)$.
Then,
\[
\dim \Theta(e) = \sum_{\psi\in\IrrC{G}} (
a_\psi^2 + a_\psi c_\psi + c_\psi^2 - a_\psi d_\psi - c_\psi d_\psi + d_\psi^2).
\]
\end{corollary}
\begin{proof}
    Choose a Wedderburn isomorphism as in the proof of Theorem \ref{lem:ThetaAlgebra}.
    On each block $\isoblock{V}$ corresponding to the character $\psi \in \IrrC{G}$, the intersection $\Theta(e)\cap \isoblock{V}$ is mapped to the set of $d_\psi \times d_\psi$ matrices of the form
    \[
    \raisebox{-\mbaseline}{
    \includegraphics[page=17]{AllPictures.pdf}}.
    \]
    The dimension as a vector space is therefore
    \[
    a_\psi^2 + (a_\psi + c_\psi)c_\psi + d_\psi(d_\psi - a_\psi - c_\psi).\qedhere
    \]
\end{proof}

\begin{example}[Exact lumping]
\label{ex:exactThetaDim}
    By Proposition~\ref{prop:exactLwIsCGeta}, the Markov chain driven by $w$ lumps exactly if and only if $\C[G]\eta_H$ is a weakly lumping Gurvits--Ledoux ideal for $w$. 
    Equivalently, if and only if $w\in\Theta(\eta_H)$. Then, $\chi_{L} = \triv\ind_H^G$ is the induced trivial character of $H$. We have $a_\psi = 0$ for all $\psi\in\IrrC{G}$. Therefore,
\begin{align*}
\dim\Theta(\eta_H) &= \sum_{\psi\in\IrrC{G}}
\big(c_\psi^2 + d_\psi(d_\psi-c_\psi)).
\end{align*}
In terms of characters, letting $\phi_G$ denote the regular character of $G$, we can rewrite the above as
\[
\dim \Theta(\eta_H) = \langle \triv\Ind_H^G, \triv\Ind_H^G\rangle + \langle \phi_G, \phi_G-\triv\Ind_H^G\rangle
= |H\backslash G/ H| + |G| - \frac{|G|}{|H|},
\]  
where we used Lemma~\ref{lemma:MackeyTrivial} (Mackey's rule for the trivial character) 
to compute that $\langle \triv\Ind_H^G, \triv\Ind_H^G\rangle
= |H\backslash G/ H|$.
Corollary \ref{cor:strongExact} (proved in \S\ref{sec: time reversal} below)
gives another method for computing this dimension.
\end{example}
\begin{example}\label{eg: exact lumping algebra Sym3 Sym4}
    Let $H = \Sym_3$ and $G = \Sym_4$. Refer to Examples \ref{eg: Wedderburn iso Sym3} and \ref{eg: Wedderburn iso Sym4} for description of their respective irreducible representations and Wedderburn decompositions. We choose the Wedderburn isomorphism of $\C[G]$ as in Example \ref{eg: induction Sym3 to Sym4}
    so, in particular, the partitions labelling the blocks from top-left to bottom-right
    are $4$, $31$, $22$, $211$, $1111$.    
    We saw that in this chosen isomorphism,
    $L = S^{3} \ind_H^G$ corresponds to the submodule shown left
    below, and so $\RId(L)$ is as drawn right below:
   \[
     \raisebox{-\mbaselinemed}{\includegraphics[page=18]{AllPictures.pdf}},
     \hspace{6em}
     \raisebox{-\mbaselinemed}{\includegraphics[page=19]{AllPictures.pdf}}.
     \] 
    Since in this case $L^\circ = \varnothing$, we have $\RId(L) = \Theta(\eta_H)$.
    The weak lumping algebra $\Theta(\eta_H)$ is therefore
    of dimension $22 = 2 + 24 - 4$, as given by the formula of Example \ref{ex:exactThetaDim}.
    Note that this weak lumping algebra is a product of parabolic subalgebras,
    but, under this Wedderburn isomorphism, 
    the parabolic subalgebra for the block corresponding to the irreducible $S^{31}$ is 
    \emph{not} a standard parabolic. 
\end{example}

\subsection{Containment of weak lumping algebras}
\label{subsec: union of weak lumping algebras}

Whenever $H$ is non-abelian, the union \eqref{eq: union Theta} defining~$\Theta$ is over an uncountable set. We want to understand to what degree 
this expression is redundant. That is, for which idempotents $e$, $\tilde{e} \in \Eb{H}$ 
do we have $\Theta(e) \subseteq \Theta(\tilde{e})$?

The following definition defines Borel and parabolic subalgebras of the group algebra $\C[G]$.

\begin{definition}\label{defn:Borels}
Let $\C[G] \cong \bigoplus_{V \in \Irr(G)} \Mat(V)$ be a fixed Wedderburn isomorphism
of $\C[G]$. Fix an isomorphism of each $\Mat(V)$ with the matrix algebra $\Mat_{\dim V}(\C)$.
Under these isomorphisms:
\begin{defnlist}
\item The \emph{standard Borel} subalgebra of $\C[G]$ is the product of the standard Borel
subalgebras of lower triangular matrices in each factor $\Mat(V)$;
\item A subalgebra of $\C[G]$ is \emph{Borel} if it is conjugate by an element of $\C[G]^\times$ to the
standard Borel;
\item A subalgebra of $\C[G]$ is \emph{parabolic} if it is conjugate by an element of $\C[G]^\times$
to a subalgebra containing the standard Borel.
\end{defnlist}
\end{definition}

Recall that $\chi_V$ denotes the character of a representation $V$.

\begin{theorem}\label{thm: containment Thetas}
    Let $e$, $\tilde{e}\in\Eb{H}$ be idempotents. Set $L = \C[G]e$ and $\tilde{L} = \C[G]\tilde{e}$. %
    We have $\Theta(e) \subseteq \Theta(\tilde{e})$ if and only if for each irreducible 
    representation $V\in\Irr(G)$
    the following two conditions hold:

    \begin{thmlist}
        \item if $\langle\triv\Ind_H^G, \chi_V\rangle \ne 0$ then
        $\tilde{L}^\circ\cap \isoblock{V} = L^\circ\cap \isoblock{V} $;
        \item $\tilde{L}^\circ\cap \isoblock{V}$ is either $0$, or $L^\circ\cap \isoblock{V}$, 
        or $\isoblock{V}$.
    \end{thmlist}
\end{theorem}

The proof of this theorem will follow from four lemmas.
The first two are elementary results from geometric representation theory; we include proofs for completeness' sake.

We say a left ideal $L$ of $\Mat_d(\C)$ is an \emph{initial column span} if it is of the form $\Mat_d(\C)e$ for $e = \diag(1,...,1,0,...,0)$. That is, if we can write
\[
L = 
\raisebox{-\mbaseline}{\includegraphics[page=11]{AllPictures.pdf}}
.
\]
Let $B_d$ be the standard Borel subalgebra of lower triangular matrices in $\Mat_d(\C)$.

\begin{lemma}\label{lem: standard borel and column spans}
    If the right idealizer $\RId_{\Mat_d(\C)}(L)$ of a left ideal $L$ of $\Mat_d(\C)$ contains 
    $B_d$ then~$L$ is an initial column span. 
\end{lemma}
\begin{proof}
Let $L = \Mat_d(\C)e$ for an idempotent $e$. By hypothesis, we have $B_d\subseteq\RId_{\Mat_d(\C)}(L)$.
Since $e$ is an idempotent, we can write 
\[
e = y^{-1}
\left(
    \begin{tikzpicture}[x = 0.9em, y = 0.9em, baseline = 2em]
    \tiny
    \node (1) at (1.25, 3.78) {${}_1$};
    \node (1) at (1.6, 3.43) {$.$};
    \node (1) at (1.75, 3.28) {$.$};
    \node (1) at (1.9, 3.13) {$.$};
    \node (1) at (2.25, 2.78) {${}_1$};
    \node (1) at (2.75, 2.28) {${}_0$};
    \node (1) at (3.05, 1.98) {$.$};
    \node (1) at (3.25, 1.78) {$.$};
    \node (1) at (3.4, 1.63) {$.$};
    \node (1) at (3.75, 1.28) {${}_0$};
    \end{tikzpicture}
\right)
y
\]
for some invertible $y\in\Mat_d(\C)$. We deduce $yLy^{-1}$ is an initial column span and that $\RId_{\Mat_d(\C)}(yLy^{-1}) = y\RId_{\Mat_d(\C)}(L) y^{-1}$ is a standard parabolic as in the proof of Lemma \ref{lemma:RId}. But
by Proposition~\ref{prop:parabolicConjugacy}, two standard parabolics are conjugate if and only if they are equal. Thus $y$ normalises $\RId_{\Mat_d(\C)}(L)$. Reasoning as in Remark \ref{remark: RId is normaliser} and using Lemma \ref{lemma:parabolicSelfNormalizing} we have
\begin{align*}
y &\in N_{\Mat_d(\C)^\times} \big(\RId_{\Mat_d(\C)}(L)\big) \\
&= \RId_{\Mat_d(\C)}(L) \cap \Mat_d(\C)^\times \\
&= N_{\Mat_d(\C)^\times}(L).
\end{align*}
Hence  $yLy^{-1} = L$ is an initial column span. 
\end{proof}

The second lemma is a combinatorial exercise. It is most natural after the following description of standard parabolics. 
The symmetric group $\Sym_d$ is generated by the
 simple transpositions $s_i = (i, i+1)$ for $i \in \{1,\ldots, d-1\}$.
We define the \emph{parabolic subgroup} of $\Sym_d$ indexed by a tuple 
$\mbf{a} = (a_1, \ldots, a_k)$, where $0 \le a_1 < \ldots \le a_k \le d$,
to be the subgroup of $\Sym_n$ generated by the~$s_i$ for 
$i \not\in \{a_1,\ldots,a_k\}$, and denote it by $\Sym_d(\mbf{a})$. 
Identifying $\Sym_d$ with the group of permutation matrices, we can write the standard parabolic
subalgebras of $\Mat_d(\C)$ as $P_d(\mbf{a}) = B_d\Sym_d(\mbf{a})B_d$.

\vspace{12pt} 
\begin{lemma}\label{lem: containment std parabolics}
    Fix $0\le c\le d$ and two parameters $0\le a, \tilde{a}\le d-c$. Let $b = a + c$ and $\tilde{b} = \tilde{a}+c$. Then, $P_d(a,b)\subseteq P_d(\tilde{a},\tilde{b})$ if and only if
\begin{thmlist}
    \item $\tilde{a} = 0$ and $c\in\{0,a,d\}$,
    \item $\tilde{a} = a$,
    \item $\tilde{a} = a+c$ and $2c = d-a$, or
    \item $\tilde{a} = d$ and $c=0$.
\end{thmlist}
\end{lemma}

\begin{proof}
    We have the containment $P_{d}(a,b) \subseteq P_{d}(\tilde{a},\tilde{b})$ if and only if $\{0, a, b, d\} \supseteq \{0, \tilde{a}, \tilde{b}, d\}$, and so if and only if $\{0, a, a+c, d\} \supseteq \{0, \tilde{a}, \tilde{a} + c, d\}$.
\begin{itemize}
    \item If $\tilde{a} = 0$, then $\{0, a, a+c, d\} \supseteq \{0, c, d\}$ gives $c\in\{0,a,d\}$.
    \item If $\tilde{a} = a$, then the containment evidently holds.
    \item If $\tilde{a} = a+c$, then $\{0,a,a+c,d\} \supseteq \{0,a+c,a+2c,d\}$ gives either $c = 0$ (which gives $\tilde{a}=a$, as above) or $a+2c = d$. 
    \item If $\tilde{a} = d$, then $c = 0$.
    \qedhere
\end{itemize}
\end{proof}

\begin{lemma}\label{lemma:commonBorel}
    Let $e$, $\tilde{e}\in\Eb{H}$ be idempotents. Set $L = \C[G]e$ and $\tilde{L} = \C[G]\tilde{e}$. The
    parabolic subalgebras
    $\Theta(e)$ and $\Theta(\tilde{e})$ of $\C[G]$ share a common Borel if and only if for each irreducible 
    representation $V\in\Irr(G)$
    the following two conditions hold:
        \begin{itemize}
            \item if $\langle\triv_H\Ind_H^G, \chi_V\rangle \ne 0$ then
            $L^\circ\cap\isoblock{V} = \tilde{L}^\circ\cap\isoblock{V}$;
            \item either $L^\circ\cap\isoblock{V} \subseteq \tilde{L}^\circ\cap\isoblock{V}$ or $\tilde{L}^\circ\cap\isoblock{V} \subseteq L^\circ\cap\isoblock{V}$.
        \end{itemize}
\end{lemma}
\begin{proof}
    Given an irreducible module $V\in\Irr(G)$ and its character $\psi\in\IrrC{G}$, define $a_\psi := \langle \chi_{L^\circ}, \psi\rangle$ as in Corollary \ref{cor:dimTheta}, and similarly let $\tilde{a}_\psi = \langle \chi_{\tilde{L}^\circ}, \psi\rangle$.

    Suppose $\Theta(e)$ and $\Theta(\tilde{e})$ share a common Borel subalgebra. Fix a Wedderburn 
    isomorphism sending this Borel to the standard Borel of $\bigoplus_V \Mat(V)$,
    in the sense of Definition~\ref{defn:Borels}. Then,  by Lemma \ref{lem: standard borel and column spans},
    each of $L$, $L^\circ$, $\tilde{L}$, and $\tilde{L}^\circ$ are sent to initial column spans on each Wedderburn block.
    In each Wedderburn block $\isoblock{V}$, we have
    \[
    \begin{cases}
        L^\circ\cap\isoblock{V} \subseteq \tilde{L}^\circ\cap\isoblock{V} & \text{if } a_\psi \le \tilde{a}_\psi,\\
        \tilde{L}^\circ\cap\isoblock{V} \subseteq L^\circ\cap\isoblock{V} & \text{if } a_\psi \ge \tilde{a}_\psi\\
    \end{cases}
    \]
    satisfying the second condition.

    Fix $V \in \Irr(G)$ with irreducible character $\psi$. 
    Suppose that $\langle\triv_H\Ind_H^G, \psi_V\rangle \ne 0$. Since $L = L^\circ \oplus \C[G]\eta_H$, we can recover $\C[G]\eta_H \cap\isoblock{V}$ as the span of the columns of $L\cap\isoblock{V}$ which are not in $L^\circ\cap\isoblock{V}$. (The relevant columns are the first $a_\psi$ columns
    in the diagram below, repeated from Corollary~\ref{cor:dimTheta}.)
    \[
    \raisebox{-\mbaseline}{
    \includegraphics[page=17]{AllPictures.pdf}}
    \]
    Similarly, we can recover $\C[G]\eta_H\cap\isoblock{V}$ as the span of those columns of $\tilde{L}\cap\isoblock{V}$ which are not in $\tilde{L}^\circ\cap\isoblock{V}$. We deduce 
    $a_\psi = \tilde{a}_\psi$ and therefore
    $L^\circ\cap\isoblock{V} = \tilde{L}^\circ\cap\isoblock{V}$ as required by the first condition.

    Conversely, suppose both conditions are satisfied. Then there is a Wedderburn decomposition sending $L$, $L^\circ$, $\tilde{L}$, and $\tilde{L}^\circ$ to initial column spans on each Wedderburn component. Under this Wedderburn isomorphism $\Theta(e)$ and $\Theta(\tilde{e})$ are standard, and they share the standard Borel
    subalgebra $B_d$. 
\end{proof}
\begin{lemma}\label{lemma:thetaContainment}
    Let $e$, $\tilde{e}\in\Eb{H}$ be idempotents. Set $L = \C[G]e$ and $\tilde{L} = \C[G]\tilde{e}$. 
    We have
    $\Theta(e)\subseteq\Theta(\tilde{e})$ if and only if
        \begin{itemize}
            \item $\Theta(e)$ and $\Theta(\tilde{e})$ share a common Borel, and
            \item $\tilde{L}^\circ\cap \isoblock{V}$ is either $0$, or $L^\circ\cap \isoblock{V}$, 
        or $\isoblock{V}$.
        \end{itemize}
\end{lemma}
\begin{proof}
    We use the notation $a_\psi, c_\psi, d_\psi$ as in Corollary \ref{cor:dimTheta}, and $\tilde{a}_\psi$ as above.
    Suppose that $\Theta(e)\subseteq\Theta(\tilde{e})$. In particular, both parabolic algebras share a common Borel subalgebra. Fix a Wedderburn 
    isomorphism sending this Borel to the standard Borel of $\bigoplus_V \Mat(V)$. Then $L$, $L^\circ$, $\tilde{L}$, and $\tilde{L}^\circ$ are sent to initial column spans on each Wedderburn component by Lemma \ref{lem: standard borel and column spans}.
    By Lemma \ref{lem: containment std parabolics}, we have $\Theta(e)\subseteq\Theta(\tilde{e})$ only if $\tilde{a}_\psi \in \{0, a_\psi, a_\psi+c_\psi, d_\psi\}$ for each $V\in\Irr(G)$ with character $\psi$. Note that if $c_\psi\ne0$ then the equality $\tilde{a}_\psi = a_\psi+c_\psi$ would imply
    \[
    0\ne \C[G]\eta_H\cap\isoblock{V} \subseteq L\cap\isoblock{V} = \tilde{L}^\circ\cap\isoblock{V},
    \]
    which is a contradiction. Thus $\tilde{L}^\circ\cap\isoblock{V}\in\{0, L^\circ\cap\isoblock{V}, \isoblock{V}\}$ for all $V\in\Irr(G)$.

    The converse is Lemma \ref{lem: containment std parabolics}.
\end{proof}

The proof of Theorem~\ref{thm: containment Thetas} is now almost immediate.

\begin{proof}[Proof of Theorem~\ref{thm: containment Thetas}]
Suppose the two hypotheses of this theorem hold. By 
the `if' direction of Lemma~\ref{lemma:commonBorel}, the hypotheses imply
that $\Theta(e)$ and $\Theta(\tilde{e})$ share a common Borel subalgebra.
This gives the first hypothesis needed for the `if' direction of Lemma~\ref{lemma:thetaContainment},
and the second is hypothesis (ii) in the theorem. Therefore
$\Theta(e) \subseteq \Theta(\tilde{e})$.
Conversely, if $\Theta(e) \subseteq \Theta(\tilde{e})$
then the `only if' direction of Lemma~\ref{lemma:thetaContainment}
implies that $\Theta(e) \subseteq \Theta(\tilde{e})$
share a common Borel subalgebra and so the `only if' direction
of Lemma~\ref{lemma:commonBorel} may be applied.
\end{proof}

\begin{example}\label{eg: weak lumping algebras Sym3 Sym4}
    Continuing Examples \ref{eg: Wedderburn iso Sym3}, \ref{eg: Wedderburn iso Sym4}, \ref{eg: induction Sym3 to Sym4}, and \ref{eg: exact lumping algebra Sym3 Sym4}
    we set $H = \Sym_3$ and $G = \Sym_4$. 
Recall that 
  $e_{111} = \frac{1}{6}\big(1 - (12) - (13) - (23) + (123) + (132)\big)$ 
    is the centrally primitive idempotent of $\C[H]$
    such that $\C[H]e$ is the sign representation $S^{111}$ of $H$.
    Set $e = \eta_H + e_{111}$ and $\tilde{e} = \eta_H$. Then
   $L = \C[G](\eta_H + e_{111})$ and $\tilde{L} = \C[G]\eta_H$. 
    Noting that $\triv_H = \chi^3$, Example~\ref{eg: induction Sym3 to Sym4}(1)
    gives that
\[ \triv_H \Ind_H^G = \triv_G  + \chi^{31}. \]
Thus the two representations relevant to hypothesis 
(i) of Theorem~\ref{thm: containment Thetas} are
the trivial representation and $S^{31}$,
and since 
  \[ L^\circ = \C[G]e_{111} \cong S^{111}\Ind_H^G \cong S^{211} \oplus S^{1111}\] 
 and $\tilde{L}^\circ = 0$, this hypothesis holds.
 Since $\tilde{L}^\circ = 0$, hypothesis (ii) of Theorem~\ref{thm: containment Thetas} is immediate.
Therefore
$\Theta(e)\subseteq\Theta(\tilde{e})$. Using the Wedderburn isomorphism
introduced in Example~\ref{eg: induction Sym3 to Sym4},
the algebra $\Theta(\tilde{e})$ was found in Example~\ref{eg: exact lumping algebra Sym3 Sym4};
the relevant diagram is repeated below.
\[
    \raisebox{-\mbaseline}{
    \includegraphics[page=19]{AllPictures.pdf}}
\]
By~\eqref{eq: Theta as intersection of idealizers}, we have
\[ \Theta(e) = \Theta(\eta_H) \cap \RId(S^{111}\Ind_H^G). \]
Since, as remarked in  Example~\ref{eg: induction Sym3 to Sym4}, our Wedderburn isomorphism
is chosen so that multiplication by the sign representation of $H = \Sym_3$ rotates diagrams
by a half-turn, the
containment $\Theta(e) \subset \Theta(\tilde{e})$ is represented diagrammatically as shown below:
    \[
    \raisebox{-\mbaselinemed}{\includegraphics[page=20]{AllPictures.pdf}}
    \, \subseteq\, 
    \raisebox{-\mbaselinemed}{\includegraphics[page=19]{AllPictures.pdf}}.
    \]
\end{example}

\begin{corollary}\label{cor:containmentImpliesEqual}
Let $e$, $\tilde{e} \in \Eb{H}$ be idempotents. Set $L = \C[G]e$ and $\tilde{L} = \C[G]\tilde{e}$.
Suppose that $L$  and $\tilde{L}$ 
are isomorphic as $\C[G]$-modules. Then, $\Theta(e) \subseteq \Theta(\tilde{e})$ if and only if $L = \tilde{L}$.
\end{corollary}
\begin{proof}
    If $L$ and $\tilde{L}$ are isomorphic $\C[G]$-modules, then so are $L^\circ$ and $\tilde{L}^\circ$. In particular, they have the same character $\chi_{L^\circ} = \chi_{\tilde{L}^\circ}$. We deduce
    \[
    \dim \bigl( L^\circ \cap \isoblock{V}\bigr) =
    \langle \chi_{L^\circ}, \chi_V\rangle =
    \langle \chi_{\tilde{L}^\circ}, \chi_V\rangle =
    \dim \bigl( \tilde{L}^\circ \cap \isoblock{V} \bigr)
    \]
    for every irreducible representation $V \in \Irr(G)$. Now Theorem \ref{thm: containment Thetas} gives the result.
\end{proof}

We remark that when $H$ is non-abelian, by the remark at the 
end of the first paragraph of \S\ref{subsec:idempotents}, there are infinitely many distinct idempotents
affording each representation of $\C[H]$, and so Corollary~\ref{cor:containmentImpliesEqual}
is a non-trivial result.

\subsection{Subgroups with full induction restriction}
\label{subsec: full induction restriction}
We say the union defining $\Theta$ is \textit{completely redundant} if $\Theta = \Theta(\eta_H)$,
or equivalently, by Theorem~\ref{thm:mainGL},
if for all  $e\in\Eb{H}$, we have $\Theta(e) \subseteq \Theta(\eta_H)$.
We say that the union is \emph{irredundant} if a containment 
$\Theta(e)\subseteq\Theta(\tilde{e})$
between weak lumping algebras implies $\C[G]e = \C[G]\tilde{e}$. 
(Compare Corollary~\ref{cor:containmentImpliesEqual}, that if $\C[G]e \cong \C[G]\tilde{e}$
then $\Theta(e) = \Theta(\tilde{e})$.)
These two probability-theoretic properties are 
at opposite ends of a spectrum. In this subsection, we show they are determined by group-theoretic
properties of $H$ and $G$.

We begin with complete redundancy,
for which we need a representation-theoretic characterisation of normality.

\begin{lemma}\label{lem: H is normal iff characters}
    The subgroup $H$ of $G$ is normal
    if and only if
    $\langle\chi,\triv\Ind_H^G\Res_H\rangle=0$ for all irreducible characters $\chi\in \Irr(H)$
    such that $\chi\ne \triv_H$.
\end{lemma}

\begin{proof}
    The orbits of $H$ acting on the left on the set of left cosets $G / H$ correspond to double cosets $HxH$. The stabiliser in $H$ of $xH$ is $H \cap xHx^{-1}$. Thinking of $\triv\ind_H^G$ as the character of the permutation module (see Example~\ref{ex:permutationRepresentations}) 
    of $G$ acting on the left cosets $G/H$, the character-theoretic
    statement of Lemma~\ref{lemma:MackeyTrivial} (Mackey's Rule for the trivial character) is
\[ \triv\Ind_H^G \res_H = \sum_{x} \triv_{H \cap xHx^{-1}} \Ind^H \]
where the sum is over a set of representatives for the double cosets $H \backslash G / H$.
If $H$ is normal in $G$ then the right hand side is $[G:H] \triv_H$, as required. Otherwise there exists $x$ such that $H \cap xHx^{-1} \not =H$, and since $\langle \triv \Ind_{H \cap xHx^{-1}}^H, \triv_H \rangle = 1$, there exists a non-trivial irreducible character~$\chi$ of $H$ in the right-hand side.
\end{proof}

\begin{proposition}\label{prop: H is normal iff redundant}
    The subgroup $H$ of 
    $G$ is normal if and only if the union defining $\Theta$ is completely redundant.
\end{proposition}

\begin{proof}
 By~\eqref{eq: Theta as intersection of idealizers} we have
    \[
    \Theta(\eta_H) =
        \RId_{\C[G]}(\C[G]\eta_H) \cap \RId_{\C[G]}(0)\\
        = \RId_{\C[G]}(\C[G]\eta_H).
    \]
    Suppose that $H$ is normal in $G$. Then $\eta_H$ is central in $\C[G]$ (i.e.~$\eta_H x = x\eta_H$
    for all $x \in \C[G]$) and so 
    $\RId_{\C[G]}(\C[G]\eta_H) = \C[G]$. Clearly $\C[G]$ contains every weak lumping algebra $\Theta(e)$.
    
    Conversely, suppose that $H$ is not normal in $G$. There exists $\chi\in \Irr(H)$
    with $\chi \ne \triv_H$ such that $\langle\chi,\triv_H\Ind_H^G\Res_H\rangle \ne 0$. Let $U$ be a left ideal of $\C[H]$ whose character (as a $\C[H]$-module) is $\chi+\triv_H$, and let $L = \C[G]U$. By Theorem \ref{thm: containment Thetas}, we deduce
    that
$\Theta(e) \not\subseteq \Theta(\eta_H)$.
\end{proof}

We now consider the other end of the spectrum where the union defining~$\Theta$ is irredundant,
proving Proposition~\ref{prop: irredundant}.
The following definition is required.

\begin{definition}\label{def: full induction restriction}
   We say that the subgroup $H$ of $G$ has \emph{full induction restriction in $G$} if the restriction to $H$ of the permutation character of $G$ acting on the cosets of $H$ contains every irreducible character of $H$. That is, if $\langle\chi,\triv_H\Ind_H^G\Res_H\rangle \ne 0$ for all $\chi \in \Irr(H)$.
\end{definition}

When the group $G$ is clear from context we abbreviate this to `$H$ has full induction restriction'.
We use the following lemma, which may be regarded
as a form of Frobenius reciprocity (see Proposition~\ref{prop:FrobeniusReciprocity})
for left ideals of group algebras.

\begin{lemma}\label{lem: frobenius for isotypic blocks}
    Let $U$ and $\tilde{U}$ be left ideals of $\C[H]$. Let $V$ be an irreducible representation of $G$ and let $\isoblock{V}$ denote its Wedderburn block.
    Then, 
   \[ \C[G]U\cap \isoblock{V} = \C[G]\tilde{U}\cap \isoblock{V}\]
    if and only if 
    \[
    U\cap \pi_H\bigl( \isoblock{V}\bigr) = \tilde{U}\cap \pi_H\bigl(\isoblock{V}\bigr).\]
\end{lemma}

\begin{proof}
The intersection $U \,\cap\, \tilde{U}$ is a left ideal of $\C[H]$. 
Since $\C[H]$ is completely reducible, it has a complement in $U$, call it $X$, 
and also a complement in $\tilde{U}$, call it $\tilde{X}$.
Thus
\[  U = (U \cap \tilde{U}) \oplus X
\quad\text{ and }\quad
   \tilde{U} = (U \cap \tilde{U}) \oplus \tilde{X},\]
where all the summands are left ideals of $\C[H]$.
It follows that
$ U + \tilde{U} = (U\cap \tilde{U}) \oplus X \oplus \tilde{X}$. 
We have
\begin{align*} \C[G]U \cap\isoblock{V}  & =   \big(\C[G](U \cap \tilde{U})\cap\isoblock{V}\big) \oplus \big(\C[G]X\cap\isoblock{V}\big)\\
=\,  \C[G]\tilde{U}\cap\isoblock{V}  & =   \big(\C[G](U \cap \tilde{U})\cap\isoblock{V}\big) \oplus \big(\C[G]\tilde{X}\cap\isoblock{V}\big)
\end{align*}
and
\[ \begin{split} \C[G](U + \tilde{U})\cap\isoblock{V} &= \big(\C[G](U \cap \tilde{U})\cap\isoblock{V}\big)\\[-1pt]  &\qquad\oplus \big(\C[G]X\cap\isoblock{V}\big) \oplus \big(\C[G]\tilde{X}\cap\isoblock{V}\big).
\end{split} \]
But for any vector subspaces $A, B, C$ of a common ambient vector space, if $A \oplus B = A \oplus C$ and $A \oplus B \oplus C$ is a direct sum then $B = C = \{0\}$. We deduce that 
$$\C[G]X\cap\isoblock{V} = 0 = \C[G]\tilde{X}\cap\isoblock{V}.$$
Therefore, $X\cap \pi_H\bigl( \isoblock{V}\bigr) = X'\cap\pi_H\bigl(\isoblock{V}\bigr) = 0$ by Frobenius reciprocity.
Since \smash{$\pi_H\bigl(\isoblock{V}\bigr)$} is a direct sum of full Wedderburn components of $\C[H]$, and since projection to full Wedderburn components respects direct sums of left-modules,
we conclude that
\[
U\cap\pi_H\bigl(\isoblock{V}\bigr) = \tilde{U}\cap\pi_H\bigl(\isoblock{V}\bigr) = (U\cap \tilde{U})\cap\pi_H\bigl( \isoblock{V}\bigr). 
\]
The converse is shown similarly. Indeed, from $U\cap \pi_H(\isoblock{V}) = \tilde{U}\cap \pi_H(\isoblock{V})$ we deduce $X\cap \pi_H(\isoblock{V}) = 0 = \tilde{X}\cap \pi_H(\isoblock{V})$ and
the conclusion now follows from Frobenius reciprocity.
\end{proof}

We are now ready to prove Proposition \ref{prop: irredundant}, whose statement we recall below.
\setcounter{tmp}{\value{theorem}}
\setcounter{section}{1}
\setcounter{theorem}{\value{irredundant}}
\begin{proposition}
The subgroup $H$ of $G$ has full induction restriction if and only if the union defining $\Theta$ is irredundant, in the sense
that no set $\Theta(e)$ is contained in another.
\end{proposition}
\setcounter{theorem}{\value{tmp}}
\setcounter{section}{7}

\begin{proof}[Proof of Proposition \ref{prop: irredundant}]
    Suppose that $H$ has full induction restriction in~$G$. Suppose that there is a containment $\Theta(e)\subseteq\Theta(\tilde{e})$ between weak lumping algebras. Let $U = \C[H]e$ and $\tilde{U} = \C[H]\tilde{e}$.
    Let $V$ be an irreducible $\C[H]$-submodule of $U^\circ+\tilde{U}^\circ$. 
    Recall that $\chi_V$ denotes the character of $V$. Since $H$ has full induction restriction,
    Frobenius reciprocity implies that $\langle\chi_V\Ind_H^G, \triv_H\Ind_H^G\rangle\ne0$.
    Thus Theorem \ref{thm: containment Thetas} gives
    \[
    \C[G]U^\circ\cap\isoblock{V} = \C[G]\tilde{U}^\circ\cap\isoblock{V}.
    \]
    We can now apply Lemma \ref{lem: frobenius for isotypic blocks} to get
    \[
    U^\circ\cap\pi_H(\isoblock{V}) = \tilde{U}^\circ\cap\pi_H\bigl(\isoblock{V}\bigr).
    \]
    In particular, since $V\subseteq\pi_H(\isoblock{V})$, we have $U^\circ\,\cap\, V 
    = \tilde{U}^\circ\,\cap\, V$. This holds for all irreducible $\C[H]$-submodules $V$ of $U^\circ+\tilde{U}^\circ$, and therefore $U^\circ=\tilde{U}^\circ$. Hence, $\Theta(e) = \Theta(\tilde{e})$.

    For the converse, we show the contrapositive statement. Suppose that~$H$ does not have full induction restriction in $G$. Then there exists $\chi\in\Irr(H)$ such that $\langle\chi,\triv_H\Ind_H^G\Res_H\rangle = 0$. Let $U$ be a left ideal of $\C[H]$ whose character (as a $\C[H]$-module) is $\chi+\triv_H$, let $L = \C[G]U$. 
    Note that the character of $L^\circ$ is $\chi\Ind_H^G$ and that $\langle\chi\Ind_H^G,\triv_H\Ind_H^G\rangle = 0$ by Frobenius reciprocity.
    Then, Theorem~\ref{thm: containment Thetas} gives
    \(
    \Theta(e) \subseteq \Theta(\eta_H)
    \)
    and the union defining $\Theta$ is not irredundant.
\end{proof}

Some of the natural applications of our work are to subgroups with full induction restriction.

\begin{samepage}
\begin{example}\label{eg: full induction restriction}\,

    \begin{enumerate}
        \item[(1)] Let $H = \Sym_{\{1, ..., k\}}$ and let $G = \Sym_n$.
        As in Examples~\ref{eg: regular rep of S3} and~\ref{eg: induction Sym3 to Sym4},
        we denote by $\chi^\lambda$ the irreducible character 
        of the symmetric group canonically \smash{$S^\lambda$}. Then, \smash{$\langle\chi^\lambda\Ind_H^G,\triv_H\Ind_H^G\rangle$} is the number of partitions of $n$ that contain both $\lambda$ and $(k)$ as subpartitions.
         This constraint is strongest for $\lambda = (1^k)$; we need partitions containing $(k,1^{k-1})$, 
         and such partitions exist if and only if $n\ge 2k-1.$
               Thus $H$ has full induction restriction if and only if $2k \le n+1$.

        \item[(2)] Let $D_{2n}$ be the dihedral group of order $2n$, generated by $\sigma$ and $\tau$ as in Example \ref{periodic example}. Then $\langle\tau\rangle$ has full induction restriction and $\langle\sigma\rangle$ does not. 
    \end{enumerate}
\end{example}
\end{samepage}

    Both examples of full induction restriction from Example \ref{eg: full induction restriction} 
    are explained by the following lemma. For instance, in (1), if $2k \le n-1$ then the permutation
    $(1,k+1)(2,k+2) \ldots (k-1,2k-1) \in \Sym_n$ satisfies
    $g^{-1}Hg \cap H = \{ 1\}$, and so by~\eqref{eq:doubleCosetCountTxH}, $|HgH| = |H|^2$.

\begin{lemma}\label{lem: maximal double coset implies full ind res}
     If there is a double coset $HxH\in H\backslash G/ H$ of the maximum possible size $|H|^2$, then
     the subgroup~$H$ of $G$ has full induction restriction.
\end{lemma}
\begin{proof}
  By~\eqref{eq:doubleCosetCountTxH} we have
    \(
    |HxH| \cdot |xHx^{-1}\cap H| = |H|^2
    \)
    for every double coset $HxH\in H\backslash G / H$. If $|HxH| = |H|^2$, we deduce $|xHx^{-1}\cap H| = 1$. Now, Mackey's rule (Lemma~\ref{lemma:MackeyTrivial}) gives
    \[
    \triv_H\Ind_H^G\Res_H = 
    \triv_H\Ind_{xHx^{-1}\cap H}^H + \cdots,
    \]
    which by the previous computation shows that the regular representation of~$H$ appears as a summand in $\triv_H\Ind_H^G\Res_H$.
\end{proof}

We remark that there is a double coset $HgH$ of the 
maximum possible size $|H|^2$ if and only if $H$ has a regular orbit on the set $G/H$ of left cosets,
and so if and only if $G$ has a base (see \cite[\S 4.13]{CameronPermutationGroups}) of
size $2$ when regarded as a permutation group on $G/H$.
The example $G = \Sym_6$ and  $H = \langle(12)(34), (12)(3456)\rangle$ 
shows that full induction restriction may hold 
even when there is no double coset of maximum size. One can check this computationally.

\section{Real idempotents and a refinement of Theorem~\ref{thm:mainGL}}
\label{sec:real}

\subsection{Weak lumping algebras of real idempotents suffice}
Since a random walk on a group is driven by a weight whose coefficients are non-negative
real numbers, 
to understand the set of weakly lumping weights we must consider the real parts of the algebras $\Theta(e)$ for $e \in \Eb{H}$. 
Notice that distinct 
weak lumping algebras may have equal real parts. 
For instance, take a subgroup $H$ of order $3$, with notation
as in the proof of Corollary~\ref{cor:HOrder2Or3}. Since complex conjugate is a multiplicative map, an elementary computation gives
\begin{align*}
\Theta(\xi_H+\eta_H) \cap \R[G] 
&= \Theta(\overline\xi_H+\eta_H) \cap \R[G].
\end{align*}
Since we are really interested in $\Delta \cap \Theta$, where $\Delta$ is the simplex of probability vectors in $\C[G]$, the displayed equation above shows 
a way in which our 
expression $\Delta \cap \bigcup_{e \in \Eb{H}} \Theta(e)$ for the set of weakly lumping weights
may be redundant. (This is different to the redundancy due to containment studied in the previous section.)
In this section we remedy this form of redundancy.
We say that a left ideal of a complex group algebra is \emph{self-conjugate} if it is equal to its own complex conjugate. 

\begin{lemma}\label{lem:realIdempotents}
Every self-conjugate left ideal of $\C[H]$ is generated as a left ideal by a real idempotent of $H$. 
\end{lemma}
\begin{proof}
Let $L$ be a self-conjugate left ideal of $\C[H]$. Consider the orthogonal projection $\pi: \C[H] \to L$ with respect to the Hermitian inner product on $\C[H]$. Because $L$ is invariant under complex conjugation, $\pi$ is equivariant with respect to complex conjugation. In particular, $\pi(h) \in \R[H]$ for every $h \in H$. Hence the idempotent 
\smash{$f = \frac{1}{|H|} \sum_{h \in H} h^{-1}\pi(h)$} is real, and $L = \C[H]f$, as explained in \S\ref{subsec:idempotents}.
\end{proof}
\begin{lemma}\label{lemma:real}
Let $e \in \Eb{H}$ and let $\overline{e}$ be the complex conjugate of $e$. Let $w$ be a weight on $G$. 
Then $w$ lumps weakly with stable ideal $\C[G]e$ if and only if $w$ lumps stably for the ideal $\C[G]\overline{e}$, and in this case $w$ also lumps stably for the ideals $\C[G]e \cap \C[G]\overline{e}$ and $\C[G]e + \C[G]\overline{e}$.  Moreover, there exist real idempotents $e^{\wedge}, e^{\vee} \in \Eb{H}\cap \R[H]$ such that $\C[G]e \cap \C[G]\overline{e}  = \C[G]e^{\wedge}$ and $\C[G]e + \C[G]\overline{e} = \C[G]e^{\vee}$. 
 \end{lemma}

\begin{proof}
 Since complex conjugation is multiplicative,
 it is clear from~\eqref{eq: Theta as intersection of idealizers}
 that
  $\Theta(\overline{e}) = \overline{\Theta(e)}$. We have seen that $w$ lumps stably for $\C[G]e$ if and only if $w \in \Theta(e)$, and since $w$ takes real values 
  this holds if and only if $w \in \Theta(\overline{e})$, so if and only if $w$ lumps stably for $\C[G]\overline{e}$. 

We have seen in \S\ref{subsec:minimalGLideal} and \S\ref{subsec:Jw} that the set of weakly lumping Gurvits--Ledoux ideals for a given weight $w$ that contain $\eta_G$ forms a lattice under intersection and addition.  So if $w$ lumps stably for $\C[G]e$ then $w$ also lumps stably for $\C[G]e \cap \C[G]\overline{e}$ and for $\C[G]e + \C[G]\overline{e}$. It remains to show that these two ideals (which are invariant under complex conjugation) are generated by idempotents in $\Eb{H} \cap \R[H]$.

Apply Lemma~\ref{lem:realIdempotents} to express $\C[H]e \cap \C[H]\overline{e} = \C[H]e^{\wedge}$ and $\C[H]e + \C[H]\overline{e} = \C[H]e^{\vee}$ for real idempotents $e^{\wedge}$ and $e^{\vee}$ in $\C[H]$. We have $\eta_H \in \C[H]e^{\vee}$ and $\eta_H \in \C[H]e^{\wedge}$ since $\eta_H \in \C[H]e$ and $\eta_H \in \C[H]\overline{e}$. Hence $e^\wedge, e^{\vee} \in \Eb{H} \cap \R[H]$. Now $$\C[G]e \cap \C[G]\overline{e} = \C[G](\C[H]e \cap \C[H]\overline{e}) = \C[G]\C[H]e^\wedge = \C[G]e^\wedge.$$ To see the first equality, suppose that $x \in \C[G]e \cap \C[G]\overline{e}$, and pick a set $g_1, \dots, g_m$ of coset representatives for $G/H$; then $x$ may be written uniquely as $\sum g_i y_i e$ with each $y_i \in \C[H]$, and also uniquely as $\sum g_i z_i \overline{e}$ with each $z_i \in \C[H]$, and for each $i$ we obtain $z_i \overline{e} = y_i e \in \C[H]e \cap \C[H]\overline{e}$.
 A similar argument shows that $\C[G]e + \C[G]\overline{e} = \C[G]e^\vee$.
\end{proof}
\begin{corollary}\label{cor:real} We have
  \[  \Theta \cap \R[G] = \bigcup_{e \in \Eb{H} \,\cap\, \R[H]} \Theta(e)\cap \R[G].\]
\end{corollary}
\begin{proof}
Combine Theorem~\ref{thm:mainGL} with Lemma~\ref{lemma:real}.
\end{proof}
  It follows that
 the  set of weakly lumping irreducible weights may be expressed as 
\begin{equation}\label{eq:realSimplex} 
\Gamma \cap \Delta \cap \Theta = \Gamma \cap \Delta \cap \bigcup_{e \in \Eb{H} \,\cap\, \R[H]} \Theta(e)
\end{equation}
where, as in the introduction, $\Gamma = \C[G] \setminus \bigcup_{K \lneq G} \C[K]$ and $\Delta$ is the simplex of probability vectors in $\C[G]$. 

\subsection{A probabilistic characterisation of stable lumping for self-conjugate left ideals}
Definition~\ref{defn:lumpsStably} defines in algebraic terms what it means for an irreducible weight $w$ to lump weakly with stable ideal $L = \C[G]e$ for $e \in \Eb{H}$. As remarked after that definition, in this case it follows from Theorem~\ref{thm:mainGL} that for every weight $\alpha \in L$ we have
\begin{defnlistE}
\item $X = \MC(\alpha, w)$ lumps weakly to $G/H$ for all $\alpha \in L$, \emph{and}
\item
for any initial distribution $\alpha \in L$ and all 
$t$,
$L$ always contains the conditional distribution of $X_t$ given 
$X_0H, \ldots, X_tH$.
\end{defnlistE}
In this section we show a converse: for a self-conjugate left ideal $L \subseteq \C[G]$ containing $\eta_G$, if conditions~(a) and (b) hold for all $\alpha \in \Delta \cap L$ then $L$ is induced from $\C[H]$ and $w$ lumps weakly to $G/H$ with stable ideal $L$.  
\begin{lemma}\label{lem:realBasis}
Let $L$ be a self-conjugate left ideal of $\C[G]$ containing $\eta_G$. Then as a $\C$-vector space, $L$ has a basis whose elements are probability vectors.
\end{lemma}
\begin{proof}
By Lemma~\ref{lem:realIdempotents} (applied to $G$ in place of $H$) we have $L = \C[G]e$ for some real idempotent $e \in E(G)$, and since $\eta_G \in L$, we have $e \in \Eb{G}$. Therefore $L$ is spanned by the elements $ge$ for $g \in G$, which are weights; some subset of these forms a basis of $L$.   
\end{proof}

\begin{proposition}\label{prop:probcharrealIdeals}
Let $L$ be a self-conjugate left ideal of $\C[G]$ containing at least one non-zero probability vector 
and let $w$ be an irreducible weight on~$G$. If for every probability vector $\alpha \in L$ we have
\begin{thmlistE}
\item $X = \MC(\alpha, w)$ lumps weakly to $G/H$ for all $\alpha \in L$, \emph{and}
\item
for any initial distribution $\alpha \in L$ and all $t$,
$L$ always contains the conditional distribution of $X_t$ given 
$X_0H, \ldots, X_tH$,
\end{thmlistE} then $L=\C[G]e$ for some $e \in \Eb{H}$, and $w$ lumps weakly to $G/H$ with stable ideal $L$. 
\end{proposition}
\begin{proof}
Since $w$ is irreducible and $L$ contains at least one non-zero weight, by the usual ergodic averaging argument we have $\eta_G \in L$. By Lemma~\ref{lem:realBasis} we may choose a basis $v_1, \dots, v_k$ of $L$ consisting of probability vectors. 
Let $V$ denote the real vector space spanned by $v_1, \dots, v_k$. Identifying $\C[G]$ with $\C^G$, we have $V = L \cap \R^G$. Let $P_w$ denote the transition matrix of the left-invariant random walk on $G$ driven by $w$. By  conditions (a) and (b), for all 
probability vectors $\alpha \in L$, applying Lemma~\ref{lem:probcharStable}, $P_w$ lumps weakly under the map $\lambda: G \to G/H$ with stable space $V$. Thus we have $v_i P \in V$ for $i=1, \dots, k$, so $Lw \subseteq L$. Similarly, $v_i\Pi_{bH} \in V$  for all $b \in G$ and $i=1, \dots, k$. Extending scalars (and writing the projections to cosets on the left), this means $\pi_{bH}(L) \subseteq L$ for all $b \in G$, hence $L = \bigoplus_{b \in G/H} \pi_{bH}(L)$. Since $L$ is also a left ideal, it is an induced ideal from the
self-conjugate left ideal $\pi_H(L)$ of $\C[H]$, which contains $\eta_H$. Hence $L = \C[G]e$ for some real 
idempotent $e \in \Eb{H}$. We have $V^\circ P_w \Lambda = 0$, where $\Lambda$ is the matrix
for the canonical map $G \rightarrow G/H$ defined at the start of \S\ref{subsec:applicationGL}.
Because $\Lambda$ is real, we have $L^\circ = V^\circ \otimes_{\R} \C$, and so $L^\circ P_w \Lambda = 0$, i.e.~$L^\circ w \subseteq L^\circ$.   
\end{proof}

\noindent It follows from Proposition~\ref{prop:lumpsStably} that for a real idempotent $e \in \Eb{H} \cap \R[H]$, and any weight~$w$ on $G$,  $w$ lumps stably for $\C[G]e$ in the sense of Definition~\ref{defn:lumpsStably} if and only if $P_w$ lumps stably for $\R[G]e$ in the sense of Definition~\ref{def: stable}.  In general, not every stable space $V \subseteq \R[G]$ for the transition matrix $P_w$ and the lumping map $f : G \rightarrow G/H$ need be a left ideal. 
However, if $w$ is irreducible and $V \subseteq \R[G]$ is any subspace such that $P_w$ lumps weakly under~$f$ with stable space $V$, then the space $\sum_{g \in G} gV$ is a left ideal of $\R[G]$ and by Lemma~\ref{L: lattice of stable spaces} it is also a stable space for $P_w$ and $f$. Then \smash{$L = (\sum_{g \in G} gV) \otimes_\R \C$} satisfies the hypotheses of 
Proposition~\ref{prop:probcharrealIdeals}. In summary, we have proved that every 
stable subspace of $\mathbb{R}^G$ for weak lumping of $P_w$ under $f$ is a subspace of a stable ideal generated by a real idempotent in $\Eb{H}$.

\section{Hecke algebras and the proof of Theorem~\ref{thm:mainTransitionMatrices}}
\label{sec: proof of mainTransitionMatrices} 

To make the proof of Theorem~\ref{thm:mainTransitionMatrices} self-contained, we begin by briefly reviewing the essential theory of orbital matrices
and Hecke algebras. We give an example in \S\ref{subsec:weakLumpingAlgebra} in the context of
 the extended dice rolling example in \S\ref{subsec: abelian example}.
For further background on Hecke algebras see \cite[\S 1.11, \S 2.2, \S 3.1]{CameronPermutationGroups}
or \cite[Chapter 4]{CeccheriniSilbersteinScarabottiTolliHarmonicAnalysis}.
The permutation representation of $G$ acting on the left cosets $G/H$
was characterised as an induced representation 
in Example~\ref{ex:inductionCharacterisation}.

\subsection{Orbital matrices}\label{subsec:orbitalMatrices}
The left action of $G$ on the set of left cosets $G/H$
induces an action of $G$ on $G/H \times G/H$ by $x(gH, g'H) = (xgH, xg'H)$. 
Throughout this section, we set $m = |G/H|$.

\begin{definition}\label{defn:orbitalMatrix} Let $x \in G$.
The \emph{orbital matrix} corresponding to the double coset $HxH$ is the $m \times m$ matrix
$M(HxH)$ with zero/one entries defined by 
$M(HxH)_{(gH, g'H)} = 1$ if and only if 
$(gH, g'H)$ is in the orbit of $G$ on $G/H \times G/H$ containing $(H, xH)$.
\end{definition}

Thus there is one orbital matrix for each 
double coset $HxH$ and the linear span of the orbital matrices has dimension equal to the number
of double cosets, namely $|H \backslash G / H|$. By the $G$-invariance property in the definition,
each orbital matrix defines a $\C[G]$-endomorphism of the $m$-dimensional permutation representation
\smash{$\C \ind_H^G$}. Observe that the row of the orbital matrix $M(HxH)$ labelled by $H$ has its non-zero entries
precisely in the columns $hxH$ for $h\in H$, corresponding to the orbit of $H$ on $G/H$ containing $xH$.

\subsection{Hecke algebras}
\label{subsec:HeckeAlgebras}
The Hecke algebra of functions on $\C[G]$ invariant under left- and right-multiplication by $H$
is isomorphic to the subalgebra $\eta_H \C[G] \eta_H$ by the map sending the function $f : G \rightarrow \C$ to
$\sum_{g \in G} f(g) g$. (Note this 
is the same way that we identify weights and distributions with elements of $\C[G]$.)
It is clear that $\eta_H \C[G]\eta_H$ has as a basis all $\eta_H x \eta_H$ for $x$
in a set of representatives for the double cosets $H \backslash G / H$. In particular
$\dim \eta_H \C[G]\eta_H = |H \backslash G / H|$.
We may therefore refer to $\eta_H \C[G] \eta_H$ as the \emph{Hecke algebra of double cosets}.

In the following proposition,
recall that the \emph{opposite algebra} of an algebra~$A$ is the algebra with the same
underlying vector space, but with multiplication defined by $a \cdot b = ba$ for $a$, $b\in A$.

\begin{proposition}\label{prop:Hecke}
For $x \in G$, let $m_x = |H / H \cap xHx^{-1}|$.
The subspace of the algebra $\Mat_m(\mathbb{C})$ spanned by the orbital matrices is closed under multiplication and is isomorphic to the opposite algebra of 
$\eta_H \C[G] \eta_H$ by the map $\eta_H x \eta_H \mapsto M(HxH)/m_x$.
\end{proposition}

\begin{proof}
See Proposition 4.2.1 and its proof
in \cite{CeccheriniSilbersteinScarabottiTolliHarmonicAnalysis}.
\end{proof}

Since the proof in \cite{CeccheriniSilbersteinScarabottiTolliHarmonicAnalysis} comes only
after a long development of theory not required in this paper, we 
outline a shorter proof: the action
of $\eta_H \C[G] \eta_H$ 
by right multiplication $\C[G] \eta_H$ defines an injective algebra homomorphism 
$\eta_H \C[G] \eta_H \rightarrow \End_{\C[G]} (\C[G] \eta_H)$. By Lemma~\ref{lemma:MackeyTrivial}, 
$\dim \End_{\C[G]} (\C[G] \eta_H) = 
|H \backslash G / H|$. Hence 
this map is an isomorphism. In the canonical basis of left cosets of $\C[G]\eta_H$, the matrix
of the endomorphism~$F_x$ determined by $F_x(\eta_H) = \eta_H x \eta_H$ 
is stochastic, having $m_x$ entries of $1/m_x$ in each row, where $m_x = |H / H \cap xHx^{-1}|$ 
is the size of the orbit of $H$ on $G/H$ containing~$xH$.
By the final paragraph of \S\ref{subsec:orbitalMatrices},
$\End_{\C[G]} (\C[G] \eta_H)$
has as a basis the orbital matrices, acting by \emph{left} multiplication;
the matrix $M(HxH)$ having $m_x$ ones in each row corresponds to the endomorphism $m_x F_x$.
Hence $\eta_H x \eta_H \mapsto M(HxH)/m_x$ is an explicit isomorphism between $\eta_H \C[G]\eta_H$ and the opposite
algebra of the algebra of orbital matrices.

\subsection{Proof of Theorem~\ref{thm:mainTransitionMatrices}}
We need one final preliminary: suppose that $Q$ 
is an $m \times m$
matrix satisfying the condition that $Q_{(gH, g'H)} = Q_{(kgH, kg'H)}$
for all $g, g', k \in G$. By the final sentence of \S\ref{subsec:HeckeAlgebras},
\begin{equation}\label{eq:orbital}
Q = \sum_{x} Q_{(H, xH)} M(HxH) \end{equation}
where the sum is over a set of representatives for the double cosets
$HxH$.

\begin{proof}[Proof of Theorem~\ref{thm:mainTransitionMatrices}]
Suppose that (i) holds so, by hypothesis $\MC(\eta_G,w)$ lumps weakly to $G/H$
and the transition matrix of the lumped chain is $Q$. We may
assume that $w$ is normalized, i.e.~$w(G) = 1$. By hypothesis
the initial distribution of $X_0$ is uniform. Therefore, conditioned
on the event $X_0 \in gH$, the distribution of $X_0$ is uniform on $gH$. Hence
for any $g' \in G$ we have 
\begin{align*} \P[X_1 \in g'H \mid X_0 \in gH] 
&= \frac{1}{|H|} \sum_{h \in H} \P[X_1 \in g'H \mid X_0 = gh] \\
&= \frac{1}{|H|} \sum_{h \in H} \sum_{h' \in H} w(h^{-1}g^{-1} g'h').\end{align*}
The right hand side is the sum of the coefficients
of $|H| \eta_H w \eta_H$ on the double coset $Hg^{-1}g'H$.
Therefore the probability just calculated is the same replacing
$w$ with $\eta_H w \eta_H$, and we have~(ii), that $Q$ is the transition
matrix of the induced random walk driven by a weight in the Hecke algebra
$\eta_H \C[G] \eta_H$.

Suppose that (ii) holds. Then by Proposition~\ref{prop:Hecke},
$Q$ is a linear combination of the orbital matrices $M(HxH)$. 
It is immediate from Definition~\ref{defn:orbitalMatrix}
that these matrices satisfy $M(HxH)_{(gH,g'H)} = 1$ if and only
if $M(HxH)_{(kgH,kg'H)} = 1$ for all $g, g', k \in G$.
Therefore (iii) holds. Conversely, if (iii) holds then 
by~\eqref{eq:orbital} and 
Proposition~\ref{prop:Hecke}, 
a suitable weight satisfying~(ii) is 
$\sum_{x} Q_{(H,xH)} \eta_H w \eta_H$.

To complete the proof it suffices to show that (ii) implies (iv)
and (iv) implies (i). Suppose that~(ii) holds, so that $w \in \eta_H \C[G] \eta_H$.
Then $w(hgh') = w(g)$ for all $h, h' \in H$ and so
$w$ is constant on left cosets $gH$ in the same double coset
$HxH$, and dually, $w(Hg)$ is constant on right cosets $Hg$ in the same double
coset $HxH$. Hence $w$ satisfies the conditions of Corollary~\ref{cor:strongExact}
to lump strongly and exactly to $G/H$, as required for (iv).
Finally (iv) implies~(i) because strong lumping implies weak lumping.
\end{proof}

\section{Duality and time-reversal}
\label{sec: time reversal}

In this section we state and prove a new theorem that relates time reversal and duality for weak lumping of general finite Markov chains. We then derive Theorem~\ref{thm:mainTimeReversal} as an application, and finally deduce Corollary~\ref{cor:strongExact}.

\subsection{Duality and time-reversal for weak lumpings of general finite Markov chains}\label{subsec:timeReversal}
Let $X = \MC(\alpha, P)$ be a DTHMC, and suppose that it is stationary,~i.e. $X_t$ is distributed
according to $\alpha$ for all times $t$.
 We may extend it to have time indexed by $\mathbb{Z}$. The time reversal of the extended chain is another stationary DTHMC which we shall denote $X^\star$. It also has time indexed by $\mathbb{Z}$ and stationary distribution $\alpha$. 
Let $A \subseteq G$ be the support of $\alpha$. Then,
in the notation from the start of \S\ref{sec: probabilityPreliminaries}, the time reversed chain 
defined on $A$
is $\MC(\alpha, P^\star)$, where $$P^\star(x,y) = \frac{\alpha(y)}{\alpha(x)}P(y,x).$$ 
In other words, if $D$ is the diagonal matrix indexed by $A$ whose $i^{th}$ diagonal entry is $\alpha(i)$, then $DP = (DP^\star)^T$. In particular, if $\alpha$ is the uniform distribution on $A$ then $P^\star = P^T$.

Now suppose $f: A \to B$ is a surjective function. Then $X^\star$ lumps weakly under $f$ if and only if $X$ does, since if the process $f(X)$ is a time-homogeneous Markov chain (necessarily stationary), then its time reversal is also a time-homogeneous Markov chain. We shall see in Corollary~\ref{cor:ExactStrongDuality} that exact lumping and strong lumping are exchanged by time reversal, under the mild condition that the stationary distribution $\alpha$ has full support.  This is a special case of the following duality statement.
See Definition~\ref{def: stable} for the definition of stable space.
\begin{theorem}[Duality]\label{thm:DualityForMarkovChainLumping}
  Let $\alpha$ be any stationary probability distribution for $P$ such that $\alpha(x) > 0$ for every $x \in A$. Let the time reversal of $\MC(\alpha, P)$ be $\MC(\alpha, P^\star)$. Suppose that $P$ lumps weakly under $f$ with stable space $V$, and $\alpha \in V$.  If $V$ is a real vector space, then define 
  \[ W = \left\{ w \in \mathbb{R}^A:\text{for all $v \in V^\circ$ we have}\sum_{x \in A} \frac{v(x) w(x)}{\alpha(x)} = 0 \right\},\] 
  and if $V$ is a complex vector space, then define  $$W = \left\{ w \in \mathbb{C}^A:\text{for all $v \in V^\circ$ we have}\sum_{x \in A} \frac{\overline{v(x)} w(x)}{\alpha(x)} = 0 \right\}.$$
   Then $\alpha \in W$ and the time reversal $P^\star$ lumps weakly under $f$ with stable space $W$.
   \end{theorem}
  \begin{proof}
First let us check that $\alpha \in W$. For every $v \in V^\circ$ we have $$\sum_{x \in A} v(x) = \sum_{b \in B} \sum_{x \in f^{-1}(b)} v(x) = \sum_{b \in B} 0 = 0,$$ as required.
We now check that $P^\star$ lumps weakly under $f$ with stable space~$W$, by checking the four conditions in Definition~\ref{def: stable}. 
\begin{defnlist}
\item[(a)] $W$ contains the probability vector $\alpha$.
\item[(b)] To show that $WP^\star \subseteq W$, let $w \in W$. We shall check that $wP^\star \in W$. For any $v \in V^\circ$ we have $vP \in V^\circ$ because $V^\circ P \subseteq V^\circ$, so
\begin{multline*} \quad\quad\quad \sum_{x \in A} \frac{\overline{v(x)}(wP^\star)(x)}{\alpha(x)} = \sum_{x,y \in A}\frac{ \overline{v(x)} P^\star(y,x) w(y)}{\alpha(x)} \\ = \sum_{x,y \in A} \frac{\overline{v(x)} P(x,y) w(y) }{\alpha(y)} = \sum_{y \in A} \frac{\overline{(vP)(y)} w(y)}{\alpha(y)} = 0.\end{multline*}
\item[(c)] To check that $W \Pi_b \subseteq W$ for all $b \in B$, note that $w \in W$ if and only if $\sum_{x \in A} \frac{\overline{v(x)}w(x)}{\alpha(x)} =  0$ for all $v$ in a basis of $V^\circ$. We may choose a basis of $V^\circ$ each of whose elements is supported on a single fiber of the map $f$. Hence this orthogonality condition holds for $w$ if and only if it holds for $w\Pi_b$ for each $b \in B$.
\item[(d)] To show that $W^{\circ}P^\star \subseteq W^\circ$, where $W^\circ = W \cap \ker F$, suppose that $w \in W^\circ$. We have already checked in (b) that $wP^\star \in W$, so we must show that $wP^\star \in \ker F$. First, let us show that for every $v \in V$ we have
\begin{equation} \label{eq: w perp to V}
  \sum_{x \in A} \frac{\overline{v(x)}w(x)}{\alpha(x)} = 0.
\end{equation}
Let $v'$ be the projection of $v$ into $V^\circ$ defined by
$$v' = v - \sum_{b \in B} \frac{v(f^{-1}(b))}{\alpha(f^{-1}(b)} (\alpha\Pi_b),$$
where $v(f^{-1}(b))$ denotes $\sum_{y \in f^{-1}(b)} v(y)$ and likewise for $\alpha(f^{-1}(b))$.  Then since $w \in W$ we have
$$\sum_{x \in A} \frac{\overline{v'(x)}w(x)}{\alpha(y)} = 0,$$
and so it suffices to check that for every $b \in B$ we have
$$ \frac{\overline{v(f^{-1}(b))}}{\alpha(f^{-1}(b))} \sum_{x \in A}\frac{\overline{(\alpha\Pi_b)(x)} w(x)}{\alpha(x)} = \frac{\overline{v(f^{-1}(b)}}{\alpha(f^{-1}(b))} \sum_{x \in f^{-1}(b)} w(x) = 0. $$ This holds because $w \in \ker F$. We have proved~\eqref{eq: w perp to V}.

We wish to show that $wP^\star F = 0$, in other words that for each $b \in B$ we have $$\sum_{x \in f^{-1}(b)} (wP^\star)(x) = 0.$$
Let $b \in B$. We have 
\begin{align*} \quad\quad\sum_{x \in f^{-1}(b)} (wP^\star)(x)  &= \sum_{y \in A}\sum_{x \in f^{-1}(b)} w(y)P^\star(y,x) \\ &= \sum_{y \in A}\sum_{x \in f^{-1}(b)} w(y)P(x,y) \frac{\alpha(x)}{\alpha(y)}
\\  &= \sum_{y \in A} \frac{((\alpha \Pi_b)P)(y) w(y)}{\alpha(y)}\\ &= \sum_{y \in A} \frac{(\alpha \Pi_bP)(y) w(y)}{\alpha(y)}.\end{align*} 
This is $0$ because $ \overline{\alpha} = \alpha \in V$ and hence $\alpha\Pi_b \in V$ and $\alpha \Pi_b P \in V$, and so we may take $v = \overline{\alpha\Pi_b P} = \alpha\Pi_b P$ in~\eqref{eq: w perp to V}. This concludes the check of condition (d), and the proof of Theorem~\ref{thm:DualityForMarkovChainLumping}.
\end{defnlist}
\end{proof}

The correspondence between $V$ and $W$ is an involution. To see this, let the inner product $\langle -,-
\rangle_\alpha$ on $\mathbb{R}^A$ or $\mathbb{C}^A$ be defined by 
$$\langle v,w\rangle_\alpha = \sum_{x \in A} \frac{\overline{v(x)}w(x)}{\alpha(x)}.$$
Since $V$ and $W$ both have $\langle \alpha \Pi_B: b \in B\rangle$ as a subspace, it follows that  $V = (W^\circ)^\perp$ and $V^\circ = W^\perp$, where the orthogonal complements are with respect to the inner product $\langle -, - \rangle_\alpha$.
The following special case is worth noting.
  
  \begin{corollary}\label{cor:ExactStrongDuality}
   Let $\alpha$ be a stationary distribution for $P$ with full support. Then $\MC(\alpha,P)$ lumps strongly under $f$ if and only if $\MC(\alpha, P^\star)$ lumps exactly under $f$. 
\end{corollary}

\begin{proof} 
  Taking $V = \mathbb{R}^A$ in Theorem~\ref{thm:DualityForMarkovChainLumping}, we obtain $W = \langle \alpha\Pi_b: b \in B \rangle$. Dually, if we take $V = \langle \alpha\Pi_b: b \in B \rangle$ then $V^\circ = 0$ and we obtain $W = \mathbb{R}^A$. 
  Now apply Theorem~\ref{thm:DualityForMarkovChainLumping}.
\end{proof}

\subsection{Proof of Theorem~\ref{thm:mainTimeReversal}}
Recall that the anti-involution $\star$ on $\C[G]$ is defined by
 $x^\star = \sum_{g \in G} \overline{x(g)} g^{-1}$ for any element $x \in \C[G]$. 
Let $\perp$ denote the orthogonal complement with respect to the Hermitian inner product on $\C[G]$ 
defined in~\eqref{eq:innerProduct}.
 
\begin{lemma}\label{lem:perpandstar}
 For any idempotent $e \in E[G]$, $(\C[G]e)^\perp = \C[G](1-e^\star)$.
\end{lemma}

\begin{proof}
 Observe that $\langle v,w \rangle = 0$ if and only if the coefficient of the identity in $wv^\star$ is $0$. Hence $ev^\star  = 0$ if and only if $\langle v, ge \rangle = 0$ for all $g \in G$, if and only if $v \in (\C[G]e)^\perp$.  
 Hence \[\C[G]e^\perp = \{v \in \C[G] : ev^\star = 0\} = \{v \in \C[G]: v e^\star = 0\} = \C[G](1-e^\star),\]
 where we used that $e^\star$ is also an idempotent.
\end{proof}

  When $w$ is a weight, the stationary random walk $\MC(\eta_G, w)$ may be extended to a stationary random walk with time indexed by $\mathbb{Z}$. The time reversal of this extension is  $\MC(\eta_G,w^\star)$, similarly extended. Thus, recalling that $P_w$ denotes the transition matrix associated to $w$, we have $(P_w)^\star = P_{w^\star}$.
We are now ready to prove Theorem~\ref{thm:mainTimeReversal}, which we restate below.

\setcounter{tmp}{\value{theorem}}
\setcounter{section}{1}\setcounter{theorem}{\value{Revthm}}
\begin{theorem}[Time reversal]
Let $e \in \Eb{H}$ and let $w \in \C[G]$ be a weight. 
The left-invariant random walk on $G$ driven by $w$ weakly lumps to $G/H$ with
stable ideal $\C[G]e$ 
if and only if the left-invariant random walk on $G$ driven by $w^\star$ lumps 
stably for 
$\C[G](1-e^\star+\eta_H)$.
\end{theorem}
\setcounter{section}{10}
\setcounter{theorem}{\value{tmp}}
\begin{proof} We apply the complex case of Theorem~\ref{thm:DualityForMarkovChainLumping} to $\MC(\eta_G,w)$. The uniform distribution $\alpha = \eta_G$ is stationary and gives positive mass to every element of $A = G$. Take $V =  \C[G]e$. 
Then
$$ W =  (V^\circ)^\perp = (\C[G](e - \eta_H))^\perp = \C[G](1 - e^\star + \eta_H)$$
by Lemma~\ref{lem:perpandstar}, using that $e-\eta_H$ is an idempotent and $\eta_H^\star = \eta_H$. As mentioned in Remark~\ref{rem: stable spaces and stable ideals}, Proposition~\ref{prop:lumpsStably} shows that the left-invariant random walk driven by $w$ lumps weakly to $G/H$ with stable ideal $\C[G]e$ if and only if the corresponding transition matrix $P_w$ lumps weakly with stable space $\C[G]e$. 
\end{proof}

\begin{remark}
By the formula for centrally primitive idempotents~\eqref{eq:centralPrimitiveIdempotent},
if~$e$ is a centrally primitive idempotent then $e = e^\star$.
It follows that if $H$ is abelian then, since every idempotent is then a sum of centrally primitive
idempotents, we have $e = e^\star$ for all $e \in \Eb{H}$ and
so the stable ideal for the time reversed walk is $\C[G](1-e-\eta_H)$. 
\end{remark}

We are now ready to prove Corollary~\ref{cor:strongExact}. 

\setcounter{tmp}{\value{theorem}}
\setcounter{section}{1}\setcounter{theorem}{\value{Revcor}}
\begin{proof}[Proof of Corollary~\ref{cor:strongExact}]
Part (i) of Corollary~\ref{cor:strongExact} is the criterion for strong lumping that was proved in \cite{BW}. By the equivalence of (i) and (iv) in Proposition~\ref{prop:strongJwIsCG},
this is equivalent 
to weak lumping with stable ideal $\C[G]$; that is, the case $e=\id_H$ of Theorem~\ref{thm:mainTimeReversal}. 
Thus strong lumping occurs if and only if the left-invariant random walk on $G$ driven by $w^\star$ lumps weakly to $G/H$ with stable ideal $\C[G]\eta_H$. That is precisely the condition for the time-reversal of the stationary walk to lump exactly to $G/H$. Since $\star$ is an anti-involution, i.e.~$(xy)^\star = y^\star x^\star$ for 
all $x$, $y \in \C[G]$, we see that $w(gH)$ is constant for left cosets $gH$ in the same double coset if and only if $w^\star(Hg)$ is constant for right cosets $Hg$ in the same double coset. Part (ii) of the corollary follows, by replacing $w$ with $w^\star$.
\end{proof}
\setcounter{section}{10}
\setcounter{theorem}{\value{tmp}}

We have already seen an example of $\star$-duality and time reversal in the 
shuffles described in \S\ref{subsec:shuffles}. For a deck of $n$ cards, the random-to-top shuffle and the top-to-random shuffle are related by time reversal. Both shuffles are irreducible with the uniform distribution on $\Sym_n$ as stationary distribution. The random-to-top shuffle lumps strongly to the top card and the top-to-random shuffle lumps exactly to the top card, exemplifying Corollary~\ref{cor:ExactStrongDuality}.  For the weight~$w$ on $G = \Sym_4$ defined in equation~\eqref{eq:w123}, and $H =\Sym_{\{2,3,4\}}$, $T = \Sym_{\{2,3\}}$, we obtain that
\[w^\star = (1-\lambda) \Id + \mfrac{\lambda}{3} \bigl(
(1,4)(2,3) + (1,3,4) + (1,3,2,4) \bigr)\] lumps stably for the ideal $\C[G](1-\eta_T + \eta_H)$.
We shall see further examples of Theorem~\ref{thm:mainTimeReversal}  when we study random rotations of a six-sided die in \S\ref{subsec: abelian example}: see in particular Example~\ref{ex:popWeakDual}.

\section{Interpolating between strong and exact lumping}\label{sec:interpolating}

Our characterisation of weak lumping from \S\ref{sec: a characterisation of WL} is in terms of Gurvits--Ledoux ideals
(see Definition~\ref{defn:GLideal}), which are induced ideals in the sense of Definition~\ref{defn:inducedIdeal}. 
A natural way of constructing ideals of $\C[H]$ is by taking ideals of $\C[T]$ where $T$ is a fixed subgroup of $H$.
Let $W = \C[T]\eta_T$ be the trivial representation of $T$, and let 
\begin{align*}
    U &= \C[H]\eta_T \cong W\Ind_T^H, \\
    L &= \C[G]\eta_T \cong U\Ind_H^G \cong W\Ind_T^G.
\end{align*}

\begin{proposition}\label{prop:interpolating}
    The left ideal $L = \C[G]\eta_T$ is a weak lumping Gurvits--Ledoux ideal for $w$ with respect to left cosets of $H$ if and only if
    \begin{enumerate}[label=\textup{(\alph*)},leftmargin=32pt, topsep=3pt, ref=\textup{(\alph*)}]
		\setlength{\itemindent}{-2pt}%
        \item\label{axiom a} the random walk driven by $w$ lumps exactly to left cosets of $T$, and
        \item \label{axiom b} $w(TgH) = \displaystyle \frac{|TgH|}{|HgH|}w(HgH)$ for all $TgH\in T\backslash G/ H$.
\end{enumerate}
\end{proposition}
\noindent

\begin{proof}
    The weak lumping algebra of $L = \C[G]\eta_T$ with respect to left cosets of $H$ is
    \begin{align*}
        \Theta(\eta_T) &= \RId_{\C[G]}(L)\cap\RId_{\C[G]}(L^\circ)
        \\ &=
        \{w\in\C[G]: \eta_T w (1-\eta_T) = 0,
        (\eta_T-\eta_H) w (1-\eta_T+\eta_H) = 0\}
        \\ &=
        \{w\in\C[G]: \eta_T w (1-\eta_T) = 0,  
        (\eta_T-\eta_H) w \eta_H = 0\}.
    \end{align*}
    The first equation defines the set of \emph{exactly} lumping weights on left cosets of $T$ given in 
    Example~\ref{ex:exactThetaDim}, that is
    \[
        \{w\in\C[G]: \eta_T w (1-\eta_T) = 0\} = \Theta(\eta_T). 
    \]
    The second equation can be rewritten as $\eta_T w \eta_H = \eta_H w \eta_H$ and broken up coefficient by coefficient into a system of $|G|$ equations. By Lemma \ref{lemma:averaging}, we have
    \[
    \eta_T w \eta_H =
    \sum_{x\in G}\left(
        \frac{w(TxH)}{|TxH|}
        \right) x
    \quad\text{and}\quad
    \eta_H w \eta_H =
    \sum_{x\in G}\left(
        \frac{w(HxH)}{|HxH|}
        \right) x.
    \]
    Therefore, $\eta_T w \eta_H = \eta_H w \eta_H$ if and only if 
    \[ w(TgH) = \frac{|TgH|}{|HgH|}w(HgH)\]
    for all $g\in G$.
\end{proof}

We remark that the conditions in Proposition~\ref{prop:interpolating} may be rewritten as:

\smallskip
    \begin{enumerate}[label=\textup{(\alph*${'}$)},leftmargin=32pt, topsep=3pt, ref=\textup{(\alph*${'}$)}]
		\setlength{\itemindent}{-2pt}%
    \item\label{axiom a'} $w(Tg)$ is constant on $Tg \subseteq TxT$ for all $TxT\in T\backslash G / T$, and
    \item \label{axiom b'} $w(TgH)$ is constant on $TgH \subseteq HxH$ for all $HxH\in H\backslash G / H$.
\end{enumerate}

\smallskip\noindent
We used this form of the conditions in the example in \S 1.2.3. 

\begin{corollary}\label{cor:exactandstrongasinterpolation}
    Let $L = \C[G]\eta_T$ be a Gurvits--Ledoux ideal for $w$ with respect to left cosets of $G/H$.
    \begin{thmlist}
        \item If $T = H$, then $L$ is weakly lumping for $w$ if and only if the left-invariant random walk driven by $w$ lumps exactly to $G/H$.
        \item If $T = 1$, then $L$ is weakly lumping for $w$ if and only if the left-invariant random walk driven by $w$ lumps strongly to $G/H$.
    \end{thmlist}
\end{corollary}
\begin{proof}
When $T=H$, condition \ref{axiom a} in Proposition~\ref{prop:interpolating}
recovers the characterisation of exact lumping to $G/H$ from Corollary \ref{cor:strongExact}(ii) and condition~\ref{axiom b} becomes void.
When $T=1$, condition \ref{axiom b'} in the equivalent restatement above
becomes the characterisation of strong lumping from Corollary \ref{cor:strongExact}(i) and \ref{axiom a'} becomes 
trivial.
\end{proof}

The conclusion of this  corollary may hold even if $T$ is neither $1$ nor $H$. This is demonstrated by the following
proposition.

\begin{proposition}\label{prop:interpolatingShuffle} Let $G = \Sym_n$ acting naturally on $\{1,\ldots, n\}$.
Define subgroups
$H = \mathrm{Stab}(1) = \Sym_{\{2, \dots, n\}}$ and $T = \mathrm{Stab}(1) \cap \mathrm{Stab}(n)$ $= \Sym_{\{2, \dots, n-1\}}$. Consider the two-step
shuffle of a deck of $n$ cards:
\begin{quote}\emph{Remove the bottom card, insert it under a random card chosen uniformly 
from the remaining deck, then move the top card to the bottom.}
\end{quote}
This shuffle lumps weakly to $G/H$ with stable ideal $\C[G]\eta_T$.
\end{proposition}

\begin{proof}
Let $w$ be the normalized weight describing the shuffle. We shall use Proposition~\ref{prop:interpolating} to show that $w$ lumps weakly to $G/H$. To check that $w$ satisfies condition~(a), that $\MC(w,\eta_G)$ lumps exactly to $G/T$, we may check that the time-reversal $w^\star$ lumps strongly to $G/T$ and apply Corollary~\ref{cor:ExactStrongDuality}. The shuffle described by $w^\star$ is performed as follows:

\begin{quote} Set aside the bottom card. Pick a card
uniformly at random in the remaining deck and move it to the bottom. Put the set-aside card on top of the deck.\end{quote}
Each element of $G/T$ corresponds to a particular ordered pair 
\vspace*{-3pt}\[\text{(value of the top card, value of the bottom card).}\]

\vspace*{-3pt}\noindent
The value of the top card after the next $w^\star$ shuffle is always the current value of the bottom card. Even conditional on the complete current order of the deck, the next value of the bottom card is uniform among the $n-1$ values not equal to the current value of the bottom card. Hence $w^\star$ lumps strongly to $G/T$, as required. 

To check that $w$ satisfies condition (b), note that there are just three double cosets in $T \backslash G / H$:
\begin{itemize}
\item $H$ itself, which is also a double coset in $H \backslash G / H$;
this is the set of permutations stabilising position $1$; 
\item $(1,n)H = T(1,n)H$; this is the set of $(n-1)!$ permutations that send position $n$ to~position~$1$, 
\item $T(1,2,n)H$; this is the set of $(n-2)(n-1)!$ permutations that send
some position in $\{2, \dots, n-1\}$ to position $1$.
\end{itemize}
Under the shuffle $w$, position $1$ is certainly sent to position $n$, so $w(H) = 0$. With probability $1/(n-1)$, we insert the bottom card immediately under the top card in the deck;
then after the top card is moved to the bottom, the effect is that position $n$ is sent to $1$.
 With the remaining probability $(n-2)/(n-1)$, the card in position
 $2$ is sent to position $1$. 
 The final two probabilities are in proportion to the sizes of the two double cosets $T(1,n)H$ and $T(1,2,n)H$ forming the double coset $H(1,2)H$. This establishes condition (b). 
 \end{proof}

 This is an instance of Proposition~\ref{prop:interpolating} that is not covered by Corollary~\ref{cor:exactandstrongasinterpolation}. Indeed 
it is easy to see that $w^\star$ does not lump strongly to $G/H$ provided 
$n \ge 3$, and that $w$ does not lump strongly to $G/H$ provided $n \ge 4$. 
By Theorem~\ref{thm:mainTimeReversal}, the shuffle $w^\star$ lumps weakly to $G/H$ with stable ideal $\C[G](1 - \eta_T + \eta_H)$. 
Note that $w^\star$ does not lump exactly to $G/T$ when $n \ge 4$, so it does not provide another direct application of Proposition~\ref{prop:interpolating}.

\section{Double coset decomposition of the weak lumping algebra}\label{sec:ThetaCharacterisation}

Theorem~\ref{thm:mainGL} and our further results suggest that it is natural to study the problem of weak lumping to $G/H$ 
`double coset by double coset'.
Corollary~\ref{cor:strongExact} and Proposition~\ref{prop:interpolating} are examples of this phenomenon.
In this section, we make this explicit by proving~\eqref{eq:ThetaSplit}
and then Proposition \ref{prop:ThetaDoubleCosetCondition}, which gives 
a necessary and sufficient condition for a weight $w$ to be in $\C[HxH] \cap \Theta(e)$, 
as well as a test for this condition.  
In~\S\ref{sec:abelian}, we further develop these results for~$H$ abelian.

\subsection{Preliminaries}
Given $K\subseteq G$, 
let $\C[K]$ denote the subspace of $\C[G]$ spanned linearly by the elements $x \in K$.
Given $w \in \C[G]$ we write $w_{HxH}$ for the component of $w$ in $\C[HxH]$;
thus $w = \sum_{x \in H \backslash G / H} w_{HxH}$ is the unique expression
of $w$ as an element of $\C[G] = \bigoplus_{x \in H \backslash G / H} \C[HxH]$.

\begin{lemma}\label{lemma:ThetaDoubleCosetDecomposition}
If $e, f \in \C[H]$ and $w \in \C[G]$ then we have $ewf = 0$
if and only if $ew_{HxH} f = 0$ for each double coset $HxH \in H \backslash G / H$.
\end{lemma}

\begin{proof}
Write $w = \sum_{x \in H \backslash G / H} w_{HxH}$.
Since $\C[HxH]$ is a left $\C[H]$-module (by left multiplication) and a right $\C[H]$-module 
(by right multiplication), we have $e w f = 0$ if and only if $e w_{HxH} f = 0$
for each double coset~$HxH$.
\end{proof}

\subsection{\texorpdfstring{The module structure of $\C[HxH]$}{}}

In order to use representation theoretic tools to study $\C[HxH]$, we require a version of Mackey's rule 
(given earlier in Lemma~\ref{lemma:MackeyTrivial} in a special case) 
with an explicit isomorphism for each double coset. For this, we need to understand twisted actions. Induced
modules give the most important examples.

\begin{example}\label{example:kAction}
Observe that
in $U \ind_H^G = \bigoplus_{x \in G / H} \langle x \rangle \otimes U$,
the subspace $ \langle x \rangle \otimes U$ is a $xHx^{-1}$-module
on which $k = xhx^{-1} \in xHx^{-1}$ acts by 
\[k (x \otimes v) = x (x^{-1}x g) \otimes u = 
x \otimes x^{-1}k x u.\]
\end{example}

\begin{definition}\label{defn:characterTwist}
Let $H$ be a subgroup of $G$, and $y \in G$. 
Given a left $\C[H]$-module $M$ 
we denote by $\conj{x}{M}$ the left $\C[xHx^{-1}]$-module with the same underlying vector space as $M$
but with the action
$xhx^{-1} \cdot v = h v$
for all $h \in H$ and $v \in M$. 
\end{definition}

Note that in this definition $hv$ is defined using the original $\C[H]$-module
structure on $M$.  
Equivalently, as expected from Example~\ref{example:kAction},
\begin{equation} k \cdot v = x^{-1}kx v \label{eq:kAction}\end{equation}
for $k \in xHx^{-1}$ and $v \in M$.
Thus we have $\chi_{{}^x\!M}(xhx^{-1}) = \chi_M(h)$, or equivalently,
$\chi_{{}^x\!M}(k) = \chi_M(x^{-1}kx)$. This is, by definition, the conjugated
character $\chi^{x^{-1}}_M$. (The inverse is necessary to be consistent
with the usual notation, see for instance \cite[Ch.~6]{Isaacs}.)

In the following lemma we need the natural right action of $\C[H]$ on  $\C[HxH]$
defined by $hxh' \cdot k = hxh' k$,
and on $\conj{x}{\C[H]}$
defined by $h' \cdot k = h'k$. 
It is routine to check from~\eqref{eq:doubleCosetAllForms}
that these right actions on $\C[HxH]$ are well defined.

\begin{lemma}\label{lemma:MackeyOnHxH}
There is an isomorphism
of left $\C[H]$-modules 
\[
\C[HxH] \cong \C[H] \otimes_{\C[H\cap xHx^{-1}]} \conj{x}{\C[H]}  \]
defined by $hxh'  \mapsto h \otimes h'$.
Moreover this map commutes with the natural right actions of $\C[H]$ on $\C[HxH]$
and on $\conj{x}{\C[H]}$. 
\end{lemma}

\begin{proof}
Clearly $\C[HxH]$ contains the subspace $\XY$ spanned by the set $(H \cap xHx^{-1})x H$.
Since $xhx^{-1} x = xh$, the subspace $\XY$ has $\{ xh : h \in H\}$ as a canonical basis.
Observe that $\XY$ is a left $\C[H \cap xHx^{-1}]$-module on which $k \in H \cap xHx^{-1}$
acts by permuting the canonical basis:
\[ kxh' = x(x^{-1}kx)h'. \]
Thus, by~\eqref{eq:kAction}, $\XY$ is isomorphic to $(\conj{x}\C [H])\res_{H \cap xHx^{-1}}$ as a module for $H \cap xHx^{-1}$.
Let $m = [H : H \cap xHx^{-1}]$ and let $h_1, \ldots, h_m$ be representatives
for the left cosets $H / H \cap xHx^{-1}$. Each $g \in HxH$ may be expressed uniquely
as $g = h_ixh$ for some $h \in H$. Hence 
$\C[HxH] = \bigoplus_{i=1}^m h_i \XY$.
It now follows from the characterisation of induced modules of Proposition~\ref{prop:inductionCharacterisation}
and the definition of induction that
$\C[HxH] \cong \C[H] \otimes_{\C[H \cap xHx^{-1}]} \XY$.
An explicit isomorphism of left $\C[H]$-modules
is defined by $h_i v \mapsto h_i \otimes v$
for $1 \le i \le m$ and $v \in \XY$. Equivalently, $hxh' \mapsto h \otimes h'$ for $h, h' \in H$.
It is routine to check that this isomorphism commutes with the right action of $\C[H]$.
\end{proof}

\begin{proposition}\label{prop:MackeyOnHxHf}
Let $f \in \C[H]$. 
There is an isomorphism
of left $\C[H]$-modules 
\[
\C[HxH]f \cong \bigl( \conj{x}(\C[H]f) \Res_{H \cap xHx^{-1}} \bigr)\Ind^H \]
defined by $hxh' f \mapsto h \otimes h'f$.
\end{proposition}

\begin{proof}
Take the isomorphism in Lemma~\ref{lemma:MackeyOnHxH} and multiply each side on the right by $f$,
using that the isomorphism commutes with the right action of $\C[H]$, to get
\[
\C[HxH]f \cong \C[H] \otimes_{\C[H\cap xHx^{-1}]} \conj{x}{(\C[H]f)}, \]
where the isomorphism is defined by $hxh'f \mapsto h \otimes h' f$.
By definition, the right-hand side is the induced module in the proposition.
\end{proof}

\subsection{Double coset decomposition of right idealizers}

Let $e$ be an idempotent and $L = \C[G]e$ be a left ideal of $\C[G]$. From Lemma \ref{lemma:ThetaDoubleCosetDecomposition}, the conditions $ew(1-e) = 0$ 
defining the right idealizer of $L$ reduce to conditions on each double coset. Denoting $\RId_{\C[G]}(L) \cap \C[HxH]$ by $\RId_{\C[HxH]}(L)$, we have
\[
\RId_{\C[G]}(L) = \bigoplus_{x \in H \backslash G / H} \RId_{\C[HxH]}(L).
\]
(To be very careful, we warn the reader
that the space $\RId_{\C[HxH]}(L)$ is not a right idealizer in the usual sense of \S\ref{subsec:rightIdealizers}, 
because $L\not\subseteq \C[HxH]$.)

\begin{proposition}
\label{prop:ThetaAlgebraDoubleCosetDimension}
Fix $x \in G$
and let $e \in \C[H]$ be an idempotent $e \in \C[H]$.
Then 
\[ \dim\left(\RId_{\C[HxH]}(L)\right) =
|HxH| - \bigl\langle \chi_{\C[H]e} \Res_{H \cap xHx^{-1}}, (\chi^{x^{\scriptscriptstyle -1}}_{\C[H](1-e)})\Res_{H\cap xHx^{-1}}\bigr\rangle
\]
\end{proposition}
\begin{proof}
    We apply Proposition \ref{prop:MackeyOnHxHf} and Lemma \ref{lem:IdempotentsVSCharacters} to the left
    $\C[H]$-module $\C[HxH](1-e)$ to obtain
    \begin{align*}
        \dim e\C[HxH](1-e) &= \bigl\langle \chi_{\C[H]e}, (\chi^{x^{\scriptscriptstyle -1}}_{\C[H](1-e)})\Res_{H\cap xHx^{-1}}\Ind^H\bigr\rangle \\
        &= \bigl\langle \chi_{\C[H]e} \Res_{H \cap xHx^{-1}}, (\chi^{x^{\scriptscriptstyle -1}}_{\C[H](1-e)})\Res_{H\cap xHx^{-1}}\bigr\rangle 
  \end{align*}
    where the second line follows from Frobenius reciprocity. 
    Now apply Lemma~\ref{lemma:RId}.
\end{proof}

We remark that it is useful to interpret the equation in the proposition as
giving, in the quantity subtracted, the number of linear equations that
define the algebra $\{w : ew(1-e) = 0 \}$ on the double coset $HxH$; equivalently
this is the codimension of the direct summand of the algebra in $\C[HxH]$. 
We now explore consequences of this proposition.
A notable case is when no equations are required. 

\begin{corollary}\label{cor:noConstraintOnDoubleCoset}
There is no constraint on an element $w \in \C[G]$ satisfying $ew(1-e) = 0$ from the double coset $HxH$
if and only if either
\[ \bigl\langle \chi_{\C[H]e} \Res_{H \cap xHx^{-1}}, (\chi^{x^{-1}}_{\C[H](1-e)}) \Res_{H \cap xHx^{-1}}
\bigr\rangle  =0 \]
or the same condition holds swapping $e$ and $1-e$.
\end{corollary}
\begin{proof}
This follows 
from Proposition~\ref{prop:ThetaAlgebraDoubleCosetDimension} and the interpretation made immediately above.
\end{proof}

In practice it is more convenient to have a condition using just one character.

\begin{corollary}\label{cor:cutThetaDimensionDoubleCosetOneCharacter}
Let $x \in G$ and let $w \in \C[HxH]$. Let $e \in \C[H]$ be an idempotent and define $L = \C[G]e$, $U = \C[H]e$ and $U^\mathrm{c} = \C[H](1-e)$.
Then
\[ \codim \left(\RId_{\C[HxH]}(L) \right)
= \frac{|HxH|}{|H|}  \phi(1)- \bigl\langle \phi  \Res_{H \cap xHx^{-1}}, \phi^{x^{-1}} 
\Res_{H \cap xHx^{-1}} \bigr\rangle \]
where $\phi$ is either $\chi_{U}$ or $\chi_{U^\mathrm{c}}$.
\end{corollary}
\begin{proof}
Recall from Example~\ref{ex:basicRepresentations} that the character of the regular representation of 
a group~$K$ is $\rho_K$, defined by $\rho_K(1) = |K|$ and $\rho_K(k) = 0$ if $k \not= 1$.
Since $\C[H] = U \oplus U^\mathrm{c}$, we have $\chi_{U} + \chi_{U^\mathrm{c}} = \rho_H$.
Setting $\phi = \chi_{U}$, 
Proposition \ref{prop:ThetaAlgebraDoubleCosetDimension}
implies that
\begin{align*}
\codim & \left(\RId_{\C[HxH]}(L) \right) \\ &=
\bigl\langle \phi \Res_{H \cap xHx^{-1}}, (\rho_H - \phi)^{x^{-1}}\Res_{H \cap xHx^{-1}} \bigr\rangle 
\\ &= \bigl\langle \phi \Res_{H \cap xHx^{-1}}, \rho_H\Res_{H \cap xHx^{-1}} \bigr\rangle
- \bigl\langle\phi \Res_{H \cap xHx^{-1}}, \phi^{x^{-1}} \Res_{H \cap xHx^{-1}}\bigr\rangle \\
&= \frac{\phi(1)|H|}{|H \cap xHx^{-1}|}
- \bigl\langle\phi \Res_{H \cap xHx^{-1}}, \phi^{x^{-1}} \Res_{H \cap xHx^{-1}} \bigr\rangle.
\end{align*}
Since $|H|/|H \cap xHx^{-1}| = |HxH|/|H|$ by~\eqref{eq:doubleCosetCountTxH} this gives
the first case when $\phi = \chi_U$, and the second case when $\phi = \chi_{U^c}$ is proved by replacing $U$ with~$U^\mathrm{c}$ throughout.
\end{proof}
In particular, the condition in Corollary~\ref{cor:noConstraintOnDoubleCoset}
 holds whenever \emph{both} 
 $xHx^{-1} = H$, 
\emph{and} $\chi_{U}$ 
is invariant under conjugation by~$x$.
On the other hand, \emph{it never} holds when $HxH$ is a double coset of maximum possible
size. 
Indeed in this case, since $|HxH| = |H|^2$, we have
\begin{align*}
    \dim \left(\RId_{\C[HxH]}(L) \right) &= \dim ~\{w\in\C[HxH]: ew(1-e) = 0\}\\
    &= |H|^2 - \dim U^\mathrm{c} \cdot \dim U.
\end{align*}

\begin{example}[Exact lumping] 
    By Proposition~\ref{prop:exactLwIsCGeta}(iv), the 
   left-invariant random walk driven by $w$ lumps exactly if and only if 
    $w\in\Theta(\eta_H)$.
      By counting the number of linear constraints imposed by Corollary \ref{cor:strongExact}
  we find that the codimension of $\Theta(\eta_H) \cap \C[HxH]$ in $\C[G]$ is 
    $|HxH|/|H|-1$. Note this is one less than the number of left cosets forming $HxH$.
    It is instructive to reprove this result using
     Corollary \ref{cor:cutThetaDimensionDoubleCosetOneCharacter}: setting $\phi = \triv_H$,  the codimension of $\Theta(\eta_H) \cap \C[HxH]$ is
    \[ \frac{|HxH|}{|H|} \triv_H(\id_H) - 
    \langle \triv_{H \cap xHx^{-1}}, \triv_{H \cap xHx^{-1}} \rangle = \frac{|HxH|}{|H|} - 1 \]
as just seen.
\end{example}

We leave it to the interested reader to formulate the analogous example for strong lumping;
the codimension in each double coset is the same, as should be expected from the $\star$-duality 
in Theorem~\ref{thm:mainTimeReversal}. Indeed, $\Theta(\triv_H) = \Theta(\eta_H)^\star$ and the $\star$ operation respects the direct sum decomposition of these algebras over double cosets in~\eqref{eq:ThetaSplit}, 
exchanging the summands for $\C[HxH]$ and $\C[Hx^{-1}H]$.


\subsection{Double coset decomposition of weak lumping algebras}
Let $e\in\Eb{H}$ be an idempotent, let $L = \C[G]e$. The double coset decompositions of $\RId_{\C[G]}(L)$ and $\RId_{\C[G]}(L^\circ)$ give in turn a decomposition of the weak lumping algebra of $e$,
\begin{equation} \Theta(e) = \bigoplus_{x \in H \backslash G / H} \Theta(e) \cap \C[HxH]\label{eq:ThetaSplit}
\end{equation} 
and so weak lumping of irreducible weights is decided double coset
by double coset.

\begin{proposition}\label{prop:ThetaDoubleCosetCondition}
Let $w \in \C[HxH]$. Let $e\in\Eb{H}$ be an idempotent, and set $L = \C[G]e$.
A necessary and sufficient condition for $w \in \Theta(e)\cap\C[HxH]$ is
that $w \in \RId_{\C[HxH]}(L^\circ)$ and there exists $v \in L^\circ$ such that $w(Hxh) + v(Hxh)$ is constant
for $h \in H$. Moreover if such a $v$ exists, then $\eta_H v$ has the same property.
\end{proposition}

More informally we may say `$w$ is constant on right cosets modulo $L^\circ$'.

\begin{proof}
Since $\Theta(e)\cap\C[HxH] = \RId_{\C[HxH]}(L^\circ) \hskip1pt\cap\hskip1pt \RId_{\C[HxH]}(L)$ we may
assume throughout that $w \in \RId_{\C[HxH]}(L^\circ)$. Since $L = L^\circ \oplus \C[G]\eta_H$ we have,
under this assumption,
\begin{align*}
w \in&~ \Theta(e)\cap\C[HxH] \\ &\iff \C[G]\eta_H w \subseteq (L^\circ \oplus \C[G]\eta_H) \cap \C[HxH] \\
&\iff \eta_H w \in (L^\circ \cap \C[HxH]) \oplus \langle h x \eta_H : h \in H \rangle \\
&\iff \eta_H w \in \eta_H (L^\circ \cap \C[HxH]) \oplus \eta_H \langle h x \eta_H : h \in H \rangle \\
&\iff \eta_H w = \eta_H v + c \eta_H x \eta_H \ \ \text{for some $v \in L^\circ \cap \C[HxH]$ and $c \in \C$,} 
\end{align*}
where the third double implication follows by multiplying on the left by the idempotent~$\eta_H$.
Now observe that by Lemma~\ref{lemma:averaging}(i) and (iii), the condition in the final line holds if and only if
there exist $v \in L^\circ \cap \C[HxH]$ and $c \in \C$ such that
\[ \frac{w(Hxh)}{|H|} = \frac{v(Hxh) }{|H|} + \frac{c}{|HxH|} \]
for all $h \in H$. Since $L^\circ = \bigoplus_{x \in H \backslash G / H} (L^\circ \cap \C[HxH])$,
it is equivalent to require $v \in L^\circ$; this proves the condition is necessary and
sufficient. Finally, by Lemma~\ref{lemma:averaging}(i), 
we have $(\eta_H v)(Hg) = v(Hg)$ for each $g$, so we may replace $v$ with $\eta_H v$. 
\end{proof}

\begin{example}[Strong lumping]\label{ex:ThetaRightCosetConditionStrong}
Let $e = 1$ and $L = \C[G]$. Thus
 $L^\circ = \C[G](1-\eta_H) = \Ann_{\C[G]}(\eta_H)$
is the space of all $w \in \C[G]$ whose sum is zero on each left coset $gH$. By Lemma~\ref{lemma:averaging}(ii),
$\eta_H \C[G] (1-\eta_H)$ is the space of all elements of $\C[G]$ constant on
each right coset $Hg$ having zero sum on each left coset $gH$. 

For instance, the left diagram overleaf shows the elements in a double coset $HxH$ for $H = \langle h \rangle \cong C_6$
with $xHx^{-1} \cap H = \langle h^3 \rangle$;
thus each $hxH \cap Hxh'$ has the form $\{g, h^3 g\}$ where $h^3 g = gh^3$.
On the right we show 
a general $w \in \eta_H \C[G] (1-\eta_H)$, where the condition ``$w(g) = a$ for all $g\in xH\cap Hx$'' is represented by placing an $a$ in the cell in column $xH$ and row $xH$ of the table, and similarly for all other cells. Note that since $w$ is constant on right cosets,
it is in particular constant on each $hxH \cap Hxh'$.

\begin{center}
    \includegraphics[page=21]{AllPictures.pdf}
\end{center}

Therefore given any $\eta_H w$, necessarily constant on right cosets,
we may add an element of $\eta_H L^\circ$, of the form shown above, to obtain an element of $\C[G]$ 
constant on $HgH$. This verifies the necessary 
and sufficient condition in Proposition~\ref{prop:ThetaDoubleCosetCondition},
and shows that the condition is equivalent to $w \in \Theta(1)$.
This is of course expected as $L = \C[G]$ and so the requirement
$Lw \subseteq L$ imposes no constraints.
\end{example}

For a further example of Proposition~\ref{prop:ThetaDoubleCosetCondition} see
Remark~\ref{remark:ThetaDoubleCosetCondition}.

\section{Abelian subgroups}
\label{sec:abelian}

In this section we consider the case where $H$ is abelian. In this case Theorem~\ref{thm:mainGL} can be made very explicit.
The irreducible representations of an abelian group $H$ are all $1$-dimensional and may be
identified with its irreducible characters, forming the group $\widehat{H}$.
The irreducible representation corresponding to the character $\beta$ is afforded by 
the idempotent element
\begin{equation}\label{eq:abelianIdempotent} e_\beta = \frac{1}{|H|} \sum_{h \in H} \beta(h^{-1}) h. \end{equation}
Note that the inverse is necessary so that we have 
\begin{equation} \label{eq:idempotentAffordsRepresentation}
h e_\beta  = \beta(h) e_\beta \end{equation} 
for each $h \in H$.
The following lemma records some basic properties we need.
Note that in (ii), the perpendicular space is taken with respect to the canonical $H$-invariant inner product on $\C[H]$,
defined by taking $H = G$ in~\eqref{eq:innerProduct}.

\begin{lemma}\label{lemma:abelianIdempotent}
Let $H$ be an abelian group, and let $\beta$, $\gamma \in \widehat{H}$. 
\begin{thmlist}
\item If $\beta, \gamma \in \widehat{H}$ then $e_\beta e_\gamma = e_\gamma e_\beta = 0$.
\item For each $\beta \in \widehat{H}$ we have $\langle e_\beta \rangle^\perp = \langle e_\gamma : \gamma \not=\beta
\rangle$.
\item We have \smash{$1 = \sum_{\beta \in \widehat{H}} e_\beta$}.
\item Every primitive idempotent of $H$ is of the form $e_\beta$ for some $\beta \in \widehat{H}$.
\item We have $x^{-1} e_\beta x = e_{\beta^x}$.
\end{thmlist}
\end{lemma}

\begin{proof}
Parts (i) and (ii) follow from orthogonality of characters; parts (iii) and (iv) can be seen from the Wedderburn decomposition noting that every block is one-dimensional.
Part (v) is most simply proved by calculation:
\[ \begin{split} e_\beta x = \frac{1}{|H|}&{} \sum_{h \in H} \beta(h^{-1}) h x =
\frac{1}{|H|} \sum_{h \in H} \beta(h^{-1}) x (x^{-1}hx) \\ &=
\frac{1}{|H|} \sum_{k \in H} \beta(xk^{-1}x^{-1}) xk =
\frac{1}{|H|} \sum_{k \in H} \beta^k(x) xk =
x e_{\beta^x} \end{split} \]
where the penultimate equality holds since $\beta(xk^{-1}x^{-1}) = \beta^x(k^{-1})$.
\end{proof}

Given a subset $\IP \subseteq \widehat{H}$ of irreducible characters of $H$, let
\[ e_\IP = \sum_{\beta \in I} e_\beta. \]
By Lemma~\ref{lemma:abelianIdempotent}(iv), every idempotent in $\C[H]$ is of the form $e_\IP$, and so 
the left $\C[G]$-modules $K$ we must consider are precisely 
those $\C[G] e_\IP$ as $I$ ranges over $\widehat{H} \backslash \{\triv_H\}$.
Note that there are only finitely many such modules; in fact this feature characterises the case of abelian $H$.

\subsection{\texorpdfstring{Double coset decomposition of weak lumping algebras: 
the case of abelian $H$}{}}
Given a subset $\XY$ of $\C[HxH]$, let $\XY^\perp$ denote the perpendicular space to $\XY$ \emph{inside $\C[HxH]$},
with  respect to the canonical $G$-invariant
inner product on $\C[G]$ defined in~\eqref{eq:innerProduct}.

\begin{proposition}\label{prop:abelianThetaPerp}
Let $H$ be an abelian subgroup of $G$.
For $\IP \subseteq \widehat{H}$ we have
\[ 
\RId_{\C[HxH]}(\C[G]e_\IP)^\perp =
 \bigl\langle e_\beta  x  e_\gamma : \beta \in \IP, \gamma \in \widehat{H} \backslash \IP \bigr\rangle.
\]
\end{proposition}

\begin{proof}
By Lemmas \ref{lemma:RId} and \ref{lemma:ThetaDoubleCosetDecomposition} 
we can write
\[ \RId_{\C[HxH]}(\C[G]e_\IP) = \C[HxH] e_\IP + (1-e_\IP)\C[HxH]. \]
The rest of this proof follows from parts (i)--(v) of Lemma \ref{lemma:abelianIdempotent}.
By part (i), the image of right
multiplication by $e_\IP$ is 
\[ \langle h x k e_\beta : \beta \in \IP, h, k \in H \rangle.\]
By~\eqref{eq:idempotentAffordsRepresentation} we may simplify this to
$\langle h x e_\beta : \beta \in \IP, h \in H \rangle$.
Now by part (ii) the perpendicular space of the image is $\langle h x e_\gamma : \gamma \not\in \IP, h \in H \rangle$.
Similarly, 
the image of left multiplication by $1-e_\IP$ is
$\langle e_\gamma x k : \gamma \in \widehat{H} \backslash \IP, k \in H \rangle$ and its perpendicular
space is $\langle e_\beta x k : \beta \in \IP, k \in H \rangle$. Taking the intersection
of the perpendicular spaces gives the result.
\end{proof}

We now obtain the analogue of Proposition~\ref{prop:ThetaDoubleCosetCondition}
for the case of abelian~$H$. 
Recall that $\Theta(e)$, as defined in~\eqref{eq:Theta}, is 
the algebra of weakly lumping weights for the idempotent $e \in \Eb{H}$ and
by~\eqref{eq:ThetaSplit} we have $\Theta(e) = \bigoplus_{x \in H\backslash G / H}
\Theta(e) \cap \C[HxH]$, and so, as we noted after this equation,
weak lumping of weights is decided double coset by double coset. 
As 
further motivation
for the hypothesis below, note that $e_\IP \in \Eb{H}$ if and only if
$\triv_H \in \IP$.

\begin{corollary}\label{cor:ThetaDoubleCosetConditionAbelian}
Let $H$ be an abelian subgroup of $G$.
Let $\IP \subseteq \widehat{H}$ contain $\triv_H$
and let $w \in \C[HxH]$. A necessary and sufficient
condition for $w \in \Theta(e_\IP)$  is that
\[ 
w \in\bigl( \bigl\langle e_\beta x e_\gamma : \beta \in \IP, \gamma \in \widehat{H}
\backslash \IP \bigr\rangle
+ \bigl\langle e_\beta x \eta_H : \beta \in \IP \backslash \{\triv_H\} 
\bigr\rangle \bigr)^\perp.
\]
\end{corollary}
\begin{proof}
By~\eqref{eq: Theta as intersection of idealizers}
we have $\Theta(e) = \RId_{\C[G]}(\C[G]e_\IP) \cap \RId_{\C[G]}(e_\IP - \eta_H)$.
By Proposition~\ref{prop:abelianThetaPerp} we have
\begin{align*}
\RId_{\C[HxH]}(\C[G]e_\IP)^\perp &= \langle e_\beta x e_\gamma : \beta \in \IP,
\gamma \in \widehat{H} \backslash \IP \rangle \\
\RId_{\C[HxH]}(\C[G](e_\IP - \eta_H))^\perp &= \langle e_\beta x e_\gamma :
\beta \in I \backslash \{\triv_H \}, \gamma \in (\widehat{H} \backslash I)
\cup \{ \triv_H \}. \end{align*}
Therefore $w \in \Theta(e) \cap \C[HxH]$ if and only if 
it is in the first perpendicular space,
and also perpendicular to all $e_\beta x \eta_H$ for $\beta \in I \backslash \{\triv_H\}$,
as required.
\end{proof}

Applying the corollary to each double coset in turn
we obtain finitely many linear equations that specify
a necessary and sufficient condition
for a weight $w$ to lie in the algebra $\Theta(e)$.
In the extreme case when there is a coset $|HxH| = |H|^2$ of maximum size,
the elements $e_\beta x e_\gamma$ are linearly independent and there are 
$|\IP|(|H|-|\IP|)$ linearly independent equations from the coset $HxH$.
In general, since
each element of $HxH$ has $|H \cap xHx^{-1}|$
different expressions in the form $hxh'$ for $h, h' \in H$,
some work is needed to get an irredundant set of equations.

We are now ready to prove our final main result, Corollary~\ref{cor:abelian}, which we restate below for convenience.

\setcounter{section}{1}
\setcounter{theorem}{11} 
\begin{corollary}
Let $D$ be a set of double coset representatives for $H\backslash G / H$.
The left-invariant random walk on $G$ driven by an irreducible weight $w$ lumps weakly on the left cosets of $H$ if 
and only if 
there exists a subset $P \subseteq \widehat{H}$ containing  $\triv_H$ such that
for all $x \in D$ we have $w \in \bigcap_{x \in X} W_x^\perp$, where 
\[ W_x=  \bigl\langle e_\beta  x  e_\gamma : 
\: \beta \in P,\: \gamma \in (\widehat{H} \backslash P) \cup \{\triv_H\} ,\: 
(\beta, \gamma) \not= (\triv_H, \triv_H) 
\bigr\rangle. \]
\end{corollary}

\begin{proof}
This follows by applying Corollary~\ref{cor:ThetaDoubleCosetConditionAbelian}
to each double coset in turn using~\eqref{eq:ThetaSplit}. Note
that $W_x^\perp$ is the perpendicular space in this corollary.
\end{proof}

\setcounter{section}{13}
\setcounter{theorem}{3}

\subsubsection{Cosets of the normalizer}
In one important case this difficulty
does not arise.
Recall from Definition~\ref{defn:characterTwist} that if $x \in G$ and
$\beta \in \widehat{H}$ then $\beta^x$ denotes the character of $x^{-1}Hx$
defined by $\beta^x(k) = \beta(xkx^{-1})$.
Given $\IP \subseteq \widehat{H}$ we write $\IP^x$ for~$\{\beta^x : \beta \in \IP\}$.
As motivation, we remark that the condition $xH = Hx$ holds if and only
$xH = Hx = HxH$ and
if $xHx^{-1} = H$, and so if and only if $x$ is in the normalizer $N_G(H)$;
in this case $HxH = xH = Hx$.

\begin{corollary}\label{cor:abelianSingletonCoset}
Let $H$ be an abelian subgroup of $G$.
Let $\IP \subseteq \widehat{H}$ contain $\triv_H$
and let $w \in \C[HxH]$ be a weight. 
Suppose that $xH = Hx$.
Then a necessary and sufficient
condition for $w \in \Theta(e_\IP)$  is that
\[ w \in \bigl\langle x e_\delta : \delta \in \IP^x \cap  
(\widehat{H}\setminus \IP) \bigr\rangle^\perp. \]
\end{corollary}

\begin{proof}
By Lemma~\ref{lemma:abelianIdempotent}(v)
we have $e_\beta x = x x^{-1} e_\beta x = xe_{\beta^x}$.
By this observation and Lemma~\ref{lemma:abelianIdempotent}(i)
we have 
\[ e_\beta x e_\gamma = \begin{cases}
e_{\beta^x} & \text{if $\beta^x = \gamma$} \\
0 & \text{otherwise.} \end{cases}\]
Now apply this to the sum of the two perpendicular spaces
given in Corollary~\ref{cor:ThetaDoubleCosetConditionAbelian},
noting that $e_{\beta^x} \eta_H = 0$ by Lemma~\ref{lemma:abelianIdempotent}(i),
since $\beta^x \not=\triv_H$.
\end{proof}

In particular, if $x \in H$, or more generally, if $Hx = xH$ \emph{and}
the conjugacy action $x$ permutes the irreducible
characters in $\IP$, then there is no constraint from the double coset $HxH$ on
the weights in $\Theta(e) \cap \C[HxH]$.

\subsection{An extended example: the six-sided die}
\label{subsec: abelian example}

Consider an ordinary six-sided die, that is rolled and then translated to its original position. 
For instance, imagine it is in an automatic dice roller which only allows for one stable position of the die.
The orientation-preserving symmetries of a cube are realised by the symmetric group $G = \Sym_4$,
 acting by permuting the four diagonals of the cube.
An observation of the top value of a die 
is invariant under the group of symmetries of the top face, 
which is the cyclic group 
$H = \langle(1,2,3,4)\rangle = \langle h\rangle \cong C_4$, where $h = (1,2,3,4)$.
Therefore repeated observations of the top face correspond to lumping of the
left-invariant random walk on $G$ to the left cosets $G/H$.

\subsubsection{Face action}
It is useful to understand the action on the faces. 
Recall that opposite faces sum to $7$. 
If the top face is \epsface{1}, we set the convention that the permutation $(1,2,3,4)$ acts on the faces as the permutation $(\epsface{2},\epsface{3},\epsface{5},\epsface{4})$ sending the face \epsface{2} to \epsface{3}, the face \epsface{3} to \epsface{5}, and so on.
We let the permutation $(1,2)$ act by $(\epsface{1},\epsface{2})(\epsface{3},\epsface{4})(\epsface{5},\epsface{6})$.
This permutation swaps two diagonals of the cube, and corresponds to the $180^\circ$ rotation about the axis through the centre of the edge between faces $\epsface{1}$ and $\epsface{2}$ and the centre of the edge between faces $\epsface{5}$ and $\epsface{6}$.

\subsubsection{Idempotents}
The subgroup $H$ is abelian. Its four primitive idempotents are
\[
    \def\arraycolsep{1pt}
    \begin{array}{rcrcrcrcrcr}
        e_\mathbbm{1} = \eta_H = \eta &=& \frac{1}{4}( & 1 &+ & h &+ & h^2&+ & h^3),\\[3pt]
        \xi &=& \frac{1}{4}( & 1 &+ & i h &- & h^2&- & ih^3),\\[3pt]
        e_{\sgn} = \varsigma &=& \frac{1}{4}( & 1 &- & h &+ &  h^2&- & h^3),\\[3pt]
        \bar\xi &=& \frac{1}{4}( & 1 &- & i h &-&  h^2&+ & h^3).\\
    \end{array}
\]
Hence $\Eb{H}$ has 8 elements, obtained by taking $\eta$ plus one of the possible combinations of the other idempotents. 
By~\eqref{eq:Theta}, $w$ is a weak lumping weight if and only if $w \in \bigcup_{e\in\Eb{H}}\Theta(e)$; by Corollary~\ref{cor:ThetaDoubleCosetConditionAbelian}
this is decided double coset by double coset by finitely many linear equations.
Among the~8 idempotents, only the $4$ which have either both
or neither of $\xi$ and $\bar\xi$ as a summand are real idempotents,
as studied in~\S\ref{sec:real}; we compute $\Theta(e)$ for each of these.

\subsubsection{Double coset decomposition}
There are three double cosets in $H\backslash G / H$, namely
$H$, $H(1,3)H$ and $H(1,2)H$,
where $H(1,3)H = H(1,3) = (1,3)H$ is of the type in Corollary~\ref{cor:abelianSingletonCoset}.
The double coset $H(1,2)H$ of size $16$ controls
which weights lump weakly to $G/H$. Its structure is shown
in Figure~\ref{fig:DoubleCosetsPopOMatic}.

\begin{figure}[h!]
\begin{center}
\includegraphics[page=22]{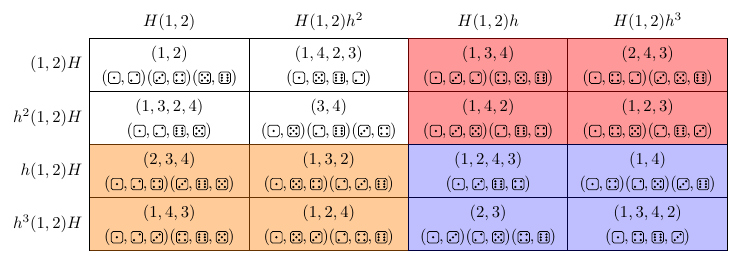}
\end{center}\smallskip
\begin{center}
\includegraphics[page=23]{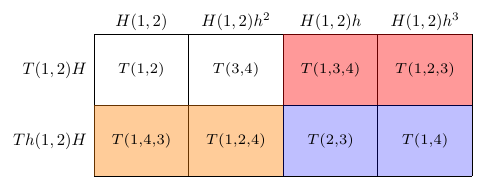}
\end{center}
\caption{The top diagram shows the double coset $H(1,2)H$ when $G = \Sym_4$ and $H = \langle(1,2,3,4)\rangle$, and its action on the six faces of a die.
Rows are left cosets and columns are right cosets. 
For instance, the permutation $(1,2)$ 
swaps two diagonals of the cube, and corresponds to the $180^\circ$ rotation about the axis through the centre of the edge between faces $\epsface{1}$ and $\epsface{2}$ and the centre of the edge between faces $\epsface{5}$ and $\epsface{6}$, resulting in a permutation $(\epsface{1},\epsface{2})(\epsface{3},\epsface{4})(\epsface{5},\epsface{6})$ of the faces.
Shaded regions indicate the double cosets $TxT$ where $T = \langle(1,2)(3,4)\rangle$.
The division into right cosets of $T$ relevant to Example~\ref{ex:popWeakDual}
is shown in the lower diagram.
For instance $T(1,2)T
= T(1,2) \cup T(3,4)$ contains the four permutations
in the white region.
}
\label{fig:DoubleCosetsPopOMatic}
\end{figure}

 We now use Corollary~\ref{cor:ThetaDoubleCosetConditionAbelian}
to compute the space perpendicular to 
$\Theta(e)\cap\C[HxH]$ for every $e \in \Eb{H}$ and every double coset $HxH$.

\subsubsection{Strong lumping}
Fix the real idempotent $1 \in \Eb{H}$. The algebra $\Theta(1)$ is the 
strong lumping algebra (see~Proposition \ref{prop:strongJwIsCG}). 
That is, $w\in\Theta(1)$ if and only if the left-invariant random walk driven by $w$ lumps weakly for all initial distributions.
By Corollary~\ref{cor:ThetaDoubleCosetConditionAbelian}, 
this holds for $w \in \C[HxH]$ if and only if
\[
x \in \langle \xi x \eta, \varsigma x \eta, \bar{\xi} x \eta \rangle^\perp.
\]
(Note the first summand in this corollary vanishes.)
If $x = \Id_H$ or $x = (1,3)$ then since $Hx = xH$ in these cases,
by Corollary~\ref{cor:abelianSingletonCoset},
there is no constraint from this double coset.
Hence 
\[
\Theta(1)^\perp = 
\bigl\langle \xi (1,2)\eta, \varsigma (1,2)\eta, \bar\xi (1,2)\eta \bigr\rangle
\]
and there are just $3$ linear constraints.

\subsubsection{Exact lumping}
Fix the real idempotent $\eta \in \Eb{H}$. The algebra
$\Theta(\eta)$ is the exact lumping algebra $\Theta(\eta)$ seen in Proposition \ref{prop:exactLwIsCGeta}. (See Definition~\ref{defn:exactLumping} for the 
definition of exact lumping.)
By Corollary~\ref{cor:ThetaDoubleCosetConditionAbelian}, 
the weight $w \in \C[HxH]$  lumps exactly if and only if
\[
x \in \langle \eta x \xi, \eta x \varsigma, \eta x \bar{\xi} \rangle^\perp.
\]
(Note the second summand in this corollary vanishes.)
Again if $x = \Id_H$ or $x = (1,3)$ then since $Hx = xH$ in these cases,
by Corollary~\ref{cor:abelianSingletonCoset},
there is no constraint from this double coset.
Hence 
\[
\Theta(\eta_H)^\perp = 
\bigl\langle \eta (1,2) \xi, \eta (1,2) \varsigma, \eta (1,2) \bar{\xi} \rangle^\perp
\]
and again there are $3$ linear constraints. This should be expected
from Theorem~\ref{thm:mainTimeReversal}, which, informally stated,
says that the exact lumping algebra $\Theta(\eta_H)$ is dual to the strong
lumping algebra $\Theta(1)$.

\subsubsection{The first weak lumping algebra}\label{subsec:weakLumpingAlgebra}
Fix $\IP = \{ \eta, \xi, \bar\xi\}\subseteq \widehat{H}$ 
and the real idempotent $e_\IP = \eta+
\xi+\bar\xi \in \Eb{H}$.
Similar arguments to the strong and exact cases using 
Corollaries~\ref{cor:ThetaDoubleCosetConditionAbelian} 
and~\ref{cor:abelianSingletonCoset} show that
\begin{equation}\label{eq:popWeak}
\Theta(e_\IP)^\perp = \bigl\langle \eta (1,2)\varsigma, \xi (1,2) \varsigma,
 \bar\xi (1,2) \varsigma,   \xi (1,2) \eta, \bar\xi (1,2)\eta
 \bigr\rangle.
\end{equation}
In particular, there is no constraint from the double cosets $H$ and $H(1,3)H$ and 
the dimension of $\Theta(e_\IP)$ is $24 - 5 = 19$. It is
instructive to verify this using the dimension formula of Corollary~\ref{cor:dimTheta}. 
The four characters of $C_4$ are $\mathbbm{1}, \gamma, {\sgn}, \bar\gamma$, corresponding to the four idempotents $\eta, \xi, \varsigma, \bar\xi$, and 
\begin{align*}
    \mathbbm{1} \Ind_H^G &= \chi^{(4)} + \chi^{(2,2)} + \chi^{(2,1,1)}\\ 
    \bar\gamma \Ind_H^G = \gamma \Ind_H^G &= \chi^{(3,1)} + \chi^{(2,1,1)}\\ 
    \sgn \Ind_H^G &= \chi^{(3,1)} + \chi^{(2,2)} + \chi^{(1,1,1,1)},
\end{align*}
and thus the relevant coefficients are as shown in the table below.
\[
\begin{array}{r|ccccc}
    \bottomrule \\[-6pt]
    \psi  & \chi^{(4)} & \chi^{(3,1)} & \chi^{(2,2)} & \chi^{(2,1,1)} & \chi^{(1,1,1,1)}\\ \midrule
    a_\psi = \langle (\gamma + \bar\gamma)\ind_H^G, \psi\rangle & 0 & 2 & 0 & 2 & 0\\ 
    c_\psi = \langle \mathbbm{1}\ind_H^G, \psi\rangle & 1 & 0 & 1 & 1 & 0\\
    d_\psi = \psi(1) & 1 & 3 & 2 & 3 & 1\\ \bottomrule
\end{array}
\]
The dimension formula therefore gives
\begin{align*}
    \dim \Theta(e_\IP) &= \sum_{\psi\in\IrrC{G}} (
a_\psi^2 + a_\psi c_\psi + c_\psi^2 - a_\psi d_\psi - c_\psi d_\psi + d_\psi^2).\\
&= 1 + 7 + 3 + 7 + 1 = 19 
\end{align*}
as expected.
To make~\eqref{eq:popWeak} explicit,
we represent a weight
\[
w=  \sum_{i = 0}^{4} \sum_{j = 0}^{4} w\big(h^i(1,2)h^j\big) \cdot h^i(1,2)h^j
\]
supported on $H(1,2)H$
as a $4\times4$ matrix, in which $w\bigl( h^i(1,2)h^j \bigr)$
is in the row $h^i(1,2)H$ and column $H(1,2)h^j$, when $i$ and $j$
are ordered $0,2,1,3$ as in
Figure~\ref{fig:DoubleCosetsPopOMatic}. 
Using this notation, the algebra $\Theta(e_\IP)$ is then the space of all
 weights whose coefficients on $H(1,2)H$ 
are orthogonal to the five matrices
\newcommand{\pluso}{1}
\newcommand{\plusi}{i}
\[  
\scalebox{0.55}{$
\left(\begin{array}{CCCCs} \pluso  & \pluso& -1 &  -1 &\\
\pluso & \pluso &  -1 &-1 \\
\pluso &  \pluso & -1 &-1  \\
\pluso & \pluso & -1 & -1 
\end{array}\right), \ 
\left( \begin{array}{CCCCs} \pluso & \pluso &  -1 & -1 &\\
-1 & -1 & \pluso &  \pluso \\
\plusi & \plusi & -i & -i \\
-i &  -i &  \plusi & \plusi
\end{array} \right), \ 
\left(\begin{array}{CCCCs} \pluso  & \pluso& -1 & -1 &\\
-1  & -1  & \pluso & \pluso \\
-i  & -i & \plusi& \plusi \\
\plusi & \plusi & -i & -i
\end{array}\right), \ 
\left(\begin{array}{CCCCs} \pluso & \pluso & \pluso & \pluso &\\
-1 & -1 & -1 & -1 \\
\plusi & \plusi & \plusi & \plusi \\
-i & -i & -i & -i
\end{array}\right), \ 
\left(\begin{array}{CCCCs}  \pluso & \pluso & \pluso & \pluso &\\
-1 & -1 & -1 & -1 \\
-i & -i & -i & -i \\
\plusi & \plusi & \plusi & \plusi
\end{array}\right)$}
\]
or equivalently, by taking obvious linear combinations, to the five matrices
\[ 
\scalebox{0.55}{$
\left(\begin{array}{CCCCs} \pluso & \pluso & -1 & -1 &\\
\pluso & \pluso &  -1 &-1 \\
\pluso & \pluso & -1 & -1  \\
\pluso & \pluso &  -1 & -1 
\end{array}\right), \ 
\left( \begin{array}{CCCCs} 1 & 1 & 0 & 0&\\
-1 & -1 & 0 & 0 \\
0 & 0 & 0 & 0 \\
0 &0 & 0 & 0
\end{array} \right), \ 
\left(\begin{array}{CCCCs}  0 & 0 & 0 & 0 &\\
0 & 0  & 0 & 0 \\
0 & 0 & 1 & 1  \\
0 & 0 & -1 & -1
\end{array}\right), \ 
\left(\begin{array}{CCCCs} \pluso & \pluso & \pluso & \pluso &\\
-1 & -1 & -1 & -1 \\
0 & 0 & 0 & 0  \\
0 & 0 & 0 & 0
\end{array}\right), \ 
\left(\begin{array}{CCCCs}  0 & 0 & 0 & 0 &\\
0 & 0 & 0 & 0 \\
1 & 1 & 1 & 1 \\ 
-1 & -1 & -1 & -1 
\end{array}\right)$}.
\]

We use this description of $\Theta(e_\IP)$  to illustrate 
Theorem~\ref{thm:mainTransitionMatrices}.
The orbital matrices (see \S\ref{subsec:orbitalMatrices})
for the three double cosets $H$, $H(1,3)$ and $H(1,2)H$ are 
\[ \left( \begin{matrix} 1 & \cdot & \cdot & \cdot & \cdot & \cdot \\ 
 \cdot & 1 & \cdot & \cdot & \cdot & \cdot \\
 \cdot & \cdot & 1 & \cdot & \cdot & \cdot \\
 \cdot & \cdot & \cdot & 1 & \cdot & \cdot \\
 \cdot & \cdot & \cdot & \cdot & 1 & \cdot \\
 \cdot & \cdot & \cdot & \cdot & \cdot & 1 \end{matrix} \right),
 \quad
 \left( \begin{matrix}\cdot & 1 & \cdot & \cdot & \cdot & \cdot \\ 
 1 & \cdot & \cdot & \cdot & \cdot & \cdot \\
 \cdot & \cdot & \cdot & 1 & \cdot & \cdot \\
 \cdot & \cdot & 1     & \cdot & \cdot & \cdot \\
 \cdot & \cdot & \cdot & \cdot & \cdot & 1 \\
  \cdot & \cdot & \cdot & \cdot & 1 & \cdot 
 \end{matrix} \right),\quad
  \left( \begin{matrix}\cdot & \cdot & 1 & 1 & 1 & 1 \\ 
\cdot & \cdot & 1 & 1 & 1 & 1 \\ 
 1 & 1 & \cdot & \cdot & 1 & 1 \\
 1 & 1 & \cdot & \cdot & 1 & 1 \\
 1 & 1 & 1 & 1 & \cdot & \cdot \\
 1 & 1 & 1 & 1 & \cdot & \cdot \\
 \end{matrix} \right)
\]
where the left cosets of $H$ appear in the order
 $H$, $(1,3)H$, $(1,2)H$,
$h^2(1,2)H$, $h(1,2)H$, $h^3(1,2)H$, and for readability $\cdot$ denotes a $0$ entry.
For example, the weight  defined by the third matrix is 
simply $\eta (1,2) \eta$, giving equal weight to all elements in the double coset $H(1,2)H$,
and so corresponding to the all-ones $4 \times 4$ matrix in the notation above. 
Clearly it is orthogonal to all five matrices.
A similar argument for the other two cosets shows that, as expected, 
weights in the Hecke algebra
$\eta \C[\Sym_4] \eta$ satisfy~\eqref{eq:popWeak} and so lump stably, in the sense of Definition~\ref{defn:lumpsStably} 
for the ideal $\C[G]e_\IP$.
Of course such weights also lump strongly and exactly,
by Theorem~\ref{thm:mainTransitionMatrices}.

\subsubsection{A weak lumping weight that does not lump strongly or exactly, from
the first weak lumping algebra}
From the final $5$ matrices above, we see that
$w \in \Theta(e_\IP)$ if and only if $w$ satisfies the five constraints
\begin{align} w\bigl(H(1,2)\bigr) + w\bigl(H(1,2)h^2\bigr)  &=  w\bigl(H(1,2)h\bigr) + w\bigl(H(1,2)h^3\bigr) \nonumber \\
 w\bigl((1,2)H\bigr) &= w\bigl(h^2(1,2)H\bigr) \nonumber \\
 w\bigl(h(1,2)H\bigr) &= w\bigl(h^3(1,2)H\bigr) \nonumber \\ 
 w\bigl((1,2)\bigr) + w\bigl((1,2)h^2\bigr) &= w\bigl(h^2(1,2)\bigr) + w\bigl(h^2(1,2)h^2\bigr) \nonumber \\
 w\bigl(h(1,2)h\bigr) + w\bigl(h(1,2)h^3\bigr) &= w\bigl(h^3(1,2)h\bigr) + w\bigl(h^3(1,2)h^3\bigr) . \label{eq:popWeakEquations}
 \end{align}
It is notable that it is not obvious that these conditions even define a subalgebra of $\C[G]$. We now use this description to give, in
the same spirit as the example in \S 1.2.3, a weight
which lumps weakly to the left cosets $G/H$ but not strongly or exactly.

\begin{example}\label{ex:popWeak}
Let $w$ be the weight supported on the double coset $H(1,2)H$ as written
below in our usual convention so 
that the entry $h^i(1,2)h^j$ is in the position indicated by Figure~\ref{fig:DoubleCosetsPopOMatic},
using the order $0,2,1,3$.
\[
\frac{1}{12}
\scalebox{1}{$\left(\begin{matrix}
    2 & 1 & 1 & 2\\
    0 & 3 & 3 & 0\\ 
    0 & 0 & 0 & 0\\
    0 & 0 & 0 & 0 
\end{matrix}\right)$}.
\]
Explicitly,
\begin{align*} 
w &= \frac{1}{12}\bigl( 2 (1,2) + (1,2)h^2 + (1,2)h + 2(1,2)h^3 + 3 h^2(1,2)h^2 + 3h^2(1,2)h \bigr) \\
  &= \frac{1}{12}\bigl( 2 (1,2) + (1,4,2,3) + (1,3,4) + 2(2,4,3) + 3(3,4) + 3(1,4,2)\bigr).
\end{align*}
This matrix is orthogonal to either set of five matrices shown above, and so the weight $w$ lumps weakly.
By Corollary~\ref{cor:strongExact}, it does not lump strongly or exactly
because the matrix has neither
constant row sums, nor constant column sums.
\end{example}

\begin{remark}\label{remark:ThetaDoubleCosetCondition}
If $L = \C[G]e_\IP$ then $L^\circ = \C[G](e_\IP - \eta_H)$ and,
by Proposition~\ref{prop:abelianThetaPerp},
a necessary and sufficient condition for $w \in \RId(L^\circ)$
is that $w$ satisfies the final four equations in~\eqref{eq:popWeakEquations};
now Proposition~\ref{prop:ThetaDoubleCosetCondition} implies 
that a weight $w$ satisfying these equations is weakly lumping
if and only if there exists $v \in L^\circ$ such that
$w\bigl( H(1,2)h^j  \bigr) + v\bigl( H(1,2)h^j \bigr)$
is constant as $j$ varies. Since $v$ must be real and $e_\IP - \eta_H = \xi + \bar\xi = \mfrac{1}{2}(1 -h +h^2 - h^3)$ an
equivalent condition is that
\begin{equation}
 w\bigl( H(1,2)h^j  \bigr) + v\bigl( H(1,2)h^j \bigr) \label{eq:popWeakPlus}
\end{equation}
is constant as $j$ varies,
for some $v \in \langle h^i(1,2) (1 - h + h^2 - h^3) : 0 \le i \le 3 \rangle$.
Since any such~$v$ satisfies $v\bigl( H(1,2) \bigr) = v(H(1,2)h^2$ and $v\bigl( H(1,2)h \bigr)
= v\bigl( H(1,2)h^3 \bigr)$,
whenever the final four equations in~\eqref{eq:popWeakEquations} hold,
the first holds if and only if condition~\eqref{eq:popWeakPlus} holds.
This verifies the conclusion of Proposition~\ref{prop:ThetaDoubleCosetCondition}
for the first weak lumping algebra.
\end{remark}

\subsubsection{The second weak lumping algebra}
The final
 real idempotent $e_\IP\in\Eb{H}$ is
 defined by taking 
 $\IP = \{ \eta, \varsigma\}\subseteq \widehat{H}$. 
This corresponds to the time reversal of the previous example, so we just
give the result of Corollaries~\ref{cor:ThetaDoubleCosetConditionAbelian} 
and~\ref{cor:abelianSingletonCoset} that
\[
\Theta(e_\IP)^\perp = \bigl\langle \varsigma (1,2) \xi, \varsigma (1,2)\xi, \varsigma (1,2)\bar\xi, \eta (1,2) \xi, \eta (1,2)\bar\xi \bigr\rangle.
\]
We finish by giving an alternative description of $\Theta(e_\IP)^\perp$
  using Proposition~\ref{prop:interpolating}. 
 Let $T = \langle (1,3)(2,4)\rangle \cong C_2$. We have $T \le H \le G$, and $\eta_T = \eta + \varsigma = e_\IP$. Hence $\C[G]e_\IP$ is a weak lumping Gurvits--Ledoux ideal 
 (in the sense of Definition~\ref{defn:GLideal}) 
 for the weight $w$ if and only if

\smallskip
\begin{thmlistE}
    \item[(a$'$)] $w(Tg)$ is constant for 
    $Tg \subseteq TxT$ for all $TxT\in T\backslash G / T$, and
    \item[(b$'$)] $w(TgH)$ is constant for $TgH \subseteq HxH$ for all $HxH\in H\backslash G / H$.
\end{thmlistE}

\smallskip\noindent
Many choices of 
$g$ and $x$ render $Tg = TxT$ or $TgH = HxH$, and thus impose no constraint on $w$. 
We give below the remaining five equations which define $\Theta(e_\IP)$:

\smallskip
\begin{thmlistE}
    \item[(a$'$)]
        $w\bigl(T(1,2)\bigr) = w\bigl( T(3,4)\bigr)$,\\
    $w\bigl(T(1,3,4)\bigr) = w\bigl( T(1,2,3)\bigr)$,\\
    $w\bigl( T(1,2,4) \bigr) = w\bigl( T(1,4,3) \bigr)$, \\
    $w\bigl(T(1,4)\bigr) = w\bigl( T(2,3)\bigr)$;
    \item[(b$'$)] $w\bigl(T(1,2)H\bigr) = w\bigl( T(1,4)H\bigr)$.
\end{thmlistE}

\smallskip\noindent
These equations specify that the weights of each right coset
in the shaded regions of  Figure~\ref{fig:DoubleCosetsPopOMatic} are equal.
For instance  $T(1,2) = \{(1,2), (1,3,2,4)\}$
and $T(3,4) = \{(3,4), (1,4,3,2) \}$ together form the top-left region
in this figure.

\begin{example}\label{ex:popWeakDual}
In Example~\ref{ex:popWeak} we saw that the weight
\begin{align*} 
w  &= \frac{1}{12}\bigl( 2 (1,2) + (1,4,2,3) + (1,3,4) + 2(2,4,3) + 3(3,4) + 3(1,4,2)\bigr)
\\
\intertext{lumps weakly to $G/H$ with stable ideal $\C[G](\eta + \xi + \bar\xi)$.
Noting that $1 - (\eta + \xi + \bar\xi) +
\eta = \eta + \varsigma$, by Theorem~\ref{thm:mainTimeReversal}, the weight}
w^\star &= 
\frac{1}{12}\bigl( 2 (1,2) + (1,3,2,4) + (1,4,3) + 2(2,3,4) + 3(3,4) + 3(1,2,4)\bigr)
\end{align*}
lumps weakly to $G/H$ with the stable ideal $\C[G](\eta + \varsigma)$
relevant to the final idempotent $e_\IP$ in this subsection.
This may be checked directly using Figure~\ref{fig:DoubleCosetsPopOMatic},
noting that $w^\star$ is supported on the $8$ permutations in the left
half of the diagram, and the weights of the two right cosets of~$T$
in the top-left are equal, and similarly for the bottom-left.
We saw in Example~\ref{ex:popWeak} that~$w$ does not lump strongly or exactly 
and so, by the $\star$-duality in Theorem~\ref{thm:mainTimeReversal}, neither does~$w^\star$.
\end{example}

\section*{Acknowledgements}
Edward Crane and Mark Wildon gratefully acknowledge financial support from the 
Heilbronn Institute for Mathematical Research, Bristol, UK.
Álvaro Gutiérrez was funded by the University of Bristol Research Training Support Grant. The work of Erin Russell was supported by the Engineering and Physical Sciences Research Council [grant number EP/W52413X/1]. This work was carried out using Magma \cite{Magma} on the computational facilities of the Advanced Computing Research Centre, University of Bristol: \texttt{http://www.bristol.ac.uk/acrc/}.


\end{document}